\documentclass[12pt,reqno]{amsart}
\usepackage{enumerate, latexsym, amsmath, amsfonts, amssymb, amsthm, color,mathrsfs}
\def\pmod #1{\ ({\rm{mod}}\ #1)}
\def\Z{\mathbb Z}
\def\C{\mathbb C}

\def\N{\mathbb N}

\def\R{\mathbb R}
\def\Q{\mathbb Q}

\def\sE{\mathscr E}
\def\sF{\mathscr F}
\def\sP{\mathscr P}

\def\jacob #1#2{\left(\frac{#1}{#2}\right)}
\def\pmod #1{\ ({\rm{mod}}\ #1)}

\def\cA{\mathcal{A}}
\def\cB{\mathcal{B}}
\def\bi{{\mathbf{i}}}

\def\1{{\bf{1}}}

\theoremstyle{plain}
\newtheorem{theorem}{Theorem}

\newtheorem{lemma}{Lemma}
\newtheorem{corollary}{Corollary}

\theoremstyle{definition}
\newtheorem*{acknowledgment}{Acknowledgments}
\theoremstyle{remark}

 \vspace{4mm}

\begin{document}

\hbox{}
\medskip

\title
{$p$-adic analogues of hypergeometric identities}

\author{Guo-Shuai Mao}
\address {Department of Mathematics, Nanjing
University, Nanjing 210093, People's Republic of China}
\email{mg1421007@smail.nju.edu.cn}
\author{Hao Pan}
\address{School of Applied Mathematics, Nanjing University of Finance and Economics, Nanjing 210046, People's Republic of China}
\email{haopan79@zoho.com}

 \begin{abstract}
We prove many congruences modulo $p^2$ for the truncated hypergeometric series in a unified way. For example, for any odd prime $p$ and two $p$-integers $\alpha,\beta$, we have the $p$-adic Gauss identity 
$$
\sum_{k=0}^{p-1}\frac{(\alpha)_k(\beta)_k}{(1)_k^2}\equiv-\frac{\Gamma_p(1-\alpha-\beta)}{\Gamma_p(1-\alpha)\Gamma_p(1-\beta)}\pmod{p^2},
$$
provided that $p$ is less than the sum of the least non-negative residues of $\alpha$ and $\beta$ modulo $p$. Furthermore, the $p$-adic analogues of some common hypergeometric identities, including the balanced ${}_4F_3$ transformation and Whipple's ${}_7F_6$ transformation, are established. We also confirm a conjecture of Deines et al.
\end{abstract}

\thanks{}
\maketitle

\section{Introduction}
\setcounter{lemma}{0}
\setcounter{theorem}{0}
\setcounter{corollary}{0}
\setcounter{remark}{0}
\setcounter{equation}{0}
\setcounter{conjecture}{0}

Define the {\it hypergeometric series}
\begin{equation}\label{hypergeometricseries}
{}_{m+1}F_m\bigg[\begin{matrix}
\alpha_0&\alpha_1&\ldots&\alpha_m\\
&\beta_1&\ldots&\beta_m
\end{matrix}\bigg|\,z\bigg]:=\sum_{k=0}^{\infty}\frac{(\alpha_0)_k(\alpha_1)_k\cdots(\alpha_m)_k}{(\beta_1)_k\cdots(\beta_m)_k}\cdot\frac{z^k}{k!},
\end{equation}
where $\alpha_0,\ldots,\alpha_m,\beta_1,\ldots,\beta_m,z\in\C$ and
$$
(\alpha)_k=\begin{cases}\alpha(\alpha+1)\cdots(\alpha+k-1),&\text{if }k\geq 1,\\
1,&\text{if }k=0.\end{cases}
$$
It is easy to see that (\ref{hypergeometricseries}) absolutely converges whenever $|z|<1$, or $|z|=1$ and $\Re(\beta_1+\cdots+\beta_m)>\Re(\alpha_0+\cdots+\alpha_m)$. If (\ref{hypergeometricseries}) is convergent, it is also called a {\it hypergeometric function}.
The hypergeometric functions play a very important role in mathematics. A reason is that many common mathematical functions, such as $e^x$, $\log x$, $\arcsin x$, $\arctan x$, can be expressed in the form of a hypergeometric series. The explicit expressions of a few hypergeometric functions are known. For example, a well-known result due to Gauss (cf. \cite[Theorem 2.2.2]{AAR99}) says that
\begin{equation}\label{Gaussidentity}
{}_{2}F_1\bigg[\begin{matrix}
\alpha&\beta\\
&\gamma
\end{matrix}\bigg|\,1\bigg]=\frac{\Gamma(\gamma)\Gamma(\gamma-\alpha-\beta)}{\Gamma(\gamma-\alpha)\Gamma(\gamma-\beta)}
\end{equation}
provided that $\Re(\gamma)>\Re(\alpha+\beta)$, where $\Gamma(\cdot)$ is the gamma function. Furthermore, there exist some transformations between the hypergeometric functions. For example, we have the Pfaff transformation (cf. \cite[2.2.6]{AAR99})
\begin{equation}\label{Eulertransformation}
{}_{2}F_1\bigg[\begin{matrix}
\alpha&\beta\\
&\gamma
\end{matrix}\bigg|\,z\bigg]=(1-z)^{-\alpha}\cdot{}_{2}F_1\bigg[\begin{matrix}
\alpha&\gamma-\beta\\
&\gamma
\end{matrix}\bigg|\,\frac{z}{z-1}\bigg].
\end{equation}

For $n=0,1,2,\ldots$, define {\it the truncated hypergeometric function}
$$
{}_{m+1}F_m\bigg[\begin{matrix}
\alpha_0&\alpha_1&\ldots&\alpha_m\\
&\beta_1&\ldots&\beta_m
\end{matrix}\bigg|\,z\bigg]_n:=\sum_{k=0}^{n}\frac{(\alpha_0)_k(\alpha_1)_k\cdots(\alpha_m)_k}{(\beta_1)_k\cdots(\beta_m)_k}\cdot\frac{z^k}{k!}.
$$
That is, a truncated hypergeometric function is the summation of the first finitely many terms in the corresponding hypergeometric series. Unlike the original hypergeometric functions, few explicit formulas are known for the truncated hypergeometric functions. On the other hand, in the recent years, the arithmetic properties of the truncated hypergeometric functions were widely studied. For example, Ahlgren and Ono \cite{AhOn00} proved that for any odd prime $p$,
\begin{equation}\label{AhlgrenOno}
{}_{4}F_3\bigg[\begin{matrix}
\frac12&\frac12&\frac12&\frac12\\
&1&1&1
\end{matrix}\bigg|\,1\bigg]_{p-1}\equiv a(p)\pmod{p^2},
\end{equation}
where $a(p)$ is the $p$-th Fourier coefficient of $\eta(2z)^4\eta(4z)^2\in S_{4}(\Gamma_0(8))$ and $\eta(\cdot)$ is the Dedekind eta-function. Ahlgren and Ono's result solves a conjectured congruence concerning the Ap\'ery numbers.
Subsequently, Kilbourn \cite{Kilbourn06} showed that (\ref{AhlgrenOno}) is still valid if $p^2$ is replaced by $p^3$. Furthermore, it has been shown that the truncated hypergeometric functions are closely related to the Gaussian hypergeometric functions \cite{Ahlgren01,FOP04,Ono98}, which is a finite field analogue of the original hypergeometric functions.

In \cite{Mortenson03,Mortenson04}, Mortenson proved that for any prime $p\geq 5$,
\begin{align}\label{Mortensoncongruences}
{}_{2}F_1\bigg[\begin{matrix}\frac12&\frac12\\&1\end{matrix}\bigg|\,1\bigg]_{p-1}&\equiv\jacob{-1}p\pmod{p^2},\quad
{}_{2}F_1\bigg[\begin{matrix}\frac13&\frac23\\&1\end{matrix}\bigg|\,1\bigg]_{p-1}\equiv\jacob{-3}p\pmod{p^2},\notag\\
{}_{2}F_1\bigg[\begin{matrix}\frac14&\frac34\\&1\end{matrix}\bigg|\,1\bigg]_{p-1}&\equiv\jacob{-2}p\pmod{p^2},\quad
{}_{2}F_1\bigg[\begin{matrix}\frac16&\frac56\\&1\end{matrix}\bigg|\,1\bigg]_{p-1}\equiv\jacob{-1}p\pmod{p^2},
\end{align}
where $\jacob{\cdot}{\cdot}$ denotes the Legendre symbol. The above congruences confirms several conjectures of Rodriguez-Villegas \cite{RVillegas03}, which are motivated by the Calabi-Yau manifolds. A key ingredient of Mortenson's proofs is the Gross-Koblitz formula, which can transfer a $p$-adic gamma function to a Gauss sum.
Moreover, in \cite{Mortenson05}, Mortenson established a general frame to study the supercongruences concerning the truncated hypergeometric functions by using the Gross-Koblitz formula and the Gaussian hypergeometric functions.
Also, an elementary proof of those congruences in (\ref{Mortensoncongruences}) was given in \cite{SunZW13}. 

On the other hand, in \cite{SunZH14}, Sun found that Mortenson's congruences can be extended to a unified form. For any $\alpha\in\Q$ which is $p$-integral (i.e. the denominator of $\alpha$ is prime to $p$), let $\langle\alpha\rangle_p$
be the least non-negative residue of $\alpha$ modulo $p$, i.e., the integer lying in $\{0,1,\ldots,p-1\}$ such that $\langle\alpha\rangle_p\equiv\alpha\pmod{p}.$ Sun proved that for any $p$-integral $\alpha\in\Q$,
\begin{equation}\label{SunZHalpha1}
{}_2F_1\bigg[\begin{matrix}
\alpha&1-\alpha\\
&1
\end{matrix}\bigg|\,1\bigg]_{p-1}\equiv(-1)^{\langle-\alpha\rangle_p}\pmod{p^2}.
\end{equation}
It is not difficult to verify that Mortenson's congruences in (\ref{Mortensoncongruences}) are the special cases of (\ref{SunZHalpha1}) when $\alpha=1/2,1/3,1/4,1/6$.
In fact, Sun obtained the following general result:
\begin{equation}\label{SunZHalphaz}
{}_2F_1\bigg[\begin{matrix}
\alpha&1-\alpha\\
&1
\end{matrix}\bigg|\,z\bigg]_{p-1}\equiv(-1)^{\langle-\alpha\rangle_p}
{}_2F_1\bigg[\begin{matrix}
\alpha&1-\alpha\\
&1
\end{matrix}\bigg|\,1-z\bigg]_{p-1}\pmod{p^2}.
\end{equation}
As we shall see later, (\ref{SunZHalphaz}) can be viewed as a $p$-adic analogue of a special case of the linear transformation due to Pfaff \cite[Eq. (2.3.14)]{AAR99}
\begin{equation}\label{nbetagammaz1z}
{}_2F_1\bigg[\begin{matrix}
-n&\beta\\
&\gamma
\end{matrix}\bigg|\,z\bigg]=\frac{(\gamma-\beta)_n}{(\gamma)_n}\cdot{}_2F_1\bigg[\begin{matrix}
-n&\beta\\
&\beta-\gamma-n+1
\end{matrix}\bigg|\,1-z\bigg],
\end{equation}
where $n\in\N=\{0,1,2,\ldots\}$. For more supercongruences involving the truncated hypergeometric functions, the readers may read \cite{AhOn00,BaSa1507,BD1805,DFLST16,He17a,He17b,He17c,KLMSY16,Kilbourn06,Liu17,Liu1610,Liu17b,
Long11,LoRa16,LTNZ1705,McCarthy12a,McCarthy12b,Mortenson03,Mortenson05,OsSc09,OsStZu1701,
OsZu16,RR1803,SunZH11,SunZH13,SunZW11,SunZW12,SunZW13,Swisher15,Tauraso12,Tauraso1701,Zudilin09}.

In this paper, we shall focus on the relation between the congruences concerning the truncated hypergeometric functions and the identities concerning the original hypergeometric functions. In particular, we can show that many hypergeometric identities have its $p$-adic analogue, i.e., a congruence modulo $p^2$ concerning the corresponding truncated hypergeometric function.

First, substituting $\gamma=1$ in the Gauss identity (\ref{Gaussidentity}), we get
\begin{equation}\label{alphabeta12F1}
{}_{2}F_1\bigg[\begin{matrix}
\alpha&\beta\\
&1
\end{matrix}\bigg|\,1\bigg]=\frac{\Gamma(1-\alpha-\beta)}{\Gamma(1-\alpha)\Gamma(1-\beta)}.
\end{equation}
In order to give a $p$-adic analogue of (\ref{alphabeta12F1}), we need the $p$-adic gamma function. For a prime $p$, let  $\Z_p$ denote the ring of all $p$-adic integers and let $$\Z_p^{\times}:=\{a\in\Z_p:\,a\text{ is prime to }p\}.$$
For each $\alpha\in\Z_p$, define the $p$-adic order $\nu_p(\alpha):=\max\{n\in\N:\, p^n\mid \alpha\}$ and the $p$-adic norm $|\alpha|_p:=p^{-\nu_p(\alpha)}$. Define the $p$-adic gamma function $\Gamma_p(\cdot)$ by
$$
\Gamma_p(n)=(-1)^n\prod_{\substack{1\leq j<n\\ (k,p)=1}}k,\qquad n=1,2,3,\ldots,
$$
and
$$
\Gamma_p(\alpha)=\lim_{\substack{|\alpha-n|_p\to 0\\ n\in\N}}\Gamma_p(n),\qquad \alpha\in\Z_p.
$$
In particular, we set $\Gamma_p(0)=1$. Throughout the whole paper, we only need to use the most basic properties of $\Gamma_p$, and all of them can be found in \cite{Murty02,Robert00}.
For example, we know that
\begin{equation}\label{Gammapx1Gammapx}
\frac{\Gamma_p(x+1)}{\Gamma_p(x)}=\begin{cases}-x,&\text{if }|x|_p=1,\\
-1,&\text{if }|x|_p>1.
\end{cases}
\end{equation}
Now we can give  the following $p$-adic analogue of (\ref{alphabeta12F1}).
\begin{theorem}\label{alphabeta11}
Let $p$ be an odd prime and $\alpha,\beta\in\Z_p$.
If $\langle-\alpha\rangle_p+\langle-\beta\rangle_p<p$, then
\begin{equation}\label{alphabeta11c1}
{}_2F_1\bigg[\begin{matrix}
\alpha&\beta\\
&1
\end{matrix}\bigg|\,1\bigg]_{p-1}\equiv-\frac{\Gamma_p(1-\alpha-\beta)}{\Gamma_p(1-\alpha)\Gamma_p(1-\beta)}\pmod{p^2}.
\end{equation}
And if $\langle-\alpha\rangle_p+\langle-\beta\rangle_p\geq p$, then
\begin{equation}\label{alphabeta11c2}
{}_2F_1\bigg[\begin{matrix}
\alpha&\beta\\
&1
\end{matrix}\bigg|\,1\bigg]_{p-1}\equiv\big(\alpha+\beta+\langle-\alpha\rangle_p+\langle-\beta\rangle_p-p\big)\cdot \frac{\Gamma_p(1-\alpha-\beta)}{\Gamma_p(1-\alpha)\Gamma_p(1-\beta)}\pmod{p^2}.
\end{equation}
\end{theorem}
For example, substituting $\alpha=1/3$ and $\beta=1/4$ in (\ref{alphabeta11c1}), we have
$$
{}_2F_1\bigg[\begin{matrix}
\frac13&\frac14\\
&1
\end{matrix}\bigg|\,1\bigg]_{p-1}\equiv-\frac{\Gamma_p(\frac5{12})}{\Gamma_p(\frac23)\Gamma_{p}(\frac34)}\pmod{p^2}
$$
for any prime $p\equiv 1\pmod{4}$.
Also, (\ref{SunZHalpha1}) can be deduced from (\ref{alphabeta11c1}) since 
\begin{equation}
\Gamma_p(\alpha)\Gamma_p(1-\alpha)=(-1)^{\langle-\alpha\rangle_p-1}.
\end{equation}
Furthermore, we mention that the negative sign in the right side of (\ref{alphabeta11c1}) is due to $\Gamma_p(1)=-1$.
In fact, as we shall see soon, if the ``$p-1$'' in the right subscript of the left side of (\ref{alphabeta11c1}) is replaced by ``$\langle-\gamma\rangle_p$'' , then we may get a complete $p$-adic analogue of (\ref{Gaussidentity}).

Before we introduce more results, let us briefly describe the way to prove (\ref{alphabeta11c1}). Our first observation is that if $\alpha=-\langle-\alpha\rangle_p$, then the truncated hypergeometric function is actually an original hypergeometric function. At this moment, the desired congruence easily follows from the identity (\ref{alphabeta12F1}). Thus for general $\alpha\in\Z_p$, we only need to prove
\begin{align}\label{alphabeta11c1ap}
&\frac1p\bigg({}_2F_1\bigg[\begin{matrix}
\alpha&\beta\\
&1
\end{matrix}\bigg|\,1\bigg]_{p-1}-{}_2F_1\bigg[\begin{matrix}
-a&\beta\\
&1
\end{matrix}\bigg|\,1\bigg]\bigg)\notag\\
\equiv&\frac1p\bigg(\frac{\Gamma_p(1+a-\beta)}{\Gamma_p(1+a)\Gamma_p(1-\beta)}-\frac{\Gamma_p(1-\alpha-\beta)}{\Gamma_p(1-\alpha)\Gamma_p(1-\beta)}\bigg)\pmod{p},
\end{align}
where $a=\langle-\alpha\rangle_p$. Obviously (\ref{alphabeta11c1ap}) is just a congruence modulo $p$, which can be easily deduced from the combinatorial identity
\begin{equation}\label{alphabeta11c1ape}
\sum_{k=0}^b\binom{b}{k}\sum_{j=1}^k\frac{(-1)^j}{j}\cdot\binom{a}{k-j}=-\binom{a+b}{b}\sum_{j=a+1}^{a+b}\frac1j
\end{equation}
where $b=\langle-\beta\rangle_p$. Furthermore, apparently (\ref{alphabeta11c1ape}) is an equivalent form of
$$
\frac{d}{dx}\bigg({}_2F_1\bigg[\begin{matrix}
-a+x&-b\\
&1
\end{matrix}\bigg|\,1\bigg]\bigg)\bigg|_{x=0}=
\frac{d}{dx}\bigg(\frac{\Gamma(1+a+b-x)}{\Gamma(1+a-x)\Gamma(1+b)}\bigg)\bigg|_{x=0}.
$$
Notice that there are two steps in the above proof of (\ref{alphabeta11c1}). One is that the congruence can be directly deduced from the corresponding  hypergeometric identity provided $\alpha=-\langle-\alpha\rangle_p$. Thus the original congruence modulo $p^2$ is reduced to a new congruence modulo $p$. The next step is that by taking the derivatives of the hypergeometric identity in the variable $\alpha$, we may get a combinatorial identity which easily implies the desired congruence modulo $p$.
Then the proof is complete. The above two steps both heavily depend on the corresponding  hypergeometric identity. In fact, as we shall see later, by constructing some suitable polynomials, the desired congruence modulo $p^2$ may be reduced to compute the derivatives of those polynomials modulo $p$, which often can be deduced by taking the derivatives in the both sides of corresponding hypergeometric identity.  With help of the same idea, in most cases, if at least one of $\alpha_0,\ldots,\alpha_m,\beta_1,\ldots,\beta_m$ is variable in a hypergeometric identity, then we might expect to get a congruence modulo $p^2$ involving the corresponding truncated hypergeometric function.

In the next section, we shall systematically introduce our main results of this paper, including the $p$-adic analogues of some  quadratic transformations of ${}_2F_1$, the transformation of balanced ${}_4F_3$ series, and Whipple's transformation on ${}_7F_6$ series. Throughout the whole paper, we always assume that $p$ is an odd prime.

\section{Main results}
\setcounter{lemma}{0}
\setcounter{theorem}{0}
\setcounter{corollary}{0}
\setcounter{remark}{0}
\setcounter{equation}{0}
\setcounter{conjecture}{0}

First, let us consider the $p$-adic analogue of the linear and quadratic transformations of the ${}_2F_1$ series.
Substituting $\beta=1-\alpha$ in the quadratic transformation (\cite[Eq. (3.1.3)]{AAR99})
\begin{equation}\label{alphabetaz4z1ze}
{}_2F_1\bigg[\begin{matrix}
\alpha&\beta\\
&\frac12+\frac12(\alpha+\beta)
\end{matrix}\bigg|\,z\bigg]={}_2F_1\bigg[\begin{matrix}
\frac12\alpha&\frac12\beta\\
&\frac12+\frac12(\alpha+\beta)
\end{matrix}\bigg|\,4z(1-z)\bigg],
\end{equation}
we may get
\begin{equation}\label{alphaz4z1ze}
{}_2F_1\bigg[\begin{matrix}
\alpha&1-\alpha\\
&1
\end{matrix}\bigg|\,z\bigg]=
{}_2F_1\bigg[\begin{matrix} 
\frac12\alpha&\frac12-\frac12\alpha\\
&1
\end{matrix}\bigg|\,4z(1-z)\bigg].
\end{equation}
The identity (\ref{alphaz4z1ze}) has the following $p$-adic analogue. 
\begin{theorem}\label{alphaz4z1z}
Suppose that $\alpha\in\Z_p$ and $\langle-\alpha\rangle_p$ is even. Then for any $z\in\Z_p$,
\begin{equation}\label{alphaz4z1zc}
{}_2F_1\bigg[\begin{matrix}
\alpha&1-\alpha\\
&1
\end{matrix}\bigg|\,z\bigg]_{p-1}\equiv
{}_2F_1\bigg[\begin{matrix}
\frac12\alpha&\frac12-\frac12\alpha\\
&1
\end{matrix}\bigg|\,4z(1-z)\bigg]_{p-1}\pmod{p^2}.
\end{equation}
\end{theorem}
For example, for each prime $p\equiv 1,3\pmod{5}$, we have
$$
{}_2F_1\bigg[\begin{matrix}
\frac25&\frac35\\
&1
\end{matrix}\bigg|\,z\bigg]_{p-1}\equiv
{}_2F_1\bigg[\begin{matrix}
\frac15&\frac3{10}\\
&1
\end{matrix}\bigg|\,4z(1-z)\bigg]_{p-1}\pmod{p^2}.
$$
In fact, as we shall see later, (\ref{alphaz4z1zc}) can viewed as a congruence for the polynomials in $z$. That is, for each $n\geq 0$, the coefficients of $z^n$ in the both sides of (\ref{alphaz4z1zc}) are congruent modulo $p^2$.

A similar example is the Clausen identity \cite[p. 116, Ex. 13]{AAR99}
\begin{equation}\label{Clausenidentity}
\bigg({}_2F_1\bigg[\begin{matrix}
\frac12\alpha&\frac12\beta\\
&\frac12+\frac12(\alpha+\beta)\end{matrix}\bigg|\,z\bigg]\bigg)^2=
{}_3F_2\bigg[\begin{matrix}
\alpha&\beta&\frac12(\alpha+\beta)\\
&\alpha+\beta&\frac12+\frac12(\alpha+\beta)
\end{matrix}\bigg|\,z\bigg].
\end{equation}
Combining (\ref{Clausenidentity}) with (\ref{alphabetaz4z1ze}), we get
\begin{equation}
\bigg({}_2F_1\bigg[\begin{matrix}
\alpha&\beta\\
&\frac12+\frac12(\alpha+\beta)
\end{matrix}\bigg|\,z\bigg]\bigg)^2=
{}_3F_2\bigg[\begin{matrix}
\alpha&\beta&\frac12(\alpha+\beta)\\
&\alpha+\beta&\frac12+\frac12(\alpha+\beta)
\end{matrix}\bigg|\,4z(1-z)\bigg].
\end{equation}
In particular,
\begin{equation}\label{alpha1alphaz124z1ze}
\bigg({}_2F_1\bigg[\begin{matrix}
\alpha&1-\alpha\\
&1\end{matrix}\bigg|\,z\bigg]\bigg)^2=
{}_3F_2\bigg[\begin{matrix}
\alpha&1-\alpha&\frac12\\
&1&1
\end{matrix}\bigg|\,4z(1-z)\bigg]
\end{equation}
We also have the $p$-adic analogue of (\ref{alpha1alphaz124z1ze}) as follows.
\begin{theorem}\label{alpha1alphaz124z1z}
Let $\alpha,z\in\Z_p$. Then
\begin{equation}\label{alpha1alphaz124z1zdxzn}\bigg({}_2F_1\bigg[\begin{matrix}
\alpha&1-\alpha\\
&1\end{matrix}\bigg|\,z\bigg]_{p-1}\bigg)^2\equiv
{}_3F_2\bigg[\begin{matrix}
\alpha&1-\alpha&\frac12\\
&1&1
\end{matrix}\bigg|\,4z(1-z)\bigg]_{p-1}\pmod{p^2}.
\end{equation}
\end{theorem}
Of course, (\ref{alpha1alphaz124z1zdxzn}) is also a congruence for  the polynomials in $z$.
Moreover, it follows from Theorems \ref{alphaz4z1z} and \ref{alpha1alphaz124z1z} that
\begin{corollary}
Suppose that $\alpha,z\in\Z_p$ and $\langle-\alpha\rangle_p$ is even.
\begin{equation}\bigg({}_2F_1\bigg[\begin{matrix}
\frac12\alpha&\frac12-\frac12\alpha\\
&1\end{matrix}\bigg|\,z\bigg]_{p-1}\bigg)^2\equiv
{}_3F_2\bigg[\begin{matrix}
\alpha&1-\alpha&\frac12\\
&1&1
\end{matrix}\bigg|\,z\bigg]_{p-1}\pmod{p^2}.
\end{equation}
\end{corollary}

However, most transformations of the ${}_2F_1$ series, unlike (\ref{alphabetaz4z1ze}) and (\ref{Clausenidentity}), will involve a factor of the form $(1-z)^{-\alpha}$, $(1-z/2)^{-\alpha}$ or  $(1-z)^{-\frac12\alpha}$. The main difficulty, which we must meet, is how to give the $p$-adic analogues of such factors. For example, the factor $(1-z)^{-\alpha}$ is involved in (\ref{Eulertransformation}).
So we need to introduce a function $\lambda_p$.

For any $a\in\Z$, $s\in\Z_p$ and $z\in\Z_p^\times$, define
\begin{equation}\label{zlambda}
z^{a+s(p-1)}:=z^a\cdot(z^{p-1})^{a}=z^{a}\sum_{k=0}^\infty\binom{s}{k}(z^{p-1}-1)^k.
\end{equation}
Since $z^{p-1}\equiv1\pmod{p}$, we have
\begin{align*}
z^{(s+1)(p-1)}=&\sum_{k=0}^\infty\binom{s+1}{k}(z^{p-1}-1)^k=
\sum_{k=0}^\infty\binom{s}{k}(z^{p-1}-1)^k+\sum_{k=1}^\infty\binom{s}{k-1}(z^{p-1}-1)^k\\
=&z^{s(p-1)}+(z^{p-1}-1)\cdot z^{s(p-1)}=z^{p-1}\cdot z^{s(p-1)},
\end{align*}
i.e., $z^{a+(s+1)(p-1)}=z^{(a+p-1)+sp}$.
Hence $z^{a+s(p-1)}$ is well-defined for each $z\in\Z_p^\times$.
For $\alpha\in\Z_p$,
define
\begin{equation}\label{lambdapalpha}
\lambda_p(\alpha):=-\langle-\alpha\rangle_p+\frac{\alpha+\langle-\alpha\rangle_p}{p}\cdot(p-1).
\end{equation}
As we shall see soon, $z^{-\lambda_p(\alpha)}$ can be viewed as a $p$-adic analogue of $z^{-\alpha}$ in most of ${}_2F_1$ transformations.
In fact, $\lambda_p$ also appears in a $p$-adic analogue of Gauss' multiplicative formula \cite[the proposition of Section 7.1.3]{Robert00}:
\begin{equation}\label{padicGMF}
\prod_{k=0}^{m-1}\Gamma_p\bigg(x+\frac km\bigg)=m^{-\lambda_p(mx)}\Gamma_p(mx)\prod_{k=0}^{m-1}\Gamma_p\bigg(\frac km\bigg),
\end{equation}
where $m\geq 1$ is prime to $p$. 
Furthermore, for any $z\in\Z_p^{\times}$, it is easy to see that
\begin{equation}\label{lambdalambdas}
z^{\lambda_p(\alpha)}\equiv z^{\lambda_p^{(2)}(\alpha)}\pmod{p^2},
\end{equation}
where
\begin{equation}
\lambda_p^{(2)}(\alpha):=\frac{\langle-\alpha\rangle_p-\langle-\alpha\rangle_{p^2}}{p}\cdot(p-1)-\langle-\alpha\rangle_p.
\end{equation}

Setting $\gamma=1$ in the Pfaff transformation (\ref{Eulertransformation}), we have
\begin{equation}\label{Euleralphabeta1}
{}_{2}F_1\bigg[\begin{matrix}
\alpha&\beta\\
&1
\end{matrix}\bigg|\,z\bigg]=(1-z)^{-\alpha}\cdot{}_{2}F_1\bigg[\begin{matrix}
\alpha&1-\beta\\
&1
\end{matrix}\bigg|\,\frac{z}{z-1}\bigg].
\end{equation}
Unfortunately, we don't know what the $p$-adic analogue of (\ref{Euleralphabeta1}) is.
On the other hand, we can give a partial $p$-adic analogue of
\begin{equation}\label{alphaalphazz1e}
{}_2F_1\bigg[\begin{matrix}
\alpha&\alpha\\
&1
\end{matrix}\bigg|\,z\bigg]=(1-z)^{-\alpha}{}_2F_1\bigg[\begin{matrix}
\alpha&1-\alpha\\
&1
\end{matrix}\bigg|\,\frac{z}{z-1}\bigg],
\end{equation}
which is evidently a special case of (\ref{Euleralphabeta1}). 
\begin{theorem}\label{alphaalphazz1} Let $\alpha\in\Z_p$ and $z\in\Z_p^\times$. 
Suppose that $z-1$ is prime to $p$. Then
\begin{align}\label{alphaalphazz1c}
&{}_2F_1\bigg[\begin{matrix}
\alpha&\alpha\\
&1
\end{matrix}\bigg|\,z\bigg]_{p-1}-z^{1-\lambda_p(\alpha)}{}_2F_1\bigg[\begin{matrix}
\alpha&\alpha\\
&1
\end{matrix}\bigg|\,\frac1z\bigg]_{p-1}\notag\\
\equiv&(1-z)^{1-\lambda_p(\alpha)}{}_2F_1\bigg[\begin{matrix}
\alpha&1-\alpha\\
&1
\end{matrix}\bigg|\,\frac{z}{z-1}\bigg]_{p-1}\pmod{p^2}.
\end{align}
\end{theorem}
For example, for any $p\equiv 1\pmod{3}$, we have
\begin{align*}
&{}_2F_1\bigg[\begin{matrix}
\frac13&\frac13\\
&1
\end{matrix}\bigg|\,z\bigg]_{p-1}-z^{1+\frac{1}{3}p(p-1)}{}_2F_1\bigg[\begin{matrix}
\frac13&\frac13\\
&1
\end{matrix}\bigg|\,\frac1z\bigg]_{p-1}\\
\equiv&(1-z)^{1+\frac{1}{3}p(p-1)}{}_2F_1\bigg[\begin{matrix}
\frac13&\frac23\\
&1
\end{matrix}\bigg|\,\frac{z}{z-1}\bigg]_{p-1}\pmod{p^2}.
\end{align*}
On the other hand, if $z$ is replaced by $z^{-1}$ in (\ref{alphaalphazz1e}), we may get
\begin{align*}
{}_2F_1\bigg[\begin{matrix}
\alpha&1-\alpha\\
&1
\end{matrix}\bigg|\,\frac{z}{z-1}\bigg]_{p-1}\equiv
(-1)^{\langle-\alpha\rangle_p}{}_2F_1\bigg[\begin{matrix}
\alpha&1-\alpha\\
&1
\end{matrix}\bigg|\,\frac{1}{1-z}\bigg]_{p-1}
\pmod{p^2},
\end{align*}
which is apparently an equivalent form of (\ref{SunZHalphaz}).

Let us turn to another quadratic transformations. 
Letting $y=z/(z-1)$, clearly we have $4y(1-y)=-4z/(1-z)^2$. So combining  (\ref{alphaalphazz1e}) with (\ref{alphaz4z1ze}), we have
\begin{equation}\label{alpha2alpha214z1ze}
{}_2F_1\bigg[\begin{matrix}
\alpha&\alpha\\
&1
\end{matrix}\bigg|\,z\bigg]=(1-z)^{-\alpha}{}_2F_1\bigg[\begin{matrix}\frac12\alpha&\frac12-\frac12\alpha\\
&1
\end{matrix}\bigg|\,-\frac{4z}{(1-z)^2}\bigg].
\end{equation}
Likewise, combining (\ref{alphaalphazz1c}) with (\ref{alphaz4z1zc}), we also obtain  a $p$-adic analogue of (\ref{alpha2alpha214z1ze}).
\begin{theorem}\label{alpha2alpha214z1z}
Suppose that $\alpha\in\Z_p$ with $\langle-\alpha\rangle_p$ is even,
and that $z\in\Z_p^\times$ with $z-1$ is prime to $p$. Then
\begin{align}\label{alpha2alpha214z1zc}
&{}_2F_1\bigg[\begin{matrix}
\alpha&\alpha\\
&1
\end{matrix}\bigg|\,z\bigg]_{p-1}-z^{1-\lambda_p(\alpha)}{}_2F_1\bigg[\begin{matrix}
\alpha&\alpha\\
&1
\end{matrix}\bigg|\,\frac{1}{z}\bigg]_{p-1}\notag\\
\equiv&(1-z)^{1-\lambda_p(\alpha)}{}_2F_1\bigg[\begin{matrix}
\frac12\alpha&\frac12-\frac12\alpha\\
&1
\end{matrix}\bigg|\,-\frac{4z}{(1-z)^2}\bigg]_{p-1}\pmod{p^2}.
\end{align}
\end{theorem}
Let us give an explanation on (\ref{alpha2alpha214z1zc}). Replacing $z$ by $z^{-1}$ in (\ref{alpha2alpha214zz1e}), we get
\begin{equation}\label{alpha2alpha214z1ze2}
z^{-\alpha}{}_2F_1\bigg[\begin{matrix}
\alpha&\alpha\\
&1
\end{matrix}\bigg|\,\frac1z\bigg]=(z-1)^{-\alpha}{}_2F_1\bigg[\begin{matrix}\frac12\alpha&\frac12-\frac12\alpha\\
&1
\end{matrix}\bigg|\,-\frac{4z}{(1-z)^2}\bigg].
\end{equation}
Thus when $\alpha$ is an even integer, 
\begin{equation}\label{alpha2alpha214z1ze3}
{}_2F_1\bigg[\begin{matrix}
\alpha&\alpha\\
&1
\end{matrix}\bigg|\,z\bigg]-z^{1-\alpha}{}_2F_1\bigg[\begin{matrix}
\alpha&\alpha\\
&1
\end{matrix}\bigg|\,\frac1z\bigg]=(1-z)^{1-\alpha}{}_2F_1\bigg[\begin{matrix}\frac12\alpha&\frac12-\frac12\alpha\\
&1
\end{matrix}\bigg|\,-\frac{4z}{(1-z)^2}\bigg].
\end{equation}
Evidently (\ref{alpha2alpha214z1zc}) factly can be viewed as a $p$-adic of (\ref{alpha2alpha214z1ze3}). In fact, in the $p$-adic analogue of a quadratic transformation of the form
$$
{}_2F_1\bigg[\begin{matrix}
*&*\\
&1
\end{matrix}\bigg|\,z\bigg]=\big(1-\upsilon(z)\big)^{-\alpha}{}_2F_1\bigg[\begin{matrix}*&*\\
&1
\end{matrix}\bigg|\,\Omega(z)\bigg]
$$
where $\Omega$ is a quadratic rational function, there often will be two terms appearing in the left side: one is concerning $z$, and the other is concerning $\varrho(z)$, where $\varrho$ is a rational function such that $\Omega\big(\varrho(z)\big)=\Omega(z)$.

Similarly, substituting $\beta=\alpha$ in the transformation \cite[Eq. (3.1.9)]{AAR99}
\begin{equation}\label{alphabetaalphabeta14zz1}
{}_2F_1\bigg[\begin{matrix}
\alpha&\beta\\
&\alpha-\beta+1
\end{matrix}\bigg|\,z\bigg]=(1+z)^{-\alpha}{}_2F_1\bigg[\begin{matrix}\frac12\alpha&\frac12+\frac12\alpha\\
&\alpha-\beta+1
\end{matrix}\bigg|\,\frac{4z}{(1+z)^2}\bigg],
\end{equation}
we have
\begin{equation}\label{alpha2alpha214zz1e}
{}_2F_1\bigg[\begin{matrix}
\alpha&\alpha\\
&1
\end{matrix}\bigg|\,z\bigg]=(1+z)^{-\alpha}{}_2F_1\bigg[\begin{matrix}\frac12\alpha&\frac12+\frac12\alpha\\
&1
\end{matrix}\bigg|\,\frac{4z}{(1+z)^2}\bigg].
\end{equation}
\begin{theorem}\label{alpha2alpha214zz1}
Suppose that $\alpha\in\Z_p$ and $z\in\Z_p^\times$ with $z+1$ is prime to $p$.
\begin{align}\label{alpha2alpha214zz1c}
&{}_2F_1\bigg[\begin{matrix}
\alpha&\alpha\\
&1
\end{matrix}\bigg|\,z\bigg]_{p-1}+z^{1-\lambda_p(\alpha)}{}_2F_1\bigg[\begin{matrix}
\alpha&\alpha\\
&1
\end{matrix}\bigg|\,\frac{1}{z}\bigg]_{p-1}\notag\\
\equiv&(1+z)^{1-\lambda_p(\alpha)}{}_2F_1\bigg[\begin{matrix}
\frac12\alpha&\frac12+\frac12\alpha\\
&1
\end{matrix}\bigg|\,\frac{4z}{(1+z)^2}\bigg]_{p-1}\pmod{p^2}.
\end{align}
\end{theorem}

Consider the quadratic transformation \cite[Theorem 3.1.3]{AAR99}
\begin{equation}\label{alphabeta12zalpha12zz1e}
{}_2F_1\bigg[\begin{matrix}
\alpha&\beta\\
&2\beta
\end{matrix}\bigg|\,z\bigg]=\bigg(1-\frac z2\bigg)^{-\alpha}{}_2F_1\bigg[\begin{matrix}\frac12\alpha&\frac12+\frac12\alpha\\
&\frac12+\beta
\end{matrix}\bigg|\,\frac{z^2}{(2-z)^2}\bigg].
\end{equation}
Replacing $\beta$ by $1/2$ in (\ref{alphabeta12zalpha12zz1e}), we have
\begin{equation}\label{alpha12zalpha12zz1e}
{}_2F_1\bigg[\begin{matrix}
\alpha&\frac12\\
&1
\end{matrix}\bigg|\,z\bigg]=\bigg(1-\frac z2\bigg)^{-\alpha}{}_2F_1\bigg[\begin{matrix}
\frac12\alpha&\frac12+\frac12\alpha\\
&1
\end{matrix}\bigg|\,\frac{z^2}{(z-2)^2}\bigg].
\end{equation}
\begin{theorem}\label{alpha12zalpha12zz1}
Let $\alpha,z\in\Z_p$. Suppose that both $z-1$ and $z-2$ are prime to $p$.
\begin{align}\label{alpha12zalpha12zz1G}
&{}_2F_1\bigg[\begin{matrix}
\alpha&\frac12\\
&1
\end{matrix}\bigg|\,z\bigg]_{p-1}+(1-z)^{1-\lambda_p(\alpha)}{}_2F_1\bigg[\begin{matrix}
\alpha&\frac12\\
&1
\end{matrix}\bigg|\,\frac{z}{z-1}\bigg]_{p-1}\notag\\
\equiv&
2\bigg(1-\frac z2\bigg)^{1-\lambda_p(\alpha)}{}_2F_1\bigg[\begin{matrix}
\frac12\alpha&\frac12+\frac12\alpha\\
&1
\end{matrix}\bigg|\,\frac{z^2}{(z-2)^2}\bigg]_{p-1}\pmod{p^2}.
\end{align}
\end{theorem}

Moreover, since $y^2/(y-2)^2=4z^2/(1+z^2)^2$ if $y=4z/(1+z)^2$, combining (\ref{alpha2alpha214zz1e}) and (\ref{alpha12zalpha12zz1e}), we get
\begin{equation}\label{alphaalphaz2alpha124zz12e}
{}_2F_1\bigg[\begin{matrix}
\alpha&\alpha\\
&1
\end{matrix}\bigg|\,z^2\bigg]=(1+z)^{-2\alpha}{}_2F_1\bigg[\begin{matrix}
\alpha&\frac12\\
&1
\end{matrix}\bigg|\,\frac{4z}{(1+z)^2}\bigg].
\end{equation}
\begin{theorem}\label{alphaalphaz2alpha124zz12}
Suppose that $\alpha,z\in\Z_p$ with $z(1+z)$ is prime to $p$.
\begin{align}\label{alphaalphaz2alpha124zz12A}
&{}_2F_1\bigg[\begin{matrix}
\alpha&\alpha\\
&1
\end{matrix}\bigg|\,z^2\bigg]_{p-1}+z^{1-2\lambda_p(\alpha)}{}_2F_1\bigg[\begin{matrix}
\alpha&\alpha\\
&1
\end{matrix}\bigg|\,\frac1{z^2}\bigg]_{p-1}\notag\\
\equiv&
(1+z)^{1-2\lambda_p(\alpha)}{}_2F_1\bigg[\begin{matrix}
\alpha&\frac12\\
&1
\end{matrix}\bigg|\,\frac{4z}{(1+z)^2}\bigg]_{p-1}\pmod{p^2}.
\end{align}
\end{theorem}
Replacing $z$ by $4z/(1+z)$ in (\ref{alpha12zalpha12zz1G}) and applying (\ref{alpha2alpha214zz1c}), we obtain that
\begin{align*}
&{}_2F_1\bigg[\begin{matrix}
\alpha&\frac12\\
&1
\end{matrix}\bigg|\,\frac{4z}{(1+z)^2}\bigg]_{p-1}+\bigg(\frac{1-z}{1+z}\bigg)^{2-2\lambda_p(\alpha)}{}_2F_1\bigg[\begin{matrix}
\alpha&\frac12\\
&1
\end{matrix}\bigg|\,-\frac{4z}{(1-z)^2}\bigg]_{p-1}\notag\\
\equiv&
\frac{2\cdot(1+z^2)^{1-\lambda_p(\alpha)}}{(1+z)^{2-2\lambda_p(\alpha)}}{}_2F_1\bigg[\begin{matrix}
\frac12\alpha&\frac12+\frac12\alpha\\
&1
\end{matrix}\bigg|\,\frac{4z^2}{(1+z^2)^2}\bigg]_{p-1}\\
\equiv&\frac{2}{(1+z)^{2-2\lambda_p(\alpha)}}
\bigg({}_2F_1\bigg[\begin{matrix}
\alpha&\alpha\\
&1
\end{matrix}\bigg|\,z^2\bigg]_{p-1}+z^{2-2\lambda_p(\alpha)}{}_2F_1\bigg[\begin{matrix}
\alpha&\alpha\\
&1
\end{matrix}\bigg|\,\frac{1}{z^2}\bigg]_{p-1}\bigg)\pmod{p^2}.
\end{align*}
It follows that
\begin{align}\label{alphaalphaz2alpha124zz12B}
&{}_2F_1\bigg[\begin{matrix}
\alpha&\alpha\\
&1
\end{matrix}\bigg|\,z^2\bigg]_{p-1}-z^{1-2\lambda_p(\alpha)}{}_2F_1\bigg[\begin{matrix}
\alpha&\alpha\\
&1
\end{matrix}\bigg|\,\frac1{z^2}\bigg]_{p-1}\notag\\
\equiv&
(1-z)^{1-2\lambda_p(\alpha)}{}_2F_1\bigg[\begin{matrix}
\alpha&\frac12\\
&1
\end{matrix}\bigg|\,-\frac{4z}{(1-z)^2}\bigg]_{p-1}\pmod{p^2}.
\end{align}
Of course, although the above process requires that both $1+z$ and $1+z^2$ is prime to $p$,
(\ref{alphaalphaz2alpha124zz12B}) is still valid for those exceptional $z$'s, since (\ref{alphaalphaz2alpha124zz12B}) can be proved directly in the similar way as (\ref{alphaalphaz2alpha124zz12A}).

Finally, note that $y/(y-1)=z^2/(4z-4)$ if $y=z^2/(z-2)^2$. It follows from (\ref{alphaalphazz1e}) and (\ref{alpha12zalpha12zz1e}) that
\begin{equation}\label{alpha12zalpha12z1ze}
{}_2F_1\bigg[\begin{matrix}
\alpha&\frac12\\
&1
\end{matrix}\bigg|\,z\bigg]=(1-z)^{-\frac12\alpha}{}_2F_1\bigg[\begin{matrix}
\frac12\alpha&\frac12-\frac12\alpha\\
&1
\end{matrix}\bigg|\,-\frac{z^2}{4z-4}\bigg].
\end{equation}
However, the $p$-adic analogue of (\ref{alpha12zalpha12z1ze}) looks a little different.
\begin{theorem}\label{alpha12zalpha12z1z}
Let $\alpha,z\in\Z_p$. If $\langle-\alpha\rangle_p$ is even and $z-1$ is prime to $p$, then
\begin{align}\label{alpha12zalpha12z1zG}
&{}_2F_1\bigg[\begin{matrix}
\alpha&\frac12\\
&1
\end{matrix}\bigg|\,z\bigg]_{p-1}+(1-z)^{\langle-\alpha\rangle_p}{}_2F_1\bigg[\begin{matrix}
\alpha&\frac12\\
&1
\end{matrix}\bigg|\,\frac{z}{z-1}\bigg]_{p-1}\notag\\
\equiv&
2(1-z)^{\frac12\langle-\alpha\rangle_p}{}_2F_1\bigg[\begin{matrix}
\frac12\alpha&\frac12-\frac12\alpha\\
&1
\end{matrix}\bigg|\,\frac{z^2}{4z-4}\bigg]_{p-1}\pmod{p^2}.
\end{align}
\end{theorem}
Evidently in (\ref{alpha12zalpha12z1zG}), the $p$-adic analogue of $(1-z)^{-\alpha}$ is $(1-z)^{\langle-\alpha\rangle_p}$, rather than $(1-z)^{-\lambda_p(\alpha)}$. For example, for prime $p\equiv1\pmod{5}$,
$$
{}_2F_1\bigg[\begin{matrix}
\frac15&\frac12\\
&1
\end{matrix}\bigg|\,z\bigg]_{p-1}+(1-z)^{\frac{p-1}5}{}_2F_1\bigg[\begin{matrix}
\frac15&\frac12\\
&1
\end{matrix}\bigg|\,\frac{z}{z-1}\bigg]_{p-1}\equiv
2(1-z)^{\frac{p-1}{10}}{}_2F_1\bigg[\begin{matrix}
\frac1{10}&\frac25\\
&1
\end{matrix}\bigg|\,\frac{z^2}{4z-4}\bigg]_{p-1}\pmod{p^2}.
$$

Now we have listed all $p$-adic quadratic ${}_2F_1$ transformations in this paper. In particular, substituting the special values of $\alpha$ and $z$ in the above $p$-adic transformations, we may obtain some simple consequences.
For example, substituting $z=1/2$ in (\ref{alphaz4z1zc}) and applying (\ref{alphabeta11c1}), we may get
\begin{equation}\label{2F1alpha1alpha12}
{}_2F_1\bigg[\begin{matrix}
\alpha&1-\alpha\\
&1
\end{matrix}\bigg|\,\frac12\bigg]_{p-1}\equiv
{}_2F_1\bigg[\begin{matrix}
\frac12\alpha&\frac12-\frac12\alpha\\
&1
\end{matrix}\bigg|\,1\bigg]_{p-1}\equiv-\frac{\Gamma_p(\frac12)}{\Gamma_p(1-\frac12\alpha)\Gamma_p(\frac12+\frac12\alpha)}\pmod{p^2}
\end{equation}
which was firstly proved by Liu \cite{Liu1610}. 
On the other hand, in \cite{CoHa91}, Coster and Hamme established a connection between the Legendre polynomials modulo $p^2$ and the elliptic curves having complex multiplication. For example, according to their results, we can obtain that
$$
{}_2F_1\bigg[\begin{matrix}
\frac12&\frac12\\
&1
\end{matrix}\bigg|\,-1\bigg]_{p-1}\equiv (-1)^{\frac{p-1}{4}}(a+b\sqrt{-1})\pmod{p^2},
$$
where the prime $p=a^2+b^2\equiv 1\pmod{4}$ with $a\equiv 1\pmod{4}$, and $\sqrt{-1}\in\Z_p$ with $\sqrt{-1}\equiv a/b\pmod{p}$. So by (\ref{alphaz4z1zc}) and using the Gross-Koblitz formula, we may get
\begin{equation}
{}_2F_1\bigg[\begin{matrix}
\frac14&\frac14\\
&1
\end{matrix}\bigg|\,-8\bigg]_{p-1}\equiv (-1)^{\frac{p+3}{4}}\cdot\frac{\Gamma_p(\frac12)}{\Gamma_p(\frac34)^2}\pmod{p^2}
\end{equation}
for each prime $p\equiv 1\pmod{4}$. In Section \ref{sectionq2F1z4z1z}, we shall explain the relation between  the Legendre polynomials and ${}_2F_1\bigg[\begin{matrix}
\frac14&\frac14\\
&1
\end{matrix}\bigg|\,z\bigg]_{p-1}$ modulo $p^2$. Furthermore, in Section \ref{section2F1cm}, with help of the results of Coster and Hamme, we can show that for some special values of $z$,
$$
{}_2F_1\bigg[\begin{matrix}
\frac14&\frac14\\
&1
\end{matrix}\bigg|\,z\bigg]_{p-1}\equiv (1-z)^{-\lambda_p(\frac14)}
{}_2F_1\bigg[\begin{matrix}
\frac14&\frac34\\
&1
\end{matrix}\bigg|\,\frac z{z-1}\bigg]_{p-1}\pmod{p^2},
$$ 
which is a complete $p$-adic analogue of (\ref{alphaalphazz1e}) for $\alpha=1/4$, provided that the corresponding elliptic curve for $z$ has complex multiplication.
However, for those cubic ${}_2F_1$ transformations, e.g., Ramanujan's cubic transformation
$$
{}_2F_1\bigg[\begin{matrix}
\frac13&\frac23\\
&1
\end{matrix}\bigg|\,-\bigg(\frac{1-z}{1+2z}\bigg)^3\bigg]=(1+2z){}_2F_1\bigg[\begin{matrix}
\frac13&\frac23\\
&1
\end{matrix}\bigg|\,z^3\bigg],
$$
we have no idea on their $p$-adic analogues.

Let us consider the identities involving ${}_3F_2$ series. The first one is the Watson identity \cite[Theorem 3.5.5 (i)]{AAR99}
\begin{equation}\label{Watsonidentity}
{}_3F_2\bigg[\begin{matrix}
\alpha&\beta&\gamma\\
&2\beta&\frac12(\alpha+\gamma+1)
\end{matrix}\bigg|\,1\bigg]=\frac{\Gamma(\frac12)\Gamma(\frac12+\beta)\Gamma(\frac12+\frac12(\alpha+\gamma))\Gamma(\frac12+\beta-\frac12(\alpha+\gamma))}{\Gamma(\frac12+\frac12\alpha)\Gamma(\frac12+\frac12\gamma)\Gamma(\frac12+\beta-\frac12\alpha)\Gamma(\frac12+\beta-\frac12\gamma)}.
\end{equation}
Substituting $\gamma=1-\alpha$ in (\ref{Watsonidentity}), we get
\begin{equation}\label{alphaalphabeta2betaH}
{}_3F_2\bigg[\begin{matrix}
\alpha&1-\alpha&\beta\\
&1&2\beta
\end{matrix}\bigg|\,1\bigg]=\frac{\Gamma(\frac12)\Gamma(\frac12+\beta)\Gamma(\beta)}{\Gamma(\frac12+\frac12\alpha)\Gamma(1-\frac12\alpha)\Gamma(\frac12+\beta-\frac12\alpha)\Gamma(\beta+\frac12\alpha)}.
\end{equation}
\begin{theorem}\label{alphaalphabeta2beta}
Let $\alpha,\beta\in\Z_p$.
Suppose that $\langle-\beta\rangle_p<p/2$ and $(2\beta)_{p-1}\not\equiv0\pmod{p^2}$. If $\langle-\alpha\rangle_p$ is even, then
\begin{align}\label{alphaalphabeta2betaAG}
{}_3F_2\bigg[\begin{matrix}
\alpha&1-\alpha&\beta\\
&1&2\beta
\end{matrix}\bigg|\,1\bigg]_{p-1}\equiv
-\frac{\Gamma_p(\frac12)\Gamma_p(\frac12+\beta)\Gamma_p(\beta)}{\Gamma_p(\frac12+\frac12\alpha)\Gamma_p(1-\frac12\alpha)\Gamma_p(\frac12+\beta-\frac12\alpha)\Gamma_p(\beta+\frac12\alpha)}\pmod{p^2}.
\end{align}
On the other hand, if $\langle-\alpha\rangle_p$ is odd, then
\begin{align}\label{alphaalphabeta2betaBG}
{}_3F_2\bigg[\begin{matrix}
\alpha&1-\alpha&\beta\\
&1&2\beta
\end{matrix}\bigg|\,1\bigg]_{p-1}
\equiv0\pmod{p^2}.
\end{align}
\end{theorem}
Notice that two requirements $\langle-\beta\rangle_p<p/2$ and $(2\beta)_{p-1}\not\equiv0\pmod{p^2}$ of Theorem \ref{alphaalphabeta2beta} are both necessary. In fact, if $\langle-\beta\rangle_p>p/2$, then for any $\langle-2\beta\rangle_p<k\leq \langle-\beta\rangle_p$, we shall have $(\beta)_k/(2\beta)_{k}$ is not $p$-integral. Also, the condition $(2\beta)_{p-1}\not\equiv0\pmod{p^2}$ can prevent a high power of $p$ appearing in the denominators of the left sides of (\ref{alphaalphabeta2betaAG}) and
(\ref{alphaalphabeta2betaBG}).
As the examples of Theorem \ref{alphaalphabeta2beta}, we have
$$
{}_3F_2\bigg[\begin{matrix}
\frac12&\frac12&\frac13\\
&1&\frac23
\end{matrix}\bigg|\,1\bigg]_{p-1}\equiv-\frac{\Gamma_p(\frac13)\Gamma_p(\frac12)\Gamma_p(\frac56)}{\Gamma_p(\frac34)^2\Gamma_p(\frac7{12})^2}\pmod{p^2}
$$
for any prime $p\equiv 1\pmod{12}$
and
$$
{}_3F_2\bigg[\begin{matrix}
\frac12&\frac12&\frac13\\
&1&\frac23
\end{matrix}\bigg|\,1\bigg]_{p-1}\equiv0\pmod{p^2}
$$
for any prime $p\equiv 7\pmod{12}$.

The next is Dixon's well-poised sum \cite[Theorem 3.4.1]{AAR99}
\begin{align}\label{dixonwellpoisedsum}
{}_3F_2\bigg[\begin{matrix}
\alpha&\beta&\gamma\\
&\alpha-\beta+1&\alpha-\gamma+1
\end{matrix}\bigg|\,1\bigg]=\frac{\Gamma(\frac12\alpha+1)\Gamma(\alpha-\beta+1)\Gamma(\alpha-\gamma+1)\Gamma(\frac12\alpha-\beta-\gamma+1)}{\Gamma(\alpha+1)\Gamma(\frac12\alpha-\beta+1)\Gamma(\frac12\alpha-\gamma+1)\Gamma(\alpha-\beta-\gamma+1)}.
\end{align}
Letting $\gamma=\alpha$ in (\ref{dixonwellpoisedsum}), we have
\begin{equation}\label{alphaalphabetaalphabeta1H}
{}_3F_2\bigg[\begin{matrix}
\alpha&\alpha&\beta\\
&1&\alpha-\beta+1
\end{matrix}\bigg|\,1\bigg]=\frac{\Gamma(1+\frac12\alpha)\Gamma(1+\alpha-\beta)\Gamma(1-\frac12\alpha-\beta)}{\Gamma(\alpha+1)\Gamma(1-\frac12\alpha)\Gamma(1-\beta)\Gamma(1+\frac12\alpha-\beta)}.
\end{equation}
\begin{theorem}\label{alphaalphabetaalphabeta1}
Let $\alpha,\beta\in\Z_p$.
 Suppose that $p^2$ doesn't divide $(\alpha-\beta+1)_{p-1}$.

\medskip\noindent
(1) Suppose that $\langle-\alpha\rangle_p$ is even and $\langle-\alpha\rangle_p\leq \langle -\beta\rangle_p<(p-\langle-\alpha\rangle_p)/2$. Then
\begin{align}\label{alphaalphabetaalphabeta1GA}
{}_3F_2\bigg[\begin{matrix}\alpha&\alpha&\beta\\
&1&\alpha-\beta+1\end{matrix}\bigg|\,1\bigg]_{p-1}
\equiv-\frac{2\Gamma_p(1+\frac12\alpha)\Gamma_p(1+\alpha-\beta)\Gamma_p(1-\frac12\alpha-\beta)}{
\Gamma_p(1+\alpha)\Gamma_p(1-\frac{1}2\alpha)\Gamma_p(1-\beta)\Gamma_p(1+\frac{1}2\alpha-\beta)}\pmod{p^2}.
\end{align}

\medskip\noindent
(2) Suppose that $\langle-\alpha\rangle_p$ is odd and $\max\{\langle-\alpha\rangle_p,(p-\langle-\alpha\rangle_p)/2\}\leq \langle-\beta\rangle_p<(p+\langle-\alpha\rangle_p)/2$. Then
\begin{align}\label{alphaalphabetaalphabeta1GB}
{}_3F_2\bigg[\begin{matrix}\alpha&\alpha&\beta\\
&1&\alpha-\beta+1\end{matrix}\bigg|\,1\bigg]_{p-1}\equiv0\pmod{p^2}.
\end{align}

\medskip\noindent
(3) Suppose that $\langle-\alpha\rangle_p$ is odd and $\langle-\alpha\rangle_p\leq \langle-\beta\rangle_p<(p-\langle-\alpha\rangle_p)/2$. Then
\begin{align}\label{alphaalphabetaalphabeta1GC}
{}_3F_2\bigg[\begin{matrix}\alpha&\alpha&\beta\\
&1&\alpha-\beta+1\end{matrix}\bigg|\,1\bigg]_{p-1}\equiv
-\frac{(\alpha+\langle-\alpha\rangle_p)\cdot\Gamma_p(1+\frac12\alpha)\Gamma_p(1+\alpha-\beta)\Gamma_p(1-\frac12\alpha-\beta)}{
\Gamma_p(1+\alpha)\Gamma_p(1-\frac{1}2\alpha)\Gamma_p(1-\beta)\Gamma_p(1+\frac{1}2\alpha-\beta)}\pmod{p^2}.
\end{align}
\end{theorem}
For example,
$$
{}_3F_2\bigg[\begin{matrix}\frac12&\frac12&\frac13\\
&1&\frac 76\end{matrix}\bigg|\,1\bigg]_{p-1}\equiv-\frac{2\Gamma_p(\frac5{12})\Gamma_p(\frac 76)\Gamma_p(\frac54)}{
\Gamma_p(\frac23)\Gamma_p(\frac34)\Gamma_p(\frac{11}{12})\Gamma_p(\frac32)}\pmod{p^2}
$$
for any prime $p\equiv 5\pmod{12}$ and
$$
{}_3F_2\bigg[\begin{matrix}\frac12&\frac12&\frac13\\
&1&\frac 76\end{matrix}\bigg|\,1\bigg]_{p-1}\equiv0\pmod{p^2}
$$
for any prime $p\equiv 11\pmod{12}$.

Third, the Pfaff-Saalsch\"utz theorem \cite[Theorem 2.2.6]{AAR99} says that
\begin{align}\label{pffafsaalschutz}
{}_3F_2\bigg[\begin{matrix}
-n&\alpha&\beta\\
&\gamma&\delta
\end{matrix}\bigg|\,1\bigg]
=\frac{(\gamma-\alpha)_n(\gamma-\beta)_n}{(\gamma)_n(\gamma-\alpha-\beta)_n},
\end{align}
where $n\in\N$ and $\gamma+\delta=\alpha+\beta+1-n$. In particular, setting $\gamma=1$ in (\ref{pffafsaalschutz}), we get
\begin{align}\label{nalphabeta1alphabetanH}
{}_3F_2\bigg[\begin{matrix}
-n&\alpha&\beta&\\
&1&\alpha+\beta-n
\end{matrix}\bigg|\,1\bigg]
=\frac{(1-\alpha)_n(1-\beta)_n}{n!\cdot (1-\alpha-\beta)_n}.
\end{align}
\begin{theorem}\label{nalphabeta1alphabetan}
Let $\alpha,\beta,\gamma\in\Z_p$.
Suppose that $\max\{\langle-\alpha\rangle_p,\langle-\beta\rangle_p\}\leq\langle-\gamma\rangle_p$ and $(\gamma)_{p-1}$ is not divisible by $p^2$. If
$
\langle-\alpha\rangle_p+\langle-\beta\rangle_p<\langle-\gamma\rangle_p$,
then
\begin{align}\label{nalphabeta1alphabetanG}
{}_3F_2\bigg[\begin{matrix}\alpha&\beta&\gamma-\alpha-\beta\\
&1&\gamma\end{matrix}\bigg|\,1\bigg]_{p-1}
\equiv-\frac{\Gamma_p(1+\alpha-\gamma)\Gamma_p(1+\beta-\gamma)\Gamma_p(1-\alpha-\beta)}{\Gamma_p(1-\alpha)\Gamma_p(1-\beta)\Gamma_p(1-\gamma)\Gamma_p(1+\alpha+\beta-\gamma)}\pmod{p^2}.
\end{align}
If
$\langle-\alpha\rangle_p+\langle-\beta\rangle_p>p$,
then
\begin{equation}\label{nalphabeta1alphabetanGB}
{}_3F_2\bigg[\begin{matrix}\alpha&\beta&\gamma-\alpha-\beta\\
&1&\gamma\end{matrix}\bigg|\,1\bigg]_{p-1}\equiv
0\pmod{p^2}.
\end{equation}
\end{theorem}
For example, 
$$
{}_3F_2\bigg[\begin{matrix}\frac14&\frac13&\frac1{12}\\
&1&\frac23\end{matrix}\bigg|\,1\bigg]_{p-1}\equiv
-\frac{\Gamma_p(\frac5{12})\Gamma_p(\frac{7}{12})}{\Gamma_p(\frac13)\Gamma_p(\frac34)\Gamma_p(\frac{11}{12})}\pmod{p^2}
$$
for any prime $p\equiv 1\pmod{12}$ and 
$$
{}_3F_2\bigg[\begin{matrix}\frac14&\frac13&\frac1{12}\\
&1&\frac23\end{matrix}\bigg|\,1\bigg]_{p-1}\equiv0\pmod{p^2}
$$
for any prime $p\equiv 7\pmod{12}$. We also mention that the special case $\gamma=1$ of (\ref{nalphabeta1alphabetanGB}) was proved by Pan and Zhang in \cite{PZ15}.

We also have the following ${}_3F_2$ transformation due to Kummer \cite[Corollary 3.3.5]{AAR99}:
\begin{equation}\label{Kummer3F2transformation}
{}_3F_2\bigg[\begin{matrix}\alpha&\beta&\gamma\\
&\delta&\epsilon\end{matrix}\bigg|\,1\bigg]\equiv\frac{\Gamma(\delta)\Gamma(\delta+\epsilon-\alpha-\beta-\gamma)}{
\Gamma(\delta-\alpha)\Gamma(\delta+\epsilon-\beta-\gamma)}\cdot{}_3F_2\bigg[\begin{matrix}\alpha&\epsilon-\beta&\epsilon-\gamma\\
&\epsilon&\epsilon+\epsilon-\beta-\gamma\end{matrix}\bigg|\,1\bigg].
\end{equation}
Setting $\epsilon=1$, we get
\begin{equation}\label{alphabetagamma1delta13F2H}
{}_3F_2\bigg[\begin{matrix}\alpha&\beta&\gamma\\
&1&\delta\end{matrix}\bigg|\,1\bigg]=\frac{\Gamma(\delta)\Gamma(1+\delta-\alpha-\beta-\gamma)}{
\Gamma(\delta-\alpha)\Gamma(1+\delta-\beta-\gamma)}\cdot{}_3F_2\bigg[\begin{matrix}\alpha&1-\beta&1-\gamma\\
&1&1+\delta-\beta-\gamma\end{matrix}\bigg|\,1\bigg].
\end{equation}
\begin{theorem}\label{alphabetagamma1delta13F2}
Let $\alpha,\beta,\gamma,\delta\in\Z_p$.
Suppose that

\medskip\noindent(i)
$\max\{\langle-\alpha\rangle_p,\langle-\beta\rangle_p,\langle-\gamma\rangle_p\}\leq \langle-\delta\rangle_p$;
\qquad(ii) $\langle-\alpha\rangle_p+\langle-\beta\rangle_p+\langle-\gamma\rangle_p<p+\langle-\delta\rangle_p$;

\medskip\noindent
(iii) $\langle-\beta\rangle_p+\langle-\gamma\rangle_p\geq\langle-\delta\rangle_p$;
\qquad(iv) $(\delta)_{p-1},(1+\delta-\beta-\gamma)_{p-1}\not\equiv0\pmod{p^2}$.

\medskip\noindent
Then
\begin{align}\label{alphabetagamma1delta13F2G}
&{}_3F_2\bigg[\begin{matrix}\alpha&\beta&\gamma\\
&1&\delta\end{matrix}\bigg|\,1\bigg]_{p-1}\notag\\
\equiv&\frac{\Gamma_p(\delta)\Gamma_p(1+\delta-\alpha-\beta-\gamma)}{
\Gamma_p(\delta-\alpha)\Gamma_p(1+\delta-\beta-\gamma)}\cdot{}_3F_2\bigg[\begin{matrix}\alpha&1-\beta&1-\gamma\\
&1&1+\delta-\beta-\gamma\end{matrix}\bigg|\,1\bigg]_{p-1}\pmod{p^2}.
\end{align}
\end{theorem}
For example, for any prime $p\equiv 1\pmod{6}$,
\begin{align*}
{}_3F_2\bigg[\begin{matrix}\frac13&\frac12&\frac23\\
&1&\frac56 \end{matrix}\bigg|\,1\bigg]_{p-1}\equiv\frac{\Gamma_p(\frac56)\Gamma_p(\frac13)}{
\Gamma_p(\frac12)\Gamma_p(\frac23)}\cdot{}_3F_2\bigg[\begin{matrix}\frac13&\frac12&\frac13\\
&1&\frac23\end{matrix}\bigg|\,1\bigg]_{p-1}\pmod{p^2}.
\end{align*}

Whipple also found a quadratic transformation concerning the ${}_3F_2$ series \cite[Eq. (3.1.15)]{AAR99}: 
\begin{align}\label{Whipplequadratictransformation3F2}
&{}_3F_2\bigg[\begin{matrix}
\alpha&\beta&\gamma\\
&\alpha-\beta+1&\alpha-\gamma+1
\end{matrix}\bigg|\,z\bigg]\notag\\
=&(1-z)^{-\alpha}{}_3F_2\bigg[\begin{matrix}
\alpha-\beta-\gamma+1&\frac12\alpha&\frac12\alpha+\frac12\\
&\alpha-\beta+1&\alpha-\gamma+1
\end{matrix}\bigg|\,-\frac{4z}{(1-z)^2}\bigg].
\end{align}
Letting $\gamma=\alpha$ in (\ref{Whipplequadratictransformation3F2}), we have
\begin{equation}\label{alpha2alpha21beta4zz1e}
{}_3F_2\bigg[\begin{matrix}
\alpha&\alpha&\beta\\
&1&\alpha-\beta+1
\end{matrix}\bigg|\,z\bigg]=(1-z)^{-\alpha}{}_3F_2\bigg[\begin{matrix}
1-\beta&\frac12\alpha&\frac12\alpha+\frac12\\
&1&\alpha-\beta+1
\end{matrix}\bigg|\,-\frac{4z}{(1-z)^2}\bigg].
\end{equation}
\begin{theorem}\label{alpha2alpha21beta4zz1}
Let $\alpha,\beta\in\Z_p$ and $z\in\Z_p^\times$ with $z-1$ is prime to $p$.
Suppose that

\medskip\noindent(i)
$\langle-\alpha\rangle_p\leq \langle-\beta\rangle_p$;\ \  (ii) $p+\langle-\alpha\rangle_p>2\langle-\beta\rangle_p$;
\ \ (iii) $(\alpha-\beta+1)_{p-1}\not\equiv0\pmod{p^2}$.

\medskip\noindent
Then
\begin{align}\label{alpha2alpha21beta4zz1G}
&{}_3F_2\bigg[\begin{matrix}
\alpha&\alpha&\beta\\
&1&\alpha-\beta+1
\end{matrix}\bigg|\,z\bigg]_{p-1}+(-z)^{1-\lambda_p(\alpha)}{}_3F_2\bigg[\begin{matrix}
\alpha&\alpha&\beta\\
&1&\alpha-\beta+1
\end{matrix}\bigg|\,\frac{1}{z}\bigg]_{p-1}\notag\\
\equiv&(1-z)^{1-\lambda_p(\alpha)}{}_3F_2\bigg[\begin{matrix}
1-\beta&\frac12\alpha&\frac12\alpha+\frac12\\
&1&\alpha-\beta+1
\end{matrix}\bigg|\,-\frac{4z}{(1-z)^2}\bigg]_{p-1}\pmod{p^2}.
\end{align}
\end{theorem}

Let us turn to the ${}_4F_3$ series. The well-known Whipple formula is a basic transformation between two balanced ${}_4F_3$ series, which says
\begin{align}\label{Whippletransformation4F3}
{}_4F_3\bigg[\begin{matrix}-n&\alpha&\beta&\gamma\\
&\delta&\epsilon&\rho\end{matrix}\bigg|\,1\bigg]=\frac{(\delta-\alpha)_n(\epsilon-\alpha)_n}{
(\delta)_n(\epsilon)_n}
\cdot{}_4F_3\bigg[\begin{matrix}-n&\alpha&\rho-\beta&\rho-\gamma\\
&\rho&1+\alpha-n-\delta&1+\alpha-n-\epsilon\end{matrix}\bigg|\,1\bigg],
\end{align}
where $n\in\N$ and $\alpha+\beta+\gamma-n+1=\delta+\epsilon+\rho$. Setting $\rho=1$ in (\ref{Whippletransformation4F3}), we get
 \begin{align}\label{alphabetagammadeltaepsilon4F3H}
{}_4F_3\bigg[\begin{matrix}-n&\alpha&\beta&\gamma\\
&1&\delta&\epsilon\end{matrix}\bigg|\,1\bigg]=\frac{(\delta-\alpha)_n(\epsilon-\alpha)_n}{
(\delta)_n(\epsilon)_n}\cdot{}_4F_3\bigg[\begin{matrix}-n&\alpha&1-\beta&1-\gamma\\
&1&1+\alpha-n-\delta&1+\alpha-n-\epsilon\end{matrix}\bigg|\,1\bigg],
\end{align}
where $\alpha+\beta+\gamma-n=\delta+\epsilon$. The $p$-adic analogue of (\ref{alphabetagammadeltaepsilon4F3H}) is a little complicated.
Under several assumptions, we have a congruence between two truncated ${}_4F_3$ functions.
\begin{theorem}\label{alphabetagammadeltaepsilon4F3}
Let $\alpha,\beta,\gamma,\delta,\epsilon\in\Z_p$.
Suppose that

\medskip\noindent
(i) $\langle-\alpha\rangle_p\leq\min\{\langle-\delta\rangle_p, \langle-\epsilon\rangle_p\}$,
$\langle-\beta\rangle_p\leq \langle-\delta\rangle_p$ and $\langle-\gamma\rangle_p\leq \langle-\epsilon\rangle_p$;

\medskip\noindent(ii) $\langle-\alpha\rangle_p+\langle-\beta\rangle_p+\langle-\gamma\rangle_p\leq\langle-\delta\rangle_p+\langle-\epsilon\rangle_p$;

\medskip\noindent
(iii) $\langle-\beta\rangle_p+\langle-\gamma\rangle_p\geq \max\{\langle-\delta\rangle_p, \langle-\epsilon\rangle_p$\};

\medskip\noindent
(iv) $(\delta)_{p-1}$, $(\epsilon)_{p-1}$, $(1+\delta-\beta-\gamma)_{p-1}$ and $(1+\epsilon-\beta-\gamma)_{p-1}$ are not divisible by $p^2$.

\medskip\noindent
Let
$\rho=\delta+\epsilon-\alpha-\beta-\gamma.$
Then
\begin{align}\label{alphabetagammadeltaepsilon4F3G}
{}_4F_3\bigg[\begin{matrix}\rho&\alpha&\beta&\gamma\\
&1&\delta&\epsilon\end{matrix}\bigg|\,1\bigg]_{p-1}
\equiv&\frac{\Gamma_p(\beta+\gamma-\delta)\Gamma_p(\beta+\gamma-\epsilon)\Gamma_p(\delta)\Gamma_p(\epsilon)}{
\Gamma_p(\delta-\rho)\Gamma_p(\epsilon-\rho)\Gamma_p(\delta-\alpha)\Gamma_p(\epsilon-\alpha)}\notag\\
&\cdot{}_4F_3\bigg[\begin{matrix}\rho&\alpha&1-\beta&1-\gamma\\
&1&1+\delta-\beta-\gamma&1+\epsilon-\beta-\gamma\end{matrix}\bigg|\,1\bigg]_{p-1}\pmod{p^2}.
\end{align}
\end{theorem}
Although many conditions are involved in Theorem \ref{alphabetagammadeltaepsilon4F3},
we emphasize that all those conditions are necessary, to make both sides of (\ref{alphabetagammadeltaepsilon4F3G}) $p$-integral and not divisible by $p$.
Also, by modifying the conditions in the above theorem, we may get a result on the divisibility of the truncated ${}_4F_3$ function.
\begin{theorem}\label{alphabetagammadeltaepsilon4F3B}
Let $\alpha\in\Z_p^\times$ and $\beta,\gamma,\delta,\epsilon\in\Z_p$. Suppose that

\medskip\noindent
(i) $\langle-\alpha\rangle_p\leq\min\{\langle-\delta\rangle_p, \langle-\epsilon\rangle_p\}$,
$\langle-\beta\rangle_p\leq \langle-\delta\rangle_p$ and $\langle-\gamma\rangle_p\leq \langle-\epsilon\rangle_p$;

\medskip\noindent
(ii) $\langle-\alpha\rangle_p+\langle-\beta\rangle_p+\langle-\gamma\rangle_p\geq p+\max\{\langle-\delta\rangle_p,\langle-\epsilon\rangle_p\}$;

\medskip\noindent
(iii) $(\delta)_{p-1},(\epsilon)_{p-1}\not\equiv0\pmod{p^2}$.

\medskip\noindent
Then
\begin{align}\label{alphabetagammadeltaepsilon4F3GB}
{}_4F_3\bigg[\begin{matrix}\rho&\alpha&\beta&\gamma\\
&1&\delta&\epsilon\end{matrix}\bigg|\,1\bigg]_{p-1}
\equiv0\pmod{p^2},
\end{align}
where $\rho=\delta+\epsilon-\alpha-\beta-\gamma$.
\end{theorem}
We need to give an explanation on the relation between Theorems \ref{nalphabeta1alphabetan}, 
\ref{alphabetagamma1delta13F2}, \ref{alphabetagammadeltaepsilon4F3} and
\ref{alphabetagammadeltaepsilon4F3B}. Evidently (\ref{nalphabeta1alphabetanG}) and (\ref{nalphabeta1alphabetanGB}) can be easily deduced from (\ref{alphabetagammadeltaepsilon4F3G}) and (\ref{alphabetagammadeltaepsilon4F3GB}) respectively, by substituting $\gamma=\epsilon=1$ in Theorems 
\ref{alphabetagammadeltaepsilon4F3} and
\ref{alphabetagammadeltaepsilon4F3B}. 
On the other hand, notice that (\ref{alphabetagamma1delta13F2H}) is also an  consequence of (\ref{alphabetagammadeltaepsilon4F3H}), if $\alpha$, $\beta$, $\gamma$,
$\delta$ and $\epsilon$ are all fixed and $n$ tends to $+\infty$.
However, clearly such a operation is invalid for the $p$-adic analogues. Seemingly
(\ref{alphabetagamma1delta13F2}) can't be deduced from (\ref{alphabetagammadeltaepsilon4F3G}) directly.

Another basic formula on the hypergeometric series is Whipple's ${}_7F_6$ transformation \cite[Theorem 3.4.5]{AAR99}:
\begin{align}\label{Whippletransformation7F6}
&{}_7F_6\bigg[\begin{matrix}\alpha&1+\frac12\alpha&\beta&\gamma&\delta&\epsilon&\rho\\
&\frac12\alpha&\alpha-\beta+1&\alpha-\gamma+1&\alpha-\delta+1&\alpha-\epsilon+1&\alpha-\rho+1\end{matrix}\bigg|\,1\bigg]\notag\\
=&\frac{\Gamma(\alpha-\beta+1)\Gamma(\alpha-\gamma+1)\Gamma(\alpha-\delta+1)\Gamma(\alpha-\beta-\gamma-\delta+1)}{\Gamma(\alpha+1)\Gamma(\alpha-\beta-\gamma+1)\Gamma(\alpha-\beta-\delta+1)\Gamma(\alpha-\gamma-\delta+1)}\notag\\
&\cdot{}_4F_3\bigg[\begin{matrix} \alpha-\epsilon-\rho+1&\beta&\gamma&\delta\\ &\beta+\gamma+\delta-\alpha&\alpha-\epsilon+1&\alpha-\rho+1\end{matrix}\bigg|\,1\bigg].
\end{align}
Setting $\rho=\alpha$ in (\ref{Whippletransformation7F6}), we have
\begin{align}\label{alphabetagammadeltaepsilon7F6H}
 &{}_7F_6\bigg[\begin{matrix} \alpha&\alpha&\frac12\alpha+1&\beta&\gamma&\delta&\epsilon\\ &1&\frac12\alpha&\alpha-\beta+1&\alpha-\gamma+1&\alpha-\delta+1&\alpha-\epsilon+1\end{matrix}\bigg|\,1\bigg]\notag\\
 =
&\frac{\Gamma(\alpha-\beta+1)\Gamma(\alpha-\gamma+1)\Gamma(\alpha-\delta+1)\Gamma(\alpha-\beta-\gamma-\delta+1)}{\Gamma(\alpha+1)\Gamma(\alpha-\beta-\gamma+1)\Gamma(\alpha-\beta-\delta+1)\Gamma(\alpha-\gamma-\delta+1)}\notag\\
&\cdot{}_4F_3\bigg[\begin{matrix} 1-\epsilon&\beta&\gamma&\delta\\ &1&\alpha-\epsilon+1&\beta+\gamma+\delta-\alpha\end{matrix}\bigg|\,1\bigg].
 \end{align}
Of course, the $p$-adic analogues of (\ref{alphabetagammadeltaepsilon7F6H}) become more complicated.
\begin{theorem}\label{alphabetagammadeltaepsilon7F6A}
Let $\alpha,\beta,\gamma,\delta,\epsilon\in\Z_p$.
Suppose that

\medskip\noindent
(i) $\langle-\beta\rangle_p<\langle-\alpha\rangle_p\leq\min\{\langle-\gamma\rangle_p, \langle-\delta\rangle_p, \langle-\epsilon\rangle_p\}$;

\medskip\noindent
(ii) $p+\langle-\alpha\rangle_p>\max\{\langle-\gamma\rangle_p+\langle-\epsilon\rangle_p,\langle-\delta\rangle_p+\langle-\epsilon\rangle_p, \langle-\beta\rangle_p+\langle-\gamma\rangle_p+\langle-\delta\rangle_p\}$;

\medskip\noindent
(iii) $\langle-\alpha\rangle_p/\alpha=\langle-\beta\rangle_p/\beta$.

\medskip\noindent
(iv) $(\alpha-\beta+1)_{p-1}$, $(\alpha-\gamma+1)_{p-1}$, $(\alpha-\delta+1)_{p-1}$, $(\alpha-\epsilon+1)_{p-1}$, $(\beta+\gamma+\delta-\alpha)_{p-1}$ are not divisible by $p^2$.

\medskip\noindent
Then
\begin{align}\label{alphabetagammadeltaepsilon7F6AG}
&{}_7F_6\bigg[\begin{matrix}\alpha&\alpha&1+\frac12\alpha&\beta&\gamma&\delta&\epsilon\\
&1&\frac12\alpha&\alpha-\beta+1&\alpha-\gamma+1&\alpha-\delta+1&\alpha-\epsilon+1\end{matrix}\bigg|\,1\bigg]_{p-1}\notag\\
\equiv&\frac{\langle-\alpha\rangle_p}{\langle-\alpha\rangle_p-\langle-\beta\rangle_p}\cdot\frac{\Gamma_p(\alpha-\beta+1)\Gamma_p(\alpha-\gamma+1)\Gamma_p(\alpha-\delta+1)\Gamma_p(\alpha-\beta-\gamma-\delta+1)}{\Gamma_p(\alpha+1)\Gamma_p(\alpha-\beta-\gamma+1)\Gamma_p(\alpha-\beta-\delta+1)\Gamma_p(\alpha-\gamma-\delta+1)}\notag\\
&\cdot{}_4F_3\bigg[\begin{matrix} 1-\epsilon&\beta&\gamma&\delta\\ &1&\alpha-\epsilon+1&\beta+\gamma+\delta-\alpha\end{matrix}\bigg|\,1\bigg]_{p-1}\pmod{p^2}.
\end{align}
\end{theorem}
(\ref{alphabetagammadeltaepsilon7F6AG}) is actually a very curious congruence. In fact, under the assumptions of Theorem \ref{alphabetagammadeltaepsilon7F6A}, if $\langle-\alpha\rangle_p\leq 2\langle-\beta\rangle_p$, then
we may have
$$
\frac{(\alpha)_k^2(1+\frac12\alpha)_k(\beta)_k(\gamma)_k(\delta)_k(\epsilon)_k}{(\frac12\alpha)_k(\alpha-\beta+1)_k(\alpha-\gamma+1)_k(\alpha-\delta+1)_k(\alpha-\epsilon+1)_k}\not\in\Z_p
$$
for any $\langle-\alpha\rangle_p-\langle-\beta\rangle_p\leq k\leq \langle-\beta\rangle_p$. Fortunately, even if $\langle-\alpha\rangle_p\leq 2\langle-\beta\rangle_p$, the truncated ${}_7F_6$ function in (\ref{alphabetagammadeltaepsilon7F6AG}) is still $p$-integral and (\ref{alphabetagammadeltaepsilon7F6AG}) is also valid.

We also give an explanation on the condition $\langle-\alpha\rangle_p/\alpha=\langle-\beta\rangle_p/\beta$. Assume that $\beta=r/d$ and $\alpha=sr/d$ where $(r,d)=1$ and $1\leq r<d/s$. If $1\leq t<d/s$ and $p$ is a prime with $tp\equiv r\pmod{d}$, then clearly $\langle-\beta\rangle_p=(tp-r)/d$ and $\langle-\alpha\rangle_p=(stp-sr)/d$. Then in such a case, we have $\langle-\alpha\rangle_p/\langle-\beta\rangle_p=\alpha/\beta=s$. For example, for any $p\equiv1\pmod{12}$,
\begin{align*}
&{}_7F_6\bigg[\begin{matrix}\frac23&\frac23&\frac43&\frac1{12}&\frac34&\frac34&\frac34\\
&1&\frac13&\frac{19}{12}&\frac{11}{12}&\frac{11}{12}&\frac{11}{12}\end{matrix}\bigg|\,1\bigg]_{p-1}\\
\equiv&\frac{8}{7}\cdot\frac{\Gamma_p(\frac32)\Gamma_p(\frac{11}{12})^2\Gamma_p(\frac1{12})}{\Gamma_p(\frac53)\Gamma_p(\frac56)^2\Gamma_p(\frac16)}\cdot{}_4F_3\bigg[\begin{matrix}\frac14&\frac1{12}&\frac34&\frac34\\ &1&\frac{11}{12}&\frac{11}{12}\end{matrix}\bigg|\,1\bigg]_{p-1}\pmod{p^2}.
\end{align*}

On the other hand, the truncated ${}_7F_6$ function can be divisible by $p^2$ under another assumptions.
\begin{theorem}\label{alphabetagammadeltaepsilon7F6B}
Let $\alpha,\beta,\gamma,\delta,\epsilon\in\Z_p$.
Suppose that

\medskip\noindent
(i) $\langle-\alpha\rangle_p\leq\min\{\langle-\beta\rangle_p, \langle-\gamma\rangle_p,\langle-\delta\rangle_p, \langle-\epsilon\rangle_p\}$;

\medskip\noindent
(ii) $p+\langle-\alpha\rangle_p>\max\{\langle-\beta\rangle_p+\langle-\gamma\rangle_p,\langle-\beta\rangle_p+\langle-\delta\rangle_p,\langle-\beta\rangle_p+\langle-\epsilon\rangle_p,\langle-\gamma\rangle_p+\langle-\delta\rangle_p\}$;

\medskip\noindent
(iii)
$2p-1+\langle-\alpha\rangle_p\leq \langle-\beta\rangle_p+\langle-\gamma\rangle_p+\langle-\delta\rangle_p+\langle-\epsilon\rangle_p$;

\medskip\noindent
(iv) $(\alpha-\beta+1)_{p-1},(\alpha-\gamma+1)_{p-1},(\alpha-\delta+1)_{p-1},(\alpha-\epsilon+1)_{p-1}\not\equiv0\pmod{p^2}.$

\medskip\noindent
Then
\begin{align}\label{alphabetagammadeltaepsilon7F6BG}
{}_7F_6\bigg[\begin{matrix}\alpha&\alpha&1+\frac12\alpha&\beta&\gamma&\delta&\epsilon\\
&1&\frac12\alpha&\alpha-\beta+1&\alpha-\gamma+1&\alpha-\delta+1&\alpha-\epsilon+1\end{matrix}\bigg|\,1\bigg]_{p-1}
\equiv0\pmod{p^2}.
\end{align}
\end{theorem}

Furthermore, we have a different $p$-analogue of (\ref{alphabetagammadeltaepsilon7F6H}) provided that the truncated ${}_7F_6$ function is divisible by $p$, but not by $p^2$.
\begin{theorem}\label{alphabetagammadeltaepsilon7F6C} Let $\alpha,\beta,\gamma,\delta,\epsilon\in\Z_p$. Suppose that

\medskip\noindent (i) $\langle -\alpha\rangle_p\leq\min\{\langle -\beta\rangle_p,\langle -\gamma\rangle_p,\langle -\delta\rangle_p,\langle -\epsilon\rangle_p\}$;

\medskip\noindent (ii) $p+\langle-\alpha\rangle_p>\max\{\langle-\beta\rangle_p+\langle-\gamma\rangle_p+\langle -\delta\rangle_p,\ \min\{\langle-\beta\rangle_p,\langle-\gamma\rangle_p,\langle -\delta\rangle_p\}+\langle-\epsilon\rangle_p\}$;

\medskip\noindent (iii) $(\alpha-\beta+1)_{p-1}$, $(\alpha-\gamma+1)_{p-1}$, $(\alpha-\delta+1)_{p-1}$, $(\alpha-\epsilon+1)_{p-1}$ and $(\beta+\gamma+\delta-\alpha)_{p-1}$ are not divisible by $p^2$.

\medskip\noindent  Then
\begin{align}\label{alphabetagammadeltaepsilon7F6CG}
&{}_7F_6\bigg[\begin{matrix} \alpha&\alpha&\frac12\alpha+1&\beta&\gamma&\delta&\epsilon\\ &1&\frac12\alpha&\alpha-\beta+1&\alpha-\gamma+1&\alpha-\delta+1&\alpha-\epsilon+1\end{matrix}\bigg|\,1\bigg]_{p-1}\notag\\
\equiv&\big(\alpha+\langle-\alpha\rangle_p\big)\cdot\frac{\Gamma_p(\alpha-\beta+1)\Gamma_p(\alpha-\gamma+1)\Gamma_p(\alpha-\delta+1)\Gamma_p(\alpha-\beta-\gamma-\delta+1)}{\Gamma_p(\alpha+1)\Gamma_p(\alpha-\beta-\gamma+1)\Gamma_p(\alpha-\beta-\delta+1)\Gamma_p(\alpha-\gamma-\delta+1)}\notag\\
&\cdot{}_4F_3\bigg[\begin{matrix} 1-\epsilon&\beta&\gamma&\delta\\ &1&\alpha-\epsilon+1&\beta+\gamma+\delta-\alpha\end{matrix}\bigg|\,1\bigg]_{p-1}\pmod{p^2}.
\end{align}
\end{theorem}

The importance of  Whipple's ${}_7F_6$ transformation 
is that (\ref{Whippletransformation7F6}) has many interesting consequences. For example, combining (\ref{Whippletransformation7F6}) with (\ref{pffafsaalschutz}), we get the Dougall's formula
\begin{align}
&{}_7F_6\bigg[\begin{matrix}\alpha&1+\frac12\alpha&\beta&\gamma&\delta&\epsilon&-n\\
&\frac12\alpha&\alpha-\beta+1&\alpha-\gamma+1&\alpha-\delta+1&\alpha-\epsilon+1&\alpha+n+1\end{matrix}\bigg|\,1\bigg]\notag\\
=&\frac{(\alpha+1)_n(\alpha-\beta-\gamma+1)_n(\alpha-\beta-\delta+1)_n(\alpha-\gamma-\delta+1)_n}{(\alpha-\beta+1)_n(\alpha-\gamma+1)_n\Gamma(\alpha-\delta+1)_n(\alpha-\beta-\gamma-\delta+1)_n},
\end{align}
where $n\in\N$ and $n=\beta+\gamma+\delta+\epsilon-2\alpha-1$. In particular,
\begin{align}\label{alphabetagammadeltan7F6H}
&{}_7F_6\bigg[\begin{matrix}\alpha&\alpha&1+\frac12\alpha&\beta&\gamma&\delta&-n\\
&1&\frac12\alpha&\alpha-\beta+1&\alpha-\gamma+1&\alpha-\delta+1&\alpha+n+1\end{matrix}\bigg|\,1\bigg]\notag\\
=&\frac{(\alpha+1)_n(\alpha-\beta-\gamma+1)_n(\alpha-\beta-\delta+1)_n(\alpha-\gamma-\delta+1)_n}{(\alpha-\beta+1)_n(\alpha-\gamma+1)_n\Gamma(\alpha-\delta+1)_n(\alpha-\beta-\gamma-\delta+1)_n},
\end{align}
where $n=\beta+\gamma+\delta-\alpha-1$. We have two $p$-adic analogues of (\ref{alphabetagammadeltan7F6H}).
\begin{theorem}\label{alphabetagammadeltan7F6A}
Let $\alpha,\beta,\gamma,\delta\in\Z_p$. Suppose that

\medskip\noindent
(i) $\langle-\beta\rangle_p<\langle-\alpha\rangle_p\leq\min\{\langle-\gamma\rangle_p, \langle-\delta\rangle_p\}$;

\medskip\noindent
(ii) $\langle-\beta\rangle_p+\langle-\gamma\rangle_p+\langle-\delta\rangle_p<p$;

\medskip\noindent
(iii) $\langle-\alpha\rangle_p/\alpha=\langle-\beta\rangle_p/\beta$;

\medskip\noindent
(iv) $(\alpha-\beta+1)_{p-1}$, $(\alpha-\gamma+1)_{p-1}$, $(\alpha-\delta+1)_{p-1}$ and $(\beta+\gamma+\delta)_{p-1}$ are not divisible by $p^2$.

\medskip\noindent
Let $\epsilon=\alpha+1-\beta-\gamma-\delta$. Then
\begin{align}
&{}_7F_6\bigg[\begin{matrix}\alpha&\alpha&1+\frac12\alpha&\beta&\gamma&\delta&\epsilon\\
&1&\frac12\alpha&\alpha-\beta+1&\alpha-\gamma+1&\alpha-\delta+1&\alpha-\epsilon+1\end{matrix}\bigg|\,1\bigg]_{p-1}\notag\\
\equiv&-\frac{\alpha}{\alpha-\beta}\cdot\frac{\Gamma_p(1-\beta-\gamma)\Gamma_p(1-\beta-\delta)\Gamma_p(1-\gamma-\delta)}{\Gamma_p(1-\beta)\Gamma_p(1-\gamma)\Gamma_p(1-\delta)\Gamma_p(1-\beta-\gamma-\delta)}\notag\\
&\cdot\frac{\Gamma_p(\alpha-\beta+1)\Gamma_p(\alpha-\gamma+1)\Gamma_p(\alpha-\delta+1)\Gamma_p(\alpha-\beta-\gamma-\delta+1)}{\Gamma_p(\alpha+1)\Gamma_p(\alpha-\beta-\gamma+1)\Gamma_p(\alpha-\beta-\delta+1)\Gamma_p(\alpha-\gamma-\delta+1)}\pmod{p^2}.
\end{align}
\end{theorem}
\begin{theorem} \label{alphabetagammadeltan7F6B}
Let $\alpha,\beta,\gamma,\delta\in\Z_p$. 
Suppose that

\medskip\noindent (i) $\langle -\alpha\rangle_p\leq\min\{\langle -\beta\rangle_p,\langle -\gamma\rangle_p,\langle -\delta\rangle_p,\langle -\epsilon\rangle_p\}$;

\medskip\noindent (ii) $\langle-\beta\rangle_p+\langle-\gamma\rangle_p+\langle-\delta\rangle_p<p$;

\medskip\noindent (iii) $(\alpha-\beta+1)_{p-1}$, $(\alpha-\gamma+1)_{p-1}$, $(\alpha-\delta+1)_{p-1}$ and $(\beta+\gamma+\delta)_{p-1}$ are not divisible by $p^2$.

\medskip\noindent Let $\epsilon=\alpha+1-\beta-\gamma-\delta$.  Then
\begin{align}
&{}_7F_6\bigg[\begin{matrix} \alpha&\alpha&\frac12\alpha+1&\beta&\gamma&\delta&\epsilon\\ &1&\frac12\alpha&\alpha-\beta+1&\alpha-\gamma+1&\alpha-\delta+1&\alpha-\epsilon+1\end{matrix}\bigg|\,1\bigg]_{p-1}\notag\\
\equiv&-\big(\alpha+\langle-\alpha\rangle_p\big)\cdot\frac{\Gamma_p(1-\beta-\gamma)\Gamma_p(1-\beta-\delta)\Gamma_p(1-\gamma-\delta)}{\Gamma_p(1-\beta)\Gamma_p(1-\gamma)\Gamma_p(1-\delta)\Gamma_p(1-\beta-\gamma-\delta)}\notag\\
&\cdot\frac{\Gamma_p(\alpha-\beta+1)\Gamma_p(\alpha-\gamma+1)\Gamma_p(\alpha-\delta+1)\Gamma_p(\alpha-\beta-\gamma-\delta+1)}{\Gamma_p(\alpha+1)\Gamma_p(\alpha-\beta-\gamma+1)\Gamma_p(\alpha-\beta-\delta+1)\Gamma_p(\alpha-\gamma-\delta+1)}\pmod{p^2}.
\end{align}
\end{theorem}
Furthermore, letting $n\to+\infty$ in (\ref{alphabetagammadeltan7F6H}), we may get
\begin{align}\label{alphaalpha12alphabetagamma5F4H}
{}_5F_4\bigg[\begin{matrix}\alpha&\alpha&1+\frac12\alpha&\beta&\gamma\\
&1&\frac12\alpha&\alpha-\beta+1&\alpha-\gamma+1\end{matrix}\bigg|\,1\bigg]
=\frac{\Gamma(1+\alpha-\beta)\Gamma(1+\alpha-\gamma)\Gamma(1-\beta-\gamma)}{\Gamma(1+\alpha)\Gamma(1-\beta)\Gamma(1-\gamma)\Gamma(1+\alpha-\beta-\gamma)}.
\end{align}
(\ref{alphaalpha12alphabetagamma5F4H}) has the following $p$-adic analogues.
\begin{theorem}\label{alphaalpha12alphabetagamma5F4A}
Let $\alpha,\beta,\gamma\in\Z_p$.
Suppose that

\medskip\noindent
(i) $\langle-\beta\rangle_p<\langle-\alpha\rangle_p\leq\langle-\gamma\rangle_p<(p+\langle-\alpha\rangle_p)/2$;

\medskip\noindent
(ii) $\langle-\beta\rangle_p+\langle-\gamma\rangle_p<p$;

\medskip\noindent
(iii) $\langle-\alpha\rangle_p/\alpha=\langle-\beta\rangle_p/\beta$;

\medskip\noindent
(iv) $(\alpha-\beta+1)_{p-1}$, $(\alpha-\gamma+1)_{p-1}$ are not divisible by $p^2$.

\medskip\noindent
Then
\begin{align}\label{alphaalpha12alphabetagamma5F4GA}
&{}_5F_4\bigg[\begin{matrix}\alpha&\alpha&1+\frac12\alpha&\beta&\gamma\\
&1&\frac12\alpha&\alpha-\beta+1&\alpha-\gamma+1\end{matrix}\bigg|\,1\bigg]_{p-1}\notag\\
\equiv&-\frac{\alpha}{\alpha-\beta}\cdot\frac{\Gamma_p(1+\alpha-\beta)\Gamma_p(1+\alpha-\gamma)\Gamma_p(1-\beta-\gamma)}{\Gamma_p(1+\alpha)\Gamma_p(1-\beta)\Gamma_p(1-\gamma)\Gamma_p(1+\alpha-\beta-\gamma)}\pmod{p^2}.
\end{align}
\end{theorem}
\begin{theorem}\label{alphaalpha12alphabetagamma5F4B}
Let $\alpha,\beta,\gamma\in\Z_p$.
Suppose that

\medskip\noindent
(i) $\langle-\alpha\rangle_p\leq\min\{\langle-\beta\rangle_p,\langle-\gamma\rangle_p\}$;\qquad(ii) $(\alpha-\beta+1)_{p-1}$, $(\alpha-\gamma+1)_{p-1}$ are not divisible by $p^2$.

\medskip\noindent
If $p\leq \langle-\beta\rangle_p+\langle-\gamma\rangle_p<p+\langle-\alpha\rangle_p$,
then
\begin{align}\label{alphaalpha12alphabetagamma5F4GB}
&{}_5F_4\bigg[\begin{matrix}\alpha&\alpha&1+\frac12\alpha&\beta&\gamma\\
&1&\frac12\alpha&\alpha-\beta+1&\alpha-\gamma+1\end{matrix}\bigg|\,1\bigg]_{p-1}
\equiv0\pmod{p^2}.
\end{align}
And if $\langle-\beta\rangle_p+\langle-\gamma\rangle_p<p$, then
\begin{align}\label{alphaalpha12alphabetagamma5F4GC}
&{}_5F_4\bigg[\begin{matrix}\alpha&\alpha&1+\frac12\alpha&\beta&\gamma\\
&1&\frac12\alpha&\alpha-\beta+1&\alpha-\gamma+1\end{matrix}\bigg|\,1\bigg]_{p-1}\notag\\
\equiv&-(\alpha+\langle-\alpha\rangle_p)\cdot\frac{\Gamma_p(1+\alpha-\beta)\Gamma_p(1+\alpha-\gamma)\Gamma_p(1-\beta-\gamma)}{\Gamma_p(1+\alpha)\Gamma_p(1-\beta)\Gamma_p(1-\gamma)\Gamma_p(1+\alpha-\beta-\gamma)}\pmod{p^2}.
\end{align}
\end{theorem}
Notice that Theorem \ref{alphaalphabetaalphabeta1} is factly the consequence of Theorems \ref{alphaalpha12alphabetagamma5F4A} and \ref{alphaalpha12alphabetagamma5F4B}, by substituting $\beta=\alpha/2$. 
Theorems \ref{alphaalpha12alphabetagamma5F4A} and \ref{alphaalpha12alphabetagamma5F4B} also can be deduced from Theorems \ref{alphabetagammadeltaepsilon7F6A}, \ref{alphabetagammadeltaepsilon7F6B} and \ref{alphabetagammadeltaepsilon7F6C} by setting $\epsilon=1+\alpha-\delta$. However, this may cause to some unnecessary additional requirements concerning $\beta$ and $\gamma$.

Let us see another application of (\ref{alphabetagammadeltaepsilon7F6H}). Letting $\epsilon$ tend to $-\infty$ in (\ref{alphabetagammadeltaepsilon7F6H}), we may get
\begin{align}\label{alphabetagammadelta6F5H}
&{}_6F_5\bigg[\begin{matrix} \alpha&\frac12\alpha+1&\alpha&\beta&\gamma&\delta\\ &\frac12\alpha&1&\alpha-\beta+1&\alpha-\gamma+1&\alpha-\delta+1\end{matrix}\bigg|\,-1\bigg]\notag\\
=&\frac{\Gamma(\alpha-\gamma+1)\Gamma(\alpha-\delta+1)}{\Gamma(\alpha+1)\Gamma(\alpha-\gamma-\delta+1)}
\cdot{}_3F_2\bigg[\begin{matrix} 1-\beta&\gamma&\delta\\
&1&\alpha-\beta+1\end{matrix}\bigg|\,1\bigg]_{p}.
\end{align}
\begin{theorem}\label{alphabetagammadelta6F5A} Let $\alpha,\beta,\gamma,\delta\in\Z_p$. Suppose that

\medskip\noindent
(i) $\langle-\beta\rangle_p<\langle-\alpha\rangle_p\leq\min\{\langle-\gamma\rangle_p,\langle-\delta\rangle_p\}$;

\medskip\noindent
(ii) $p+\langle-\alpha\rangle_p>\langle-\gamma\rangle_p+\langle-\delta\rangle_p$;

\medskip\noindent
(iii) $\langle-\alpha\rangle_p/\alpha=\langle-\beta\rangle_p/\beta$;

\medskip\noindent
(iv) $(\alpha-\beta+1)_{p-1}$, $(\alpha-\gamma+1)_{p-1}$, $(\alpha-\delta+1)_{p-1}$ are not divisible by $p^2$.

\medskip\noindent
Then
\begin{align}\label{alphabetagammadelta6F5AG}
&{}_6F_5\bigg[\begin{matrix} \alpha&\alpha&\frac12\alpha+1&\beta&\gamma&\delta\\ &1&\frac12\alpha&\alpha-\beta+1&\alpha-\gamma+1&\alpha-\delta+1\end{matrix}\bigg|\,-1\bigg]_{p-1}\notag\\
\equiv&\frac{\alpha}{\alpha-\beta}\cdot\frac{\Gamma_p(\alpha-\beta+1)\Gamma_p(\alpha-\gamma+1)}{\Gamma_p(\alpha+1)\Gamma_p(\alpha-\beta-\gamma+1)}
\cdot{}_3F_2\bigg[\begin{matrix} 1-\delta&\beta&\gamma\\
&1&\alpha-\delta+1\end{matrix}\bigg|\,1\bigg]_{p-1}\pmod{p^2}.
\end{align}
\end{theorem}
\begin{theorem}\label{alphabetagammadelta6F5B} Let $\alpha,\beta,\gamma,\delta\in\Z_p$. Suppose that

\medskip\noindent
(i) $\langle-\alpha\rangle_p\leq\min\{\langle-\beta\rangle_p,\langle-\gamma\rangle_p,\langle-\delta\rangle_p\}$;

\medskip\noindent
(ii) $p+\langle-\alpha\rangle_p>\max\{\langle-\beta\rangle_p+\langle-\gamma\rangle_p, \langle-\beta\rangle_p+\langle -\delta\rangle_p, \langle-\gamma\rangle_p+\langle -\delta\rangle_p\}$;

\medskip\noindent
(iii) $(\alpha-\beta+1)_{p-1}$, $(\alpha-\gamma+1)_{p-1}$ $(\alpha-\delta+1)_{p-1}$ are not divisible by $p^2$.

\medskip\noindent
Then
\begin{align}\label{alphabetagammadelta6F5BG}
&{}_6F_5\bigg[\begin{matrix} \alpha&\alpha&\frac12\alpha+1&\beta&\gamma&\delta\\ &1&\frac12\alpha&\alpha-\beta+1&\alpha-\gamma+1&\alpha-\delta+1\end{matrix}\bigg|-1\bigg]_{p-1}\notag\\
\equiv&
\big(\alpha+\langle-\alpha\rangle_p\big)\cdot
\frac{\Gamma_p(\alpha-\beta+1)\Gamma_p(\alpha-\gamma+1)}{\Gamma_p(\alpha+1)\Gamma_p(\alpha-\beta-\gamma+1)}\cdot
{}_3F_2\bigg[\begin{matrix} 1-\delta&\beta&\gamma\\
&1&\alpha-\delta+1\end{matrix}\bigg|\,1\bigg]_{p-1}\pmod{p^2}.
\end{align}
\end{theorem}
We can set $\delta=1+\alpha-\gamma$ in (\ref{alphabetagammadelta6F5AG}) and (\ref{alphabetagammadelta6F5BG}) now. Choose $\delta\in\Z_p$ such that neither $(\gamma)_{p-1}$ nor $(\alpha-\gamma+1)_{p-1}$ is divisible by $p^2$, and
$$
\langle\gamma\rangle_p=\begin{cases}\frac12(p+\langle-\alpha\rangle_p-1),&\text{if }\langle-\alpha\rangle_p\text{ is even},\\
\frac12(p+\langle-\alpha\rangle_p)-1,&\text{if }\langle-\alpha\rangle_p\text{ is odd}.
\end{cases}
$$
Then substituting $\delta=1+\alpha-\gamma$ in (\ref{alphabetagammadelta6F5AG}) and (\ref{alphabetagammadelta6F5BG}) and applying Theorem \ref{alphabeta11}, we obtain that
\begin{corollary}\label{} Let $\alpha,\beta\in\Z_p$.
(1) Suppose that
$\langle-\beta\rangle_p<\langle-\alpha\rangle_p$, $\langle-\alpha\rangle_p/\alpha=\langle-\beta\rangle_p/\beta$ and
$p^2$ doesn't divide $(\alpha-\beta+1)_{p-1}$.
Then
\begin{align}\label{4F3alphaalpha12alpha1betam1}
{}_4F_3\bigg[\begin{matrix} \alpha&\alpha&\frac12\alpha+1&\beta\\ &1&\frac12\alpha&\alpha-\beta+1\end{matrix}\bigg|-1\bigg]_{p-1}
\equiv
-\frac{\alpha}{\alpha-\beta}\cdot
\frac{\Gamma_p(\alpha-\beta+1)}{\Gamma_p(\alpha+1)\Gamma_p(1-\beta)}\pmod{p^2}.
\end{align}

\medskip\noindent
(2) Suppose that
$\langle-\alpha\rangle_p\leq \langle-\beta\rangle_p\leq (p+\langle-\alpha\rangle_p-1)/2$ and
$(\alpha-\beta+1)_{p-1}\not\equiv0\pmod{p^2}$.
Then
\begin{align}\label{}
{}_4F_3\bigg[\begin{matrix} \alpha&\alpha&\frac12\alpha+1&\beta\\ &1&\frac12\alpha&\alpha-\beta+1\end{matrix}\bigg|-1\bigg]_{p-1}
\equiv
-\big(\alpha+\langle-\alpha\rangle_p\big)\cdot
\frac{\Gamma_p(\alpha-\beta+1)}{\Gamma_p(\alpha+1)\Gamma_p(1-\beta)}\pmod{p^2}.
\end{align}
\end{corollary}
Notice that $\langle-\alpha/2\rangle=\langle-\alpha\rangle/2$ if and only if $\langle-\alpha\rangle$ is even. So substituting $\beta=\alpha/2$ in (\ref{4F3alphaalpha12alpha1betam1}), we obtain that
\begin{corollary}\label{2F1alphaalpham1}
Suppose that $\alpha\in\Z_p$ and $\langle-\alpha\rangle_p$ is even.
\begin{equation}\label{alphaalpham12F1}
{}_2F_1\bigg[\begin{matrix}
\alpha&\alpha\\
&1
\end{matrix}\bigg|\,-1\bigg]_{p-1}\equiv-\frac{2\Gamma_p(1+\frac12\alpha)}{\Gamma_p(1+\alpha)\Gamma_p(1-\frac12\alpha)}\pmod{p^2}.
\end{equation}
\end{corollary}
We mention that the special case $\alpha=1/2$  of (\ref{alphaalpham12F1}) was firstly proved by Coster and van Hamme \cite{CoHa91}, and an alternative proof was given in \cite{DFLST16}.

All above $p$-adic transformations will be proved in a unified way. In fact, with help of the same idea and an identity of Kummer, we can prove the following result, which confirms a conjecture of Deines, Fuselier, Long, Swisher and Tu \cite{DFLST16}.
\begin{theorem}\label{F2112121F321113232} Let $p\equiv1\pmod 4$ be a prime. Then
\begin{equation}\label{F2112121F321113232p2}
{}_2F_1\bigg[\begin{matrix} \frac12&\frac12\\ &1\end{matrix}\bigg|\,-1\bigg]_{p-1}
\equiv p^2\cdot
{}_3F_2\bigg[\begin{matrix} 1&1&1\\ &\frac32&\frac32\end{matrix}\bigg|\,-1\bigg]_{p-1}\pmod{p^2}.
\end{equation}
\end{theorem}
The remainder sections will be organized as follows. In Section \ref{sectiontwolemmas}, we shall firstly prove two key lemmas, which are frequently used throughout the whole paper. Then in Section \ref{sectionGauss}, we shall give a generalization of Theorem \ref{alphabeta11}, which is also a complete $p$-adic analogue of Gauss' identity (\ref{Gaussidentity}). 

From Sections \ref{sectionq2F1z4z1z} to Section \ref{sectionql2F1}, the proofs of those linear and quadratic $p$-adic transformations of ${}_2F_1$ series will be given successively. In particular, we just sketch the proof of Theorem \ref{alpha2alpha21beta4zz1} in Section \ref{sectionq2F1z4z1z2}, since its proof is very similar as Theorem \ref{alpha2alpha214zz1}. Furthermore, in Section \ref{section2F1cm}, we shall discuss the relation between some elliptic curves having complex multiplication and ${}_2F_1\bigg[\begin{matrix} \frac14&\frac14\\ &1\end{matrix}\bigg|\,z\bigg]_{p-1}$, ${}_2F_1\bigg[\begin{matrix} \frac14&\frac34\\ &1\end{matrix}\bigg|\,z\bigg]_{p-1}$ modulo $p^2$. 

In Section \ref{section3F24F3},  we shall prove the $p$-adic analogues of those ${}_3F_2$ and ${}_4F_3$ transformations. Of course, as we have mentioned, Theorem \ref{nalphabeta1alphabetan} easily follows of Theorems \ref{alphabetagammadeltaepsilon4F3} and
\ref{alphabetagammadeltaepsilon4F3B}, and Theorem \ref{alphaalphabetaalphabeta1} easily follows from Theorems \ref{alphaalpha12alphabetagamma5F4A} and \ref{alphaalpha12alphabetagamma5F4B}. So the proofs of those two Theorems would not be given.

In Sections \ref{section7F6I} and \ref{section7F6II}, we shall complete the proofs of those $p$-adic Whipple's ${}_7F_6$ transformation, and explain how to deduce Theorems \ref{alphabetagammadeltan7F6A} and \ref{alphabetagammadeltan7F6B} from Theorems \ref{alphabetagammadeltaepsilon7F6A} and \ref{alphabetagammadeltaepsilon7F6C}. However, Theorems \ref{alphaalpha12alphabetagamma5F4A}-\ref{alphabetagammadelta6F5B} factly can be proved in a very similar way as Theorems \ref{alphabetagammadeltaepsilon7F6A}-\ref{alphabetagammadeltaepsilon7F6C}. So we won't give their detailed proofs here. Finally, in the last section, Theorem \ref{F2112121F321113232}, i.e., the conjecture of Deines, Fuselier, Long, Swisher and Tu, will be proved.

\section{Two auxiliary lemmas}
\label{sectiontwolemmas}
\setcounter{lemma}{0}
\setcounter{theorem}{0}
\setcounter{corollary}{0}
\setcounter{remark}{0}
\setcounter{equation}{0}
\setcounter{conjecture}{0}

In this section, we introduce two auxiliary lemmas, which are the key ingredients of most proofs in this paper.
The first one is nearly trivial.
\begin{lemma}\label{taylorexpansionrational}
Suppose that $P(x),Q(x)\in\Z_p[x]$ are polynomials over $\Z_p$ and let $f(x)=P(x)/Q(x)$. 
Suppose that $\alpha\in\Z_p$ and $p\nmid Q(\alpha)$. Then for any $s\in\Z_p$, we have the $p$-adic expansion
$$
f(\alpha+sp)=f(\alpha)+f'(\alpha)\cdot sp+\frac{f''(\alpha)}{2!}\cdot  s^2p^2+\cdots+\frac{f^{(k)}(\alpha)}{k!}\cdot  s^kp^k+\cdots,
$$
where $f^{(k)}(x)$ denotes the $k$-th ordinary derivative of $f(x)$ as a rational function, not the $p$-adic derivative.
\end{lemma}
Let us turn the derivative of the $p$-adic gamma function. For a continuous function $f:\,\Z_p\to\Z_p$, define the $p$-adic derivative
$$
f'(\alpha):=\lim_{|x-\alpha|_p\to 0}\frac{f(x)-f(\alpha)}{x-\alpha}.
$$
Clearly if the above limit exists, then
$$
f'(\alpha)=\lim_{n\to\infty}\frac{f(\alpha+p^n)-f(\alpha)}{p^n}.
$$
Let $\Gamma_p'(x)$ denote the $p$-adic derivative of $\Gamma_p(x)$. The following lemma establishes a connection between the $p$-adic derivative of $\Gamma_p(x)$ and the common derivative of $\Gamma(x)$. 
\begin{lemma}\label{Gammapadicderivativealphabeta}
Let $s_1,\ldots,s_{m},t_1,\ldots,t_n$ be some $p$-integral rational numbers with $s_1+\cdots+s_m=t_1+\cdots+t_n$. 
Then for any $\alpha_1,\ldots,\alpha_m,\beta_1,\ldots,\beta_n\in\Z_p$,
\begin{align}
&\frac{\Gamma_p(\alpha_1+s_1p)\cdots\Gamma_p(\alpha_m+s_mp)}{\Gamma_p(\beta_1+t_1p)\cdots\Gamma_p(\beta_n+t_np)}-\frac{\Gamma_p(\alpha_1)\cdots\Gamma_p(\alpha_m)}{\Gamma_p(\beta_1)\cdots\Gamma_p(\beta_n)}\notag\\
\equiv&(-1)^{\delta}p\cdot\frac{d}{d x}\bigg(\frac{\Gamma(a_1+s_1x)\cdots\Gamma(a_m+s_mx)}{\Gamma(b_1+t_1x)\cdots\Gamma(b_n+t_nx)}\bigg)\bigg|_{x=0}\pmod{p^2},
\end{align}
where $a_i=p-\langle-\alpha_i\rangle_p$, $b_j=p-\langle-\beta_j\rangle_p$ and $\delta=a_1+\cdots+a_m-b_1-\cdots-b_n$.
\end{lemma}
\begin{proof} 
We know that
\begin{equation}\label{padicgammaderivative}
\frac{\Gamma_p'(\alpha)}{\Gamma_p(\alpha)}\equiv \Gamma_p'(0)+H_{p-\langle-\alpha\rangle_p-1}\pmod{p}
\end{equation}
for each $\alpha\in\Z_p$, where the $n$-th harmonic number
$$
H_n:=\sum_{k=1}^n\frac1k.
$$
On the other hand, we have (cf. \cite[Theorem 1.2.5]{AAR99})
\begin{equation}\label{Gammaderivative}
\frac{\Gamma'(x)}{\Gamma(x)}=-\frac1x-\gamma_0+\sum_{n=1}^{\infty}\bigg(\frac1n-\frac1{n+x}\bigg),
\end{equation}
where $\gamma_0$ is the Euler constant. So comparing (\ref{padicgammaderivative}) with (\ref{Gammaderivative}), we obtain that
\begin{equation}
\label{padicgammaderivativecommongammaderivative}
\frac{\Gamma_p'(\alpha)}{\Gamma_p(\alpha)}-\Gamma_p'(0)\equiv 
\frac{\Gamma'(a)}{\Gamma(a)}+\gamma_0\pmod{p}
\end{equation}
where $a=p-\langle-\alpha\rangle_p$.

Clearly 
\begin{align*}
&\frac{\Gamma_p(\alpha_1+s_1p)\cdots\Gamma_p(\alpha_m+s_mp)}{\Gamma_p(\beta_1+t_1p)\cdots\Gamma_p(\beta_n+t_np)}\\
\equiv&\frac{(\Gamma_p(\alpha_1)+\Gamma_p'(\alpha_1)\cdot s_1p)\cdots(\Gamma_p(\alpha_m)+\Gamma_p'(\alpha_m)\cdot s_mp)}{(\Gamma_p(\beta_1)+\Gamma_p'(\beta_1)\cdot t_1p)\cdots(\Gamma_p(\beta_n)+\Gamma_p'(\beta_n)\cdot t_np)}\\
\equiv&\frac{\Gamma_p(\alpha_1)\cdots\Gamma_p(\alpha_m)}{\Gamma_p(\beta_1)\cdots\Gamma_p(\beta_n)}\bigg(1+p\sum_{i=1}^ms_i\cdot \frac{\Gamma_p'(\alpha_i)}{\Gamma_p(\alpha_i)}-
p\sum_{j=1}^nt_j\cdot \frac{\Gamma_p'(\beta_j)}{\Gamma_p(\beta_j)}\bigg)\pmod{p^2}.
\end{align*}
Since $s_1+\cdots+s_m=t_1+\cdots+t_n$, in view of (\ref{padicgammaderivativecommongammaderivative}), we have
\begin{align*}
\sum_{i=1}^ms_i\cdot\frac{\Gamma_p'(\alpha_i)}{\Gamma_p(\alpha_i)}-
\sum_{j=1}^nt_j\cdot\frac{\Gamma_p'(\beta_j)}{\Gamma_p(\beta_j)}
=&\sum_{i=1}^ms_i\bigg(\frac{\Gamma_p'(\alpha_i)}{\Gamma_p(\alpha_i)}-\Gamma_p'(0)\bigg)-
\sum_{j=1}^nt_j\bigg(\frac{\Gamma_p'(\beta_j)}{\Gamma_p(\beta_j)}-\Gamma_p'(0)\bigg)\\
\equiv&\sum_{i=1}^ms_i\bigg(\frac{\Gamma'(a_i)}{\Gamma(a_i)}+\gamma_0\bigg)-
\sum_{j=1}^nt_j\bigg(\frac{\Gamma'(b_j)}{\Gamma(b_j)}+\gamma_0\bigg)\\
=&\sum_{i=1}^ms_i\cdot\frac{\Gamma'(a_i)}{\Gamma(a_i)}-
\sum_{j=1}^nt_j\cdot\frac{\Gamma'(b_j)}{\Gamma(b_j)}\pmod{p}.
\end{align*}
So
\begin{align*}
&\frac{\Gamma_p(\alpha_1+s_1p)\cdots\Gamma_p(\alpha_m+s_mp)}{\Gamma_p(\beta_1+t_1p)\cdots\Gamma_p(\beta_n+t_np)}-
\frac{\Gamma_p(\alpha_1)\cdots\Gamma_p(\alpha_m)}{\Gamma_p(\beta_1)\cdots\Gamma_p(\beta_n)}\\
\equiv&p\cdot\frac{\Gamma_p(a_1)\cdots\Gamma_p(a_m)}{\Gamma_p(b_1)\cdots\Gamma_p(b_n)}\bigg(\sum_{i=1}^m s_i\cdot\frac{\Gamma'(a_i)}{\Gamma(a_i)}-
\sum_{j=1}^nt_j\cdot\frac{\Gamma'(b_j)}{\Gamma(b_j)}\bigg)\\
=&(-1)^{\delta}p\cdot\frac{d}{d x}\bigg(\frac{\Gamma(a_1+s_1x)\cdots\Gamma(a_m+s_mx)}{\Gamma(b_1+t_1x)\cdots\Gamma(b_n+t_nx)}\bigg)\bigg|_{x=0}\pmod{p^2}.
\end{align*}
\end{proof}

\section{The $p$-adic analogue of Gauss' identity}
\label{sectionGauss}
\setcounter{lemma}{0}
\setcounter{theorem}{0}
\setcounter{corollary}{0}
\setcounter{remark}{0}
\setcounter{equation}{0}
\setcounter{conjecture}{0}

In this section, we shall give the detailed proof of Theorem \ref{alphabeta11}. In fact, we have the following $p$-adic analogue of Gauss' identity (\ref{Gaussidentity}).
\begin{theorem}\label{alphabetagamma11}
Let $\alpha,\beta\in\Z_p$ and $\gamma\in\Z_p^\times$ with $\langle-\alpha\rangle_p,\langle-\beta\rangle_p\leq\langle-\gamma\rangle_p$.

\medskip\noindent
(i) If $\langle-\alpha\rangle_p+\langle-\beta\rangle_p\leq\langle-\gamma\rangle_p$, then
\begin{equation}\label{alphabetagamma11c1}
{}_2F_1\bigg[\begin{matrix}
\alpha&\beta\\
&\gamma
\end{matrix}\bigg|\,1\bigg]_{\langle-\gamma\rangle_p}\equiv\frac{\Gamma_p(\gamma)\Gamma_p(\gamma-\alpha-\beta)}{\Gamma_p(\gamma-\alpha)\Gamma_p(\gamma-\beta)}\pmod{p^2}.
\end{equation}

\medskip\noindent
(ii) If $\langle-\alpha\rangle_p+\langle-\beta\rangle_p>\langle-\gamma\rangle_p$, then
\begin{equation}\label{alphabetagamma11c2}
{}_2F_1\bigg[\begin{matrix}
\alpha&\beta\\
&\gamma
\end{matrix}\bigg|\,1\bigg]_{\langle-\gamma\rangle_p}\equiv (\gamma+\langle-\gamma\rangle_p-\alpha-\langle-\alpha\rangle_p-\beta-\langle-\beta\rangle_p)\cdot \frac{\Gamma_p(\gamma)\Gamma_p(\gamma-\alpha-\beta)}{\Gamma_p(\gamma-\alpha)\Gamma_p(\gamma-\beta)}\pmod{p^2}.
\end{equation}
\end{theorem}
\begin{proof}
Let $a=\langle-\alpha\rangle_p$, $b=\langle-\beta\rangle_p$ and $c=\langle-\gamma\rangle_p$.
Since $a\leq c$, by (\ref{Gaussidentity}), we have
\begin{align*}
{}_2F_1\bigg[\begin{matrix}
-a&\beta\\
&\gamma
\end{matrix}\bigg|\,1\bigg]_{c}={}_2F_1\bigg[\begin{matrix}
-a&\beta\\
&\gamma
\end{matrix}\bigg|\,1\bigg]=\frac{\Gamma(\gamma)\Gamma(\gamma+a-\beta)}{\Gamma(\gamma+a)\Gamma(\gamma-\beta)}.
\end{align*}

First, assume that $a+b\leq c$. Since $\gamma\not\equiv -j\pmod{p}$ for any $0\leq j<a$, in view of (\ref{Gammapx1Gammapx}), we have
$$
\frac{\Gamma(\gamma)}{\Gamma(\gamma+a)}=\prod_{j=0}^{a-1}\frac1{\gamma+j}=(-1)^a\cdot\frac{\Gamma_p(\gamma)}{\Gamma_p(\gamma+a)}.
$$
Similarly,
$$
\frac{\Gamma(\gamma+a-\beta)}{\Gamma(\gamma-\beta)}=(-1)^a\cdot\frac{\Gamma_p(\gamma+a-\beta)}{\Gamma_p(\gamma-\beta)}.
$$
Thus
$$
{}_2F_1\bigg[\begin{matrix}
-a&\beta\\
&\gamma
\end{matrix}\bigg|\,1\bigg]_{c}=
\frac{\Gamma_p(\gamma)\Gamma_p(\gamma+a-\beta)}{\Gamma_p(\gamma+a)\Gamma_p(\gamma-\beta)}.
$$

Let
$$
\Psi(x)={}_2F_1\bigg[\begin{matrix}
-a+x&\beta\\
&\gamma
\end{matrix}\bigg|\,1\bigg]_{c},\qquad 
\psi(x)={}_2F_1\bigg[\begin{matrix}
-a+x&-b\\
&p-c
\end{matrix}\bigg|\,1\bigg].
$$
Clearly $\psi(x)$ is also a polynomial in $x$. And we also have 
$\Psi'(0)\equiv \psi'(0)\pmod{p}$, since $(\beta)_k\equiv(-b)_k\pmod{p}$ and 
$(\gamma)_k\equiv(p-c)_k\pmod{p}$ for each $0\leq k\leq b$.
 Let $s=(\alpha+a)/p$. By Lemma \ref{taylorexpansionrational},
\begin{align*}
&{}_2F_1\bigg[\begin{matrix}
\alpha&\beta\\
&\gamma
\end{matrix}\bigg|\,1\bigg]_{c}-{}_2F_1\bigg[\begin{matrix}
-2&\beta\\
&\gamma
\end{matrix}\bigg|\,1\bigg]_{c}=\Psi(sp)-\Psi(0)\equiv
\Psi'(0)\cdot sp\equiv\psi'(0)\cdot sp\pmod{p^2}.
\end{align*}
On the other hand, using (\ref{Gaussidentity}) and Lemma \ref{Gammapadicderivativealphabeta}, we also get
\begin{align*}
sp\cdot \psi'(0)=&p\cdot \frac{d}{dx}\bigg({}_2F_1\bigg[\begin{matrix}
-a+sx&-b\\
&p-c
\end{matrix}\bigg|\,1\bigg]\bigg)\bigg|_{x=0}\\
=&\frac{\Gamma(p-c)}{\Gamma(p-c+b)}\cdot
p\cdot \frac{d}{dx}\bigg(\frac{\Gamma(p-c+a+b-sx)}{\Gamma(p-c+a-sx)}\bigg)\bigg|_{x=0}\\
\equiv&\frac{\Gamma_p(\gamma)}{\Gamma_p(\gamma-\beta)}\bigg(\frac{\Gamma_p(\gamma-\beta+a-sp)}{\Gamma_p(\gamma+a-sp)}-\frac{\Gamma_p(\gamma-\beta+a)}{\Gamma_p(\gamma+a)}\bigg)\pmod{p^2}.
\end{align*}
It follows that
\begin{align*}
&{}_2F_1\bigg[\begin{matrix}
\alpha&\beta\\
&\gamma
\end{matrix}\bigg|\,1\bigg]_{c}-
\frac{\Gamma_p(\gamma)}{\Gamma_p(\gamma-\beta)}\cdot\frac{\Gamma_p(\gamma-\beta-\alpha)}{\Gamma_p(\gamma-\alpha)}\\
\equiv&{}_2F_1\bigg[\begin{matrix}
-a&\beta\\
&\gamma
\end{matrix}\bigg|\,1\bigg]_{c}-\frac{\Gamma_p(\gamma)}{\Gamma_p(\gamma-\beta)}\cdot\frac{\Gamma_p(\gamma-\beta+a)}{\Gamma_p(\gamma+a)}=0\pmod{p^2}.
\end{align*}
Then (\ref{alphabetagamma11c1}) is concluded.

Next, assume that $a+b\geq c+1$. Since $p$ divides $\gamma-\beta+c-b$, now we have
\begin{align*}
{}_2F_1\bigg[\begin{matrix}
-a&\beta\\
&\gamma
\end{matrix}\bigg|\,1\bigg]_{c}=\frac{\Gamma(\gamma)\Gamma(\gamma+a-\beta)}{\Gamma(\gamma+a)\Gamma(\gamma-\beta)}=&
(\gamma-\beta+c-b)\cdot\frac{\Gamma_p(\gamma)\Gamma_p(\gamma+a-\beta)}{\Gamma_p(\gamma+a)\Gamma_p(\gamma-\beta)}\\
\equiv&
(\gamma-\beta+c-b)\cdot\frac{\Gamma_p(\gamma)\Gamma_p(\gamma-\alpha-\beta)}{\Gamma_p(\gamma-\alpha)\Gamma_p(\gamma-\beta)}\pmod{p^2}.
\end{align*}
Moreover,
\begin{align*}
\frac{d}{dx}\bigg({}_2F_1\bigg[\begin{matrix}
-a+x&\beta\\
&\gamma
\end{matrix}\bigg|\,1\bigg]_c\bigg)\bigg|_{x=0}\equiv&
\frac{d}{dx}\bigg({}_2F_1\bigg[\begin{matrix}
p-a+x&-b\\
&p-c
\end{matrix}\bigg|\,1\bigg]_c\bigg)\bigg|_{x=0}\\
=&
\frac{\Gamma(p-c)}{\Gamma(p-c+b)}\cdot \frac{d}{dx}\bigg(\frac{\Gamma(a+b-c-x)}{\Gamma(a-c-x)}\bigg)\bigg|_{x=0}\pmod{p^2}.
\end{align*}
Note that
\begin{align*}
\frac{d}{dx}\bigg(\frac{\Gamma(a+b-c-x)}{\Gamma(a-c-x)}\bigg)\bigg|_{x=0}=&
\frac{d\big((a-c-x)_b\big)}{dx}\bigg|_{x=0}\\
=&-\prod_{\substack{0\leq j<b\\ j\neq c-a}}(a-c+j)
=(-1)^{b-1}\cdot\frac{\Gamma_p(a+b-c)}{\Gamma_p(a-c)}.
\end{align*}
It follows from Lemma \ref{taylorexpansionrational} that
\begin{align*}
{}_2F_1\bigg[\begin{matrix}
\alpha&\beta\\
&\gamma
\end{matrix}\bigg|\,1\bigg]_{c}-{}_2F_1\bigg[\begin{matrix}
-a&\beta\\
&\gamma
\end{matrix}\bigg|\,1\bigg]_c\equiv&
-sp\cdot\frac{\Gamma_p(p-c)\Gamma_p(a+b-c)}{\Gamma_p(p-c+b)\Gamma_p(a-c)}\\
\equiv&
-sp\cdot\frac{\Gamma_p(\gamma)\Gamma_p(\gamma-\alpha-\beta)}{\Gamma_p(\gamma-\beta)\Gamma_p(\gamma-\alpha)}\pmod{p^2}.
\end{align*}
Thus (\ref{alphabetagamma11c2}) is also concluded since $sp=\alpha+a$.
\end{proof}
We mention that by using the similar discussions, in the other theorems, we also may replace ${}_{m+1}F_m\Big[\begin{matrix}
*&*&*&\ldots&*\\
&1&*&\ldots&*
\end{matrix}\Big|\,z\Big]_{p-1}$  by ${}_{m+1}F_m\Big[\begin{matrix}
*&*&*&\ldots&*\\
&\gamma&*&\ldots&*
\end{matrix}\Big|\,z\Big]_{\langle-\gamma\rangle_p}$. For example, assume that

\medskip\noindent
(i) $\langle-\alpha\rangle_p,\langle-\gamma\rangle_p$ are even and $\langle-\alpha\rangle_p+\langle-\gamma\rangle_p<p$;

\medskip\noindent
(ii) $\langle\beta\rangle<p/2$ and $(2\beta)_{p-1}\not\equiv0\pmod{p^2}$.

\medskip\noindent
Then we have a complete $q$-analogue of Watson's identity (\ref{Watsonidentity}):
\begin{align}\label{padicWatsonidentity}
&{}_3F_2\bigg[\begin{matrix}
\alpha&\beta&\gamma\\
&2\beta&\frac12(\alpha+\gamma+1)
\end{matrix}\bigg|\,1\bigg]_{\frac12(p-1+\langle-\alpha\rangle_p+\langle-\gamma\rangle_p)}\notag\\
\equiv&\frac{\Gamma_p(\frac12)\Gamma_p(\frac12+\beta)\Gamma_p(\frac12+\frac12(\alpha+\gamma))\Gamma_p(\frac12+\beta-\frac12(\alpha+\gamma))}{\Gamma_p(\frac12+\frac12\alpha)\Gamma_p(\frac12+\frac12\gamma)\Gamma_p(\frac12+\beta-\frac12\alpha)\Gamma_p(\frac12+\beta-\frac12\gamma)}\pmod{p^2}.
\end{align}

Furthermore, we have the following $p$-adic analogue of (\ref{nbetagammaz1z}), which also generalizes both (\ref{SunZHalphaz}) and (\ref{alphabetagamma11c1}).
\begin{theorem}\label{alphabetagammaz1z}
Let $\alpha\in\Z_p$ and $\beta,\gamma\in\Z_p^\times$ with $\langle-\alpha\rangle_p+\langle-\beta\rangle_p\leq\langle-\gamma\rangle_p$.
Then
\begin{align}
{}_2F_1\bigg[\begin{matrix}
\alpha&\beta\\
&\gamma
\end{matrix}\bigg|\,z\bigg]_{\langle-\gamma\rangle_p}
\equiv
\frac{\Gamma_p(\gamma)\Gamma_p(\gamma-\alpha-\beta)}{\Gamma_p(\gamma-\alpha)\Gamma_p(\gamma-\beta)}\cdot{}_2F_1\bigg[\begin{matrix}
\alpha&\beta\\
&\alpha+\beta-\gamma+1
\end{matrix}\bigg|\,1-z\bigg]_{\langle\gamma-\alpha-\beta-1\rangle_p}\pmod{p^2}.
\end{align}
\end{theorem}
\begin{proof}
Let $a=\langle-\alpha\rangle_p$, $b=\langle-\beta\rangle_p$, $c=\langle-\gamma\rangle_p$ and $d=\langle\gamma-\alpha-\beta-1\rangle_p$. Clearly $d=p+a+b-c-1$ and $d\geq\max\{a,b\}$. Let
$$
\Psi(x)={}_2F_1\bigg[\begin{matrix}
-a+x&\beta\\
&\gamma
\end{matrix}\bigg|\,z\bigg]_{c},
\qquad
\Omega(x)=\frac{\Gamma_p(\gamma)\Gamma_p(\gamma+a-\beta-x)}{\Gamma_p(\gamma+a-x)\Gamma_p(\gamma-\beta)}
$$
and
$$
\Phi(x)={}_2F_1\bigg[\begin{matrix}
-a+x&\beta\\
&1-a+\beta-\gamma+x
\end{matrix}\bigg|\,1-z\bigg]_{d}.
$$
In view of (\ref{nbetagammaz1z}), we have
\begin{align*}
\Psi(0)={}_2F_1\bigg[\begin{matrix}
-a&\beta\\
&\gamma
\end{matrix}\bigg|\,z\bigg]
=\frac{\Gamma_p(\gamma)\Gamma_p(\gamma+a-\beta)}{\Gamma_p(\gamma+a)\Gamma_p(\gamma-\beta)}\cdot{}_2F_1\bigg[\begin{matrix}
-a&\beta\\
&1-a+\beta-\gamma
\end{matrix}\bigg|\,1-z\bigg]=\Omega(0)\Phi(0).
\end{align*}

Let
$$
\omega(x)=\frac{\Gamma(p-c)}{\Gamma(p+b-c)}\cdot
\frac{\Gamma(p+a+b-c-x)}{\Gamma(p+a-c-x)}
$$
and
$$
\phi(x)={}_2F_1\bigg[\begin{matrix}
-a+x&-b\\
&1-a-b+c-p+x
\end{matrix}\bigg|\,1-z\bigg].
$$
Clearly
$$
\Phi(0)\equiv\phi(0)\pmod{p},\qquad\Phi'(0)\equiv\phi'(0)\pmod{p}.
$$
By Lemma \ref{Gammapadicderivativealphabeta}, we also have
$$
\Omega(0)\equiv\omega(0)\pmod{p},\qquad\Omega'(0)\equiv\omega'(0)\pmod{p}.
$$
Hence
\begin{align*}
\Psi'(0)\equiv&\frac{d}{dx}\bigg({}_2F_1\bigg[\begin{matrix}
-a+x&-b\\
&p-c
\end{matrix}\bigg|\,z\bigg]\bigg)\bigg|_{x=0}\equiv\frac{d\big(\omega(x)\phi(x)\big)}{dx}\bigg|_{x=0}\\
=&\omega'(0)\phi(0)+\omega(0)\phi'(0)\equiv\frac{d\big(\Omega(x)\Phi(x)\big)}{dx}\bigg|_{x=0}\pmod{p}.
\end{align*}
Letting $s=(\alpha+a)/p$, we get
$$
\Psi(sp)\equiv\Psi(0)+sp\cdot \Psi'(0)\equiv \Omega(0)\Phi(0)+sp\cdot\frac{d\big(\Omega(x)\Phi(x)\big)}{dx}\bigg|_{x=0}\equiv\Omega(sp)\Phi(sp)\pmod{p^2}.
$$
\end{proof}

\section{The $p$-adic quadratic ${}_2F_1$ transformation I: $z\to 4z(1-z)$}
\label{sectionq2F1z4z1z}
\setcounter{lemma}{0}
\setcounter{theorem}{0}
\setcounter{corollary}{0}
\setcounter{remark}{0}
\setcounter{equation}{0}
\setcounter{conjecture}{0}

In this section, we shall prove Theorems \ref{alphaz4z1z} and \ref{alpha1alphaz124z1z}.
The proof of Theorem \ref{alphaz4z1z} is simple.
\begin{proof}[Proof of Theorem \ref{alphaz4z1z}]
Let $a=\langle-\alpha\rangle_p$.
Let
$$
\Psi(x,z)={}_2F_1\bigg[\begin{matrix}
-a-x&1+a+x\\
&1
\end{matrix}\bigg|\,z\bigg]_{p-1}
$$
and
$$
\Phi(x,z)={}_2F_1\bigg[\begin{matrix}
-\frac12(a+x)&\frac12+\frac12(a+x)\\
&1
\end{matrix}\bigg|\,4z(1-z)\bigg]_{p-1}
$$
Since $a$ is even, by (\ref{alphaz4z1ze}),
\begin{align*}
&\Psi(0,z)={}_2F_1\bigg[\begin{matrix}
-a&1+a\\
&1
\end{matrix}\bigg|\,z\bigg]
={}_2F_1\bigg[\begin{matrix}
-\frac12a&\frac12+\frac12a\\
&1
\end{matrix}\bigg|\,4z(1-z)\bigg]=\Phi(0,z).
\end{align*}

Let
$$
F(z)=\frac{\partial\Psi(x,z)}{\partial x}\bigg|_{x=0}\quad\text{and}\quad G(z)=\frac{\partial\Phi(x,z)}{\partial x}\bigg|_{x=0}.
$$
Letting $s=-(\alpha+a)/p$, Theorem \ref{alphaz4z1z} is equivalent to $\Psi(sp,z)\equiv\Phi(sp,z)\pmod{p^2}$. In view of Lemma \ref{taylorexpansionrational},
it suffices to show that
\begin{equation}\label{FzGzzn}
[z^n]F(z)\equiv
[z^n]G(z)\pmod{p}
\end{equation}
for each $0\leq n\leq 2p-2$, where $[z^n]F(z)$ denotes the coefficient of $z^n$ in $F(z)$.
Suppose that $p\leq n\leq 2p-2$. 
Note that $\binom{k}{n-k}$ vanishes for each $0\leq k<n/2$ and $(-a/2)_k=0$ for each $a/2<k\leq p-1$.
It is easy to see that
\begin{align*}
[z^n]G(z)=&(-1)^n\sum_{k=0}^{p-1}\binom{k}{n-k}(-4)^k\cdot
\frac{d}{dx}\bigg(\frac{(-\frac12(a+x))_k}{(1)_k}\cdot\frac{(\frac12+\frac12(a+x))_k}{k!}\bigg)\bigg|_{x=0}\\
=&(-1)^n\sum_{\frac 12n\leq k\leq p-1}\binom{k}{n-k}(-4)^k\cdot\frac{(\frac12+\frac12a)_k}{k!}\cdot
\frac{d}{dx}\bigg(\frac{(-\frac12(a+x))_k}{(1)_k}\bigg)\bigg|_{x=0}.
\end{align*}
Also, if $(p-a+1)/2\leq k\leq p-1$, then $$\bigg(\frac12+\frac a2\bigg)_k=\prod_{j=0}^{k-1}\frac{1+a+2j}{2}\equiv 0\pmod{p}.
$$
Thus we always have
$
[z^n]G(z)\equiv0\pmod{p}
$
for those $p\leq n\leq 2p-2$.
On the other hand, for any $0\leq n\leq p-1$, clearly we have
\begin{align*}
[z^n]F(z)=&\frac{\partial}{\partial x}\bigg([z^n]{}_2F_1\bigg[\begin{matrix}
-a-x&1+a+x\\
&1
\end{matrix}\bigg|\,z\bigg]\bigg)\bigg|_{x=0}\\
=&\frac{\partial}{\partial x}\bigg([z^n]{}_2F_1\bigg[\begin{matrix}
-\frac12(a+x)&\frac12+\frac12(a+x)\\
&1
\end{matrix}\bigg|\,4z(1-z)\bigg]\bigg)\bigg|_{x=0}=[z^n]G(z).
\end{align*}
Thus (\ref{FzGzzn}) is derived.
\end{proof}
Theorem \ref{alpha1alphaz124z1z} can be proved in the same way.
\begin{proof}[Proof of Theorem \ref{alpha1alphaz124z1z}]
Let $a=\langle-\alpha\rangle_p$. 
Let
$$
\Psi(x,z)=\bigg({}_2F_1\bigg[\begin{matrix} -a-x&1+a+x\\ &1\end{matrix}\bigg|\,z\bigg]_{p-1}\bigg)^2
$$
and
$$
\Phi(x,z)={}_3F_2\bigg[\begin{matrix} -a-x&1+a+x&\frac12\\ &1&1\end{matrix}\bigg|\,4z(1-z)\bigg]_{p-1}.
$$
Let $F(z)=\Psi_x'(0,z)$ and $G(z)=\Phi_x'(0,z)$. Then for each $p\leq n\leq 2p-2$, 
$$
[z^n]G(z)=(-1)^n\sum_{k=0}^{p-1}\binom{k}{n-k}(-4)^k\cdot
\frac{d}{dx}\bigg(\frac{(-a-x)_k(1+a+x)_k(\frac12)_k}{(1)_k^3}\bigg)\bigg|_{x=0}\equiv 0\pmod{p},
$$
since $\binom{k}{n-k}=0$ for $0\leq k\leq (p-1)/2$ and $(\frac12)_k\equiv 0\pmod{p}$ for $(p+1)/2\leq k\leq p-1$. Also, note that $(-a)_k\equiv 0\pmod{p}$ for $a+1\leq k\leq p-1$ and $(1+a)_k\equiv 0\pmod{p}$ for $p-a\leq k\leq p-1$. Hence
$$
[z^h]\sum_{k=0}^{p-1}\frac{(-a)_k(1+a)_k}{(1)_k^2}\cdot z^k\equiv 0\pmod{p}
$$
for $h\geq\min\{a+1,p-a\}$, and
$$
[z^h]\sum_{k=0}^{p-1}\bigg(\frac{(-a)_k}{(1)_k}\cdot\frac{d}{d x}\bigg(\frac{(1+a+x)_k}{(1)_k}\bigg)+\frac{(1+a)_k}{(1)_k}\cdot\frac{d}{d x}\bigg(\frac{(-a-x)_k}{(1)_k}\bigg)\bigg)\bigg|_{x=0}\cdot z^k\equiv 0\pmod{p}
$$
for $h\geq\max\{a+1,p-a\}$. So for $p\leq n\leq 2p-2$, we have
$$
[z^n]F(z)=[z^n]2\bigg(\sum_{k=0}^{p-1}\frac{(-a)_k(1+a)_k}{(1)_k^2}\cdot z^k\bigg)\cdot\bigg(\sum_{k=0}^{p-1}\frac{d}{d x}\bigg(\frac{(-a-x)_k(1+a+x)_k}{(1)_k^2}\bigg)\bigg|_{x=0}\cdot z^k\bigg)
$$
is divisible by $p$.
And we also have $[z^n]F(z)=[z^n]G(z)$ for $0\leq n\leq p-1$. Since $\Psi(0,z)=\Psi(0,z)$ by (\ref{alpha1alphaz124z1ze}),
 Theorem \ref{alpha1alphaz124z1z} is concluded, .\end{proof}

In the remainder part of this section, we shall give an explain about the relation between ${}_2F_1\bigg[\begin{matrix}
\frac12&\frac12\\
&1
\end{matrix}\bigg|\,z\bigg]_{p-1}$, ${}_2F_1\bigg[\begin{matrix}
\frac14&\frac34\\
&1
\end{matrix}\bigg|\,z\bigg]_{p-1}$ and the Legendre polynomials.
The Legendre polynomial $P_n(x)$ is given by
$$
\sum_{n=0}^\infty P_n(x)t^n=\frac1{\sqrt{1-2xt+t^2}}.
$$
$P_n(x)$ also has the explicit expression
$$
P_n(x)={}_2F_1\bigg[\begin{matrix}
-n&1+n\\
&1
\end{matrix}\bigg|\,\frac{1-x}{2}\bigg].
$$
Notice that for $0\leq k\leq (p-1)/2$,
$$
\frac{(-\frac{p-1}2)_k(1+\frac{p-1}2)_k}{(1)_k^2}=\prod_{j=0}^{k-1}\frac{(\frac12+j-\frac p2)(\frac12+j-\frac p2)}{(1+j)^2}\equiv\prod_{j=0}^{k-1}\frac{(\frac12+j)^2}{(1+j)^2}=\frac{(\frac12)_k^2}{(1)_k^2}\pmod{p^2},
$$
and $(\frac12)_k\equiv0\pmod{p}$ for $(p+1)/2\leq k\leq p-1$.
Thus we have
\begin{equation}\label{F211212zLegendre}
{}_2F_1\bigg[\begin{matrix}
\frac{1}{2}&\frac{1}{2}\\
&1
\end{matrix}\bigg|\,z\bigg]_{p-1}
\equiv P_{\frac{p-1}{2}}(1-2z)\pmod{p^2}.
\end{equation}
Consequently, by Theorem \ref{alpha1alphaz124z1z},
\begin{equation}
{}_3F_2\bigg[\begin{matrix}
\frac{1}{2}&\frac{1}{2}&\frac12\\
&1&1
\end{matrix}\bigg|\,4z(1-z)\bigg]_{p-1}
\equiv P_{\frac{p-1}{2}}(1-2z)^2\pmod{p^2}.
\end{equation}

Furthermore, we also have
\begin{theorem}
\begin{equation}\label{F2114144z1zLegendre}
{}_2F_1\bigg[\begin{matrix}
\frac{1}{4}&\frac{1}{4}\\
&1
\end{matrix}\bigg|\,4z(1-z)\bigg]_{p-1}
\equiv\begin{cases}
P_{\frac{p-1}{2}}(1-2z)\pmod{p^2},&\text{if }p\equiv1\pmod{4},\\
P_{\frac{3p-1}{2}}(1-2z)\pmod{p^2},&\text{if }p\equiv3\pmod{4}.\end{cases}
\end{equation}
\end{theorem}
\begin{proof}
The case $p\equiv 1\pmod{4}$ immediately follows from  (\ref{F211212zLegendre}) and Theorem \ref{alphaz4z1z}.
Assume that $p\equiv 3\pmod{4}$. Let
$$
\Psi(x)={}_2F_1\bigg[\begin{matrix}
-\frac{3p-1}{4}+\frac32x&1+\frac{3p-1}{4}-\frac32x\\
&1
\end{matrix}\bigg|\,4z(1-z)\bigg]_{p-1}.
$$
Clearly $\Psi(p)=\Psi(0)$. Thus
\begin{align*}
{}_2F_1\bigg[\begin{matrix}
\frac{1}{4}&\frac{1}{4}\\
&1
\end{matrix}\bigg|\,4z(1-z)\bigg]_{p-1}=&\Psi\bigg(\frac p2\bigg)\equiv \Psi(0)+\Psi'(0)\cdot\frac{p}{2}\equiv\Psi(0)+\frac{\Psi(p)-\Psi(0)}{p}\cdot\frac{p}{2}\\
=&\Psi(0)=
{}_2F_1\bigg[\begin{matrix}
-\frac{3p-1}{4}&1+\frac{3p-1}{4}\\
&1
\end{matrix}\bigg|\,4z(1-z)\bigg]\\
=&{}_2F_1\bigg[\begin{matrix}
-\frac{3p-1}{2}&1+\frac{3p-1}{2}\\
&1
\end{matrix}\bigg|\,z\bigg]=P_{\frac{3p-1}{2}}(1-2z)\pmod{p^2}.
\end{align*}
\end{proof}
Combining (\ref{F2114144z1zLegendre}) with Theorems \ref{alphaalphazz1} and \ref{alpha1alphaz124z1z}, we have
\begin{equation}
{}_2F_1\bigg[\begin{matrix}
\frac{1}{4}&\frac{3}{4}\\
&1
\end{matrix}\bigg|\,\frac{4z(1-z)}{(1-2z)^2}\bigg]_{p-1}\equiv (1-z)^{\lambda_p(\frac14)-1}P_{\langle-\frac14\rangle_p}(1-2z)-
\bigg(1-\frac1z\bigg)^{\lambda_p(\frac14)-1}P_{\langle-\frac14\rangle_p}\bigg(1-\frac2z\bigg)\pmod{p^2},
\end{equation}
and
\begin{align}
&{}_3F_2\bigg[\begin{matrix}
\frac{1}{4}&\frac{3}{4}&\frac12\\
&1&1
\end{matrix}\bigg|\,\frac{16z(1-z)(1-8z+8z^2)}{(1-2z)^4}\bigg]_{p-1}\notag\\
\equiv&\bigg((1-z)^{\lambda_p(\frac14)-1}P_{\langle-\frac14\rangle_p}(1-2z)-
\bigg(1-\frac1z\bigg)^{\lambda_p(\frac14)-1}P_{\langle-\frac14\rangle_p}\bigg(1-\frac2z\bigg)\bigg)^2\pmod{p^2}.
\end{align}

In \cite{CoHa91}, with help of the theory of elliptic curves, Coster and Hamme computed $P_{\frac{p-1}{2}}(z)$ modulo $p^2$ for some special values of $z$. We shall give a further discussion on this topic in Section \ref{section2F1cm}.

\section{The $p$-adic quadratic ${}_2F_1$ transformation II: $z\to 4z/(1+z)^2$}
\label{sectionq2F1z4z1z2}
\setcounter{lemma}{0}
\setcounter{theorem}{0}
\setcounter{corollary}{0}
\setcounter{remark}{0}
\setcounter{equation}{0}
\setcounter{conjecture}{0}

Let us consider Theorems \ref{alpha2alpha214zz1}
and \ref{alpha2alpha21beta4zz1}.
The proof of Theorem \ref{alpha2alpha214zz1} is a little more complicated, since the function $\lambda_p(\alpha)$ is involved now. Now we can't take the derivative of $z^{-\lambda_p(a+x)}$. So we need the following lemma.
\begin{lemma}\label{polynomialzsp}
Suppose that $P(x),Q(x)\in\Z_p[x]$ are polynomials over $\Z_p$ with $Q(0)\in\Z_{p}^\times$. Let $f(x)=P(x)/Q(x)$. 
Suppose that $z\in\Z_p^\times$ and $s\in\Z$. Then
$$
z^{1+s(p-1)}f(sp)-zf(0)\equiv s\big(z^pf(p)-zf(0)\big)\pmod{p^2}.
$$
\end{lemma}
\begin{proof}
Note that
$$
z^{s(p-1)}-1=\big(1+(z^{p-1}-1)\big)^s-1\equiv s(z^{p-1}-1)\pmod{p^2}.
$$
By Lemma \ref{taylorexpansionrational}, 
\begin{align*}
z^{1+s(p-1)}f(sp)-zf(0)=&
z^{1+s(p-1)}\cdot\big(f(sp)-f(0)\big)+\big(z^{1+s(p-1)}-z\big)\cdot f(0)\\
\equiv&z^{1+s(p-1)}\cdot spf'(0)+s(z^{p}-z)\cdot f(0)\\
\equiv&z^p\cdot s(f(p)-f(0))+s(z^{p}-z)\cdot f(0)=s\big(z^pf(p)-zf(0)\big)\pmod{p^2}.
\end{align*}
\end{proof}
\begin{proof}[Proof of Theorem \ref{alpha2alpha214zz1}]
Trivially
$$
{}_2F_1\bigg[\begin{matrix}
\alpha&\beta\\
&1
\end{matrix}\bigg|\,z\bigg]_{p-1}\equiv
{}_2F_1\bigg[\begin{matrix}
-\langle-\alpha\rangle_{p^2}&-\langle-\beta\rangle_{p^2}\\
&1
\end{matrix}\bigg|\,z\bigg]_{p-1}\pmod{p^2}.
$$
Without loss of generality, assume that $\alpha=-\langle-\alpha\rangle_{p^2}$.
Let $a=\langle-\alpha\rangle_p$ and $s=-(\alpha+a)/p$.
clearly $a+s(p-1)=-\lambda_p(\alpha).$
Let
$$
\Psi_1(x)={}_2F_1\bigg[\begin{matrix}
-a-x&-a-x\\
&1
\end{matrix}\bigg|\,z\bigg]_{p-1},\quad \Psi_2(x)=z^{a}{}_2F_1\bigg[\begin{matrix}
-a-x&-a-x\\
&1
\end{matrix}\bigg|\,\frac1z\bigg]_{p-1},
$$
and let
$$
\Phi(x)=(1+z)^{a}{}_2F_1\bigg[\begin{matrix}
-\frac{1}2(a+x)&-\frac{1}2(a+x)+\frac12\\
&1
\end{matrix}\bigg|\,\frac{4z}{(1+z)^2}\bigg]_{p-1}.
$$

Note that either $a/2$ or $a/2-1/2$ is a non-negative integer. By (\ref{alpha2alpha214zz1e}),
\begin{align}\label{Psi1zPsi21zPhi}\Psi_1(0)+z\cdot\Psi_2(0)=&{}_2F_1\bigg[\begin{matrix}
-a&-a\\
&1
\end{matrix}\bigg|\,z\bigg]+z^{a+1}{}_2F_1\bigg[\begin{matrix}
-a&-a\\
&1
\end{matrix}\bigg|\,\frac1z\bigg]\notag\\
=&(1+z)^{a+1}{}_2F_1\bigg[\begin{matrix}
-\frac{a}2&-\frac{a}2+\frac12\\
&1
\end{matrix}\bigg|\,\frac{4z}{(1+z)^2}\bigg]=(1+z)\Phi(0).
\end{align}

On the other hand, since $0\leq a\leq p-1$, we also have
\begin{align*}
(1+z)^{p+1}\Phi(p)=&(1+z)^{a+p+1}{}_2F_1\bigg[\begin{matrix}
-\frac{1}2(a+p)&-\frac{1}2(a+p)+\frac12\\
&1
\end{matrix}\bigg|\,\frac{4z}{(1+z)^2}\bigg]\\
=&{}_2F_1\bigg[\begin{matrix}
-a-p&-a-p\\
&1
\end{matrix}\bigg|\,z\bigg]+z^{a+p+1}{}_2F_1\bigg[\begin{matrix}
-a-p&-a-p\\
&1
\end{matrix}\bigg|\,\frac1z\bigg].
\end{align*}
Clearly $(-a-p)_k$ is divisible by $p$ for each $a+1\leq k\leq p-1$. Note that
\begin{equation}\label{aak1akak1k}
\frac{(-a)_k}{(1)_k}=(-1)^k\cdot\binom{a}{k}=
(-1)^k\cdot\binom{a}{a-k}=(-1)^a\cdot\frac{(-a)_{a-k}}{(1)_{a-k}}
\end{equation}
for any $0\leq k\leq a$.
Hence
\begin{align}\label{apap1zapapzapap1z2F1c}
&{}_2F_1\bigg[\begin{matrix}
-a-p&-a-p\\
&1
\end{matrix}\bigg|\,z\bigg]\notag\\
\equiv&\sum_{k=0}^a\frac{(-a-p)_k^2}{(1)_{k}}\cdot\frac{z^k}{k!}+
\sum_{k=0}^a\frac{(-a-p)_{p+k}^2}{(1)_{p+k}}\cdot\frac{z^{p+k}}{(p+k)!}\notag\\
=&\sum_{k=0}^a\frac{(-a-p)_k^2}{(1)_{k}}\cdot\frac{z^k}{k!}+
\sum_{k=0}^a\frac{(-a-p)_{k}^2}{(1)_{k}}\cdot\frac{z^{p+a-k}}{k!}\notag\\
\equiv&{}_2F_1\bigg[\begin{matrix}
-a-p&-a-p\\
&1
\end{matrix}\bigg|\,z\bigg]_{p-1}+z^{a+p}{}_2F_1\bigg[\begin{matrix}
-a-p&-a-p\\
&1
\end{matrix}\bigg|\,\frac1z\bigg]_{p-1}\pmod{p^2}.
\end{align}
It follows that
\begin{align*}
(1+z)^{p}\Phi(p)
\equiv&\Psi_1(p)+z^{p}\Psi_2(p)\pmod{p^2}.
\end{align*}
By Lemma \ref{polynomialzsp} and (\ref{Psi1zPsi21zPhi}), we get that
\begin{align*}
(1+z)^{1+s(p-1)}\Phi(sp)-\Phi(0)\equiv&s\cdot \big((1+z)^p\Phi(p)-(1+z)\Phi(0)\big)\\
\equiv&s\cdot \big(\Psi_1(p)-\Psi_1(0)\big)+s\cdot\big(z^{p}\Psi_2(p)-z\Psi_2(0)\big)\\
\equiv&\big(\Psi_1(sp)-\Psi_1(0)\big)+\big(z^{sp}\Psi_2(sp)-z\Psi_2(0)\big)\pmod{p^2},
\end{align*}
i.e.,
\begin{align*}
(1+z)^{1+s(p-1)}\Phi(sp)
\equiv\Psi_1(sp)+z^{s(p-1)+1}\Psi_2(sp)\pmod{p^2}.
\end{align*}
\end{proof}
Below we just sketch the proof of Theorem  \ref{alpha2alpha21beta4zz1}, which is very similar as the proof of Theorem \ref{alpha2alpha214zz1}.
It is not difficult to check that 
\begin{align*}
\sum_{k=0}^{a}\frac{(-a-p)_{p+k}^2(\beta)_{p+k}}{(1)_{p+k}^2(1-\beta-a-p)_{p+k}}\cdot z^{k}=\frac{(\beta)_p(-\beta-p-a+1)_a}{(-\beta-p-a+1)_p(\beta)_a}\sum_{k=0}^{a}\frac{(-a-p)_k^2(\beta)_{k}}{(1)_k^2(1-\beta-a-p)_{k}}\cdot z^{a-k},
\end{align*} 
where $a=\langle-\alpha\rangle_p$.
It follows that
\begin{align*}
&{}_3F_2\bigg[\begin{matrix} -a-p&-a-p&\beta&\\ &1&1-\beta-a-p\end{matrix}\bigg|\,z\bigg]\\
\equiv&{}_3F_2\bigg[\begin{matrix} -a-p&-a-p&\beta&\\ &1&1-\beta-a-p\end{matrix}\bigg|\,z\bigg]_{p-1}\\
&+(-z)^{a+p}{}_3F_2\bigg[\begin{matrix} -a-p&-a-p&\beta&\\ &1&1-\beta-a-p\end{matrix}\bigg|\,\frac1z\bigg]_{p-1}\pmod{p^2}.
\end{align*}
Then Theorems \ref{alpha2alpha21beta4zz1} can be proved by using the same discussions in the proof of Theorem \ref{alpha2alpha214zz1}.

\section{The $p$-adic quadratic ${}_2F_1$ transformation III: $z\to z^2/(z-2)^2$}
\label{sectionq2F1zz2z22}
\setcounter{lemma}{0}
\setcounter{theorem}{0}
\setcounter{corollary}{0}
\setcounter{remark}{0}
\setcounter{equation}{0}
\setcounter{conjecture}{0}

In order to prove Theorem \ref{alpha12zalpha12zz1}, we need the following lemma.
\begin{lemma}\label{apkbetapkzkt} Suppose that $0\leq a\leq p-1$ and $\beta\in\Z_p$. Let
$
t=(\beta+\langle-\beta\rangle_p)/p
$.
Then
\begin{align}
(1-t)\sum_{k=0}^a\frac{(-a-p)_{p+k}(\beta)_{p+k}}{(1)_{p+k}^2}\cdot z^k\equiv
t(1-z)^a\sum_{k=0}^a\frac{(-a-p)_{p+k}(1-\beta)_{p+k}}{(1)_{p+k}^2}\cdot\frac{z^k}{(z-1)^k}\pmod{p^2}.
\end{align}
\end{lemma}
\begin{proof} 
Clearly for any $0\leq k\leq a$, in view of (\ref{padicgammaderivative}), we have
\begin{align*}
\frac{(-a-p)_{p+k}}{(1)_{p+k}}=&\frac{\Gamma_p(-a+k)\Gamma_p(1)}{\Gamma_p(-a-p)\Gamma_p(p+k+1)}\cdot\frac{-p}{p}\\
\equiv&-
\frac{\Gamma_p(-a+k)\Gamma_p(1)}{\Gamma_p(-a)\Gamma_p(k+1)}\bigg(1+p\cdot\frac{\Gamma_p'(-a)}{\Gamma_p(-a)}-p\cdot\frac{\Gamma_p'(k+1)}{\Gamma_p(k+1)}\bigg)\\
\equiv&-
\frac{(-a)_k}{(1)_k}-p\cdot\frac{(-a)_k}{(1)_k}\cdot(H_{p-a-1}-H_k)\pmod{p^2}.
\end{align*}
Let $b=\langle-\beta\rangle_p$. If $0\leq k\leq b$, then
\begin{align*}
\frac{(\beta)_{p+k}}{(1)_{p+k}}=&\frac{\Gamma_p(\beta+p+k)\Gamma_p(1)}{\Gamma_p(\beta)\Gamma_p(p+k+1)}\cdot\frac{tp}{p}\\\equiv&
t\cdot\frac{\Gamma_p(\beta+k)\Gamma_p(1)}{\Gamma_p(\beta)\Gamma_p(k+1)}\cdot\bigg(1+p\cdot\frac{\Gamma_p'(\beta+k)}{\Gamma_p(\beta+k)}-p\cdot\frac{\Gamma_p'(k+1)}{\Gamma_p(k+1)}\bigg)\\
\equiv&
t\cdot\frac{(\beta)_k}{(1)_k}+tp\cdot\frac{(\beta)_k}{(1)_k}\cdot\bigg(H_{p-1-b}+\sum_{j=0}^{k-1}\frac{1}{j+\beta}-H_k\bigg)\pmod{p^2},
\end{align*}
by noting that
$$
H_{p-1-b+k}=H_{p-1-b}+\sum_{j=1}^k\frac{1}{p-1-b+j}\equiv
H_{b}+\sum_{j=1}^{k}\frac{1}{j+\beta-1}\pmod{p}.$$
And if $k\geq b+1$, we also have
\begin{align*}
\frac{(\beta)_{p+k}}{(1)_{p+k}}=&\frac{\Gamma_p(\beta+p+k)\Gamma_p(1)}{\Gamma_p(\beta)\Gamma_p(p+k+1)}\cdot\frac{tp}{p}\cdot\frac{(t+1)p}{p}\equiv
(t+1)\cdot\frac{(\beta)_k}{(1)_k}\\
\equiv&t\cdot\frac{(\beta)_k}{(1)_k}+tp\cdot\frac{(\beta)_k}{(1)_k}\cdot\bigg(H_{p-1-b}+\sum_{j=0}^{k-1}\frac{1}{j+\beta}-H_k\bigg)\pmod{p^2}.
\end{align*}

Now by the quadratic transformation (\ref{Eulertransformation}),
\begin{align*}
\sum_{k=0}^a\frac{(-a)_{k}(\beta)_{k}}{(1)_{k}^2}\cdot z^k=&
{}_2F_1\bigg[\begin{matrix}
-a&\beta\\
&1
\end{matrix}\bigg|\,z\bigg]
=
(1-z)^a{}_2F_1\bigg[\begin{matrix}
-a&1-\beta\\
&1
\end{matrix}\bigg|\,\frac{z}{z-1}\bigg]\\
=&(1-z)^a\sum_{k=0}^a\frac{(-a)_{k}(1-\beta)_{k}}{(1)_{k}^2}\cdot\frac{z^k}{(z-1)^k}.
\end{align*}
Also, $\langle 1-\beta\rangle_p=p-1-b$ and
$
H_b\equiv H_{p-1-b}\pmod{p}.
$
So it suffices to show that
\begin{align}\label{akbetak1k2zk}
&\sum_{k=0}^a\frac{(-a)_{k}(\beta)_{k}}{(1)_{k}^2}\cdot z^k\cdot\bigg(\sum_{j=0}^{k-1}\frac{1}{j+\beta}-2H_k\bigg)\notag\\
=&
(1-z)^a\sum_{k=0}^a\frac{(-a)_{k}(1-\beta)_{k}}{(1)_{k}^2}\cdot\frac{z^k}{(z-1)^k}\cdot\bigg(\sum_{j=0}^{k-1}\frac{1}{j+1-\beta}-2H_k\bigg).
\end{align}
Clearly (\ref{akbetak1k2zk}) immediately follows from
$$
\frac{d}{dx}\bigg({}_2F_1\bigg[\begin{matrix}
-a&\beta+x\\
&1+2x
\end{matrix}\bigg|\,z\bigg]\bigg)\bigg|_{x=0}
=\frac{d}{dx}\bigg(
(1-z)^a{}_2F_1\bigg[\begin{matrix}
-a&1-\beta+x\\
&1+2x
\end{matrix}\bigg|\,\frac{z}{z-1}\bigg]\bigg)\bigg|_{x=0},
$$
by noting that
$$
\frac{d}{d x}\bigg(\frac{(\beta+x)_k}{(1+2x)_k}\bigg)\bigg|_{x=0}=\frac{(\beta)_k}{(1)_k}\cdot\bigg(\sum_{j=0}^{k-1}\frac{1}{\beta+j}-\sum_{j=0}^{k-1}\frac{2}{1+j}\bigg).$$
\end{proof}
In particular, substituting $\beta=1/2$ in (\ref{apkbetapkzkt}), we get
\begin{corollary}\label{apk12pkzk} Suppose that $0\leq a\leq p-1$. Then
\begin{equation}\label{apk12pkzkcong}
\sum_{k=0}^a\frac{(-a-p)_{p+k}(\frac12)_{p+k}}{(1)_{p+k}^2}\cdot z^k\equiv
(1-z)^a\sum_{k=0}^a\frac{(-a-p)_{p+k}(\frac12)_{p+k}}{(1)_{p+k}^2}\cdot\frac{z^k}{(z-1)^k}\pmod{p^2}.
\end{equation}
\end{corollary}

We are ready to prove Theorem \ref{alpha12zalpha12zz1}. Let $a=\langle-\alpha\rangle_p$ and $s=-(\alpha+a)/p$. 
Let
$$
\Psi_1(x)={}_2F_1\bigg[\begin{matrix}
-a-x&\frac12\\
&1
\end{matrix}\bigg|\,z\bigg]_{p-1},\qquad
\Psi_2(x)=(1-z)^{a}{}_2F_1\bigg[\begin{matrix}
-a-x&\frac12\\
&1
\end{matrix}\bigg|\,\frac{z}{z-1}\bigg]_{p-1},
$$
and let
$$
\Phi(x)=\bigg(1-\frac z2\bigg)^{a}\cdot{}_2F_1\bigg[\begin{matrix}
-\frac12(a+x)&\frac12-\frac12(a+x)\\
&1
\end{matrix}\bigg|\,\frac{z^2}{(z-2)^2}\bigg]_{p-1}.
$$
Since either $a/2$ and $a/2-1/2$ is an non-negative integer, by (\ref{alpha12zalpha12zz1e}),
\begin{align*}
\Psi_1(0)+(1-z)\Psi_2(0)=&{}_2F_1\bigg[\begin{matrix}
-a&\frac12\\
&1
\end{matrix}\bigg|\,z\bigg]+(1-z)^{a+1}{}_2F_1\bigg[\begin{matrix}
-a&\frac12\\&1
\end{matrix}\bigg|\,\frac{z}{z-1}\bigg]\\
=&2\bigg(1-\frac z2\bigg)^{a+1}\cdot{}_2F_1\bigg[\begin{matrix}
-\frac12a&\frac12-\frac12a\\
&1
\end{matrix}\bigg|\,\frac{z^2}{(z-2)^2}\bigg]=(2-z)\cdot\Phi(0).
\end{align*}
On the other hand, by Lemma \ref{polynomialzsp}, we have
\begin{align*}
&
\big(\Psi_1(sp)-\Psi_1(0)\big)+(1-z)^{1+s(p-1)}\cdot\big(\Psi_2(sp)-\Psi_2(0)\big)\notag\\
\equiv&s\big(\Psi_1(p)-\Psi_1(0)\big)+s\big((1-z)^{p}\cdot\Psi_2(p)-(1-z)\cdot\Psi_2(0)\big)
\pmod{p^2}
\end{align*}
and
\begin{align*}
\bigg(1-\frac z2\bigg)^{1+s(p-1)}\cdot\Phi(sp)-\bigg(1-\frac z2\bigg)\cdot\Phi(0)
\equiv s\bigg(1-\frac z2\bigg)^{p}\cdot\Phi(p)-s\bigg(1-\frac z2\bigg)\cdot\Phi(0)
\pmod{p^2}.
\end{align*}
It suffices to show that
$$
\Psi_1(p)+(1-z)^{p}\cdot\Psi_2(p)\equiv 2\bigg(1-\frac z2\bigg)^{p}\cdot\Phi(p)\pmod{p^2}.
$$
Note that either $(a+p)/2$ or $(a+p-1)/2$ is an integer lying in $[0,p-1]$ now. Since
$$
{}_2F_1\bigg[\begin{matrix}
-a-p&\frac12\\
&1
\end{matrix}\bigg|\,z\bigg]=
{}_2F_1\bigg[\begin{matrix}
-a-p&\frac12\\
&1
\end{matrix}\bigg|\,z\bigg]_{p-1}+\sim_{k=0}^a\frac{(-a-p)_{p+k}(\frac12)_{p+k}}{(1)_{p+k}^2}\cdot z^{p+k},
$$
we have
\begin{align*}
2\bigg(1-\frac z2\bigg)^{p}\cdot\Phi(p)=
&
2\bigg(1-\frac z2\bigg)^{a+p}\cdot{}_2F_1\bigg[\begin{matrix}
-\frac12(a+p)&\frac12-\frac12(a+p)\\
&1
\end{matrix}\bigg|\,\frac{z^2}{(z-2)^2}\bigg]\\
=&{}_2F_1\bigg[\begin{matrix}
-a-p&\frac12\\
&1
\end{matrix}\bigg|\,z\bigg]+(1-z)^{a+p}{}_2F_1\bigg[\begin{matrix}
-a-p&\frac12\\&1
\end{matrix}\bigg|\,\frac{z}{z-1}\bigg]\\
\equiv&\Psi_1(p)+(1-z)^p\Psi_2(p)\pmod{p^2},
\end{align*}
where the last step follows from (\ref{apk12pkzkcong}).\qed

\section{The $p$-adic quadratic ${}_2F_1$ transformation IV: $z\to 4\sqrt{z}/(1+\sqrt{z})^2$}
\label{sectionq2F1z4sz1sz2}
\setcounter{lemma}{0}
\setcounter{theorem}{0}
\setcounter{corollary}{0}
\setcounter{remark}{0}
\setcounter{equation}{0}
\setcounter{conjecture}{0}

Let us turn to the proof of Theorem \ref{alphaalphaz2alpha124zz12}.
Let $a=\langle-\alpha\rangle_p$. Let
$$
\Psi_1(x)={}_2F_1\bigg[\begin{matrix}
-a-x&-a-x\\
&1
\end{matrix}\bigg|\,z^2\bigg]_{p-1},\quad
\Psi_2(x)=z^{2a-1}{}_2F_1\bigg[\begin{matrix}
-a-x&-a-x\\
&1
\end{matrix}\bigg|\,\frac1{z^2}\bigg]_{p-1}
$$
and
$$
\Phi(x)=(1+z)^{2a-1}\cdot{}_2F_1\bigg[\begin{matrix}
-a-x&\frac12\\
&1
\end{matrix}\bigg|\,\frac{4z}{(1+z)^2}\bigg]_{p-1}.
$$
In order to prove Theorem \ref{alphaalphaz2alpha124zz12}, we need a lemma which is similar as Corollary \ref{apk12pkzk}.
\begin{lemma}\label{apk12pk4zz1k}
\begin{align}\label{apk12pk4zz1kc}
&2\sum_{k=0}^a\frac{(-a-p)_{p+k}(\frac12)_{p+k}}{(1)_{p+k}^2}\cdot \bigg(\frac{4z}{(1+z)^2}\bigg)^k\notag\\
\equiv&-\frac{\Phi(0)}{4^{p-1}(1+z)^{2a-1}}-\frac{p}{(1+z)^{2a}}\cdot\frac{\Psi_1'(0)+z\cdot\Psi_2'(0)}{2}\pmod{p^2}.
\end{align}
\end{lemma}
\begin{proof}
According to the discussions in the proof of Lemma \ref{apkbetapkzkt}, we know that
$$
\frac{(-a-p)_{p+k}(\frac12)_{p+k}}{(1)_{p+k}^2}\equiv
-\frac{(-a)_{k}(\frac12)_{k}}{2\cdot(1)_{k}^2}\bigg(1+p\bigg(H_{p-1-a}+H_{\frac{p-1}2}+\sum_{j=0}^{k-1}\frac{1}{j+\frac12}-2H_k\bigg)\bigg)\pmod{p^2}
$$
for any $0\leq k\leq a$.
Let
$$
\phi(x)={}_2F_1\bigg[\begin{matrix}
-a&\frac12+x\\
&1+2x
\end{matrix}\bigg|\,\frac{4z}{(1+z)^2}\bigg].
$$
Clearly for $0\leq k\leq a$,
$$
\frac{d}{dx}\bigg(\frac{(\frac12+x)_k}{(1+2x)_k}\bigg)\bigg|_{x=0}=
\frac{(\frac12+x)_k}{(1+2x)_k}\sum_{j=0}^{k-1}\bigg(\frac{1}{j+\frac12}-\frac{2}{j+1}\bigg).
$$
Furthermore, we have $H_{p-1-a}\equiv H_a\pmod{p}$ and
$$
H_{\frac{p-1}{2}}\equiv \frac{2(2^{1-p}-1)}{p}\equiv \frac{4^{1-p}-1}{p}\pmod{p},
$$
since
$$
2^p=2+\sum_{j=1}^{p-1}\frac{p}{j}\cdot\binom{p-1}{j-1}\equiv
2+p\cdot(H_{p-1}-H_{\frac{p-1}{2}})\equiv 2-p\cdot H_{\frac{p-1}{2}}\pmod{p^2}.
$$
Hence
\begin{align*}
\sum_{k=0}^a\frac{(-a-p)_{p+k}(\frac12)_{p+k}}{(1)_{p+k}^2}\cdot \bigg(\frac{4z}{(1+z)^2}\bigg)^k
\equiv-\frac{\phi(0)}{2}\cdot(4^{1-p}+p\cdot H_{a})-\frac{p}{2}\cdot\phi'(0)\pmod{p^2}.
\end{align*}
On the other hand, we have the following general form of (\ref{alphaalphaz2alpha124zz12e}) \cite[Eq. (3.1.11)]{AAR99}
\begin{equation}\label{alphabetaz2}
{}_2F_1\bigg[\begin{matrix}
\alpha&\beta\\
&\alpha-\beta+1
\end{matrix}\bigg|\,z^2\bigg]=(1+z)^{-2\alpha}{}_2F_1\bigg[\begin{matrix}
\alpha&\alpha-\beta+\frac12\\
&2\alpha-2\beta+1
\end{matrix}\bigg|\,\frac{4z}{(1+z)^2}\bigg].
\end{equation}
Hence
$$
\phi(x)=\frac1{(1+z)^{2a}}\cdot
{}_2F_1\bigg[\begin{matrix}
-a&-a-x\\
&1+x
\end{matrix}\bigg|\,z^2\bigg]=
\frac{z^{2a}}{(1+z)^{2a}}\cdot
{}_2F_1\bigg[\begin{matrix}
-a&-a-x\\
&1+x
\end{matrix}\bigg|\,\frac1{z^2}\bigg].
$$
Note that for any $0\leq k\leq a$, we have
$$
\frac{d}{dx}\bigg(\frac{(-a-x)_k}{(1+x)_k}\bigg)\bigg|_{x=0}=
\frac{(-a)_k}{(1)_k}\sum_{j=0}^{k-1}\bigg(\frac{1}{a-j}-\frac1{j+1}\bigg)=
\frac{(-a)_k}{(1)_k}\cdot(H_a-H_{a-k}-H_k).
$$
Also, clearly
$$
\frac{d}{dx}\bigg(\frac{(-a-x)_k}{(1)_k}\bigg)\bigg|_{x=0}=
\frac{(-a)_k}{(1)_k}\cdot(H_a-H_{a-k}).$$
So by (\ref{aak1akak1k}), we get
\begin{align*}
\phi(0)\cdot H_a+\phi'(0)
=&H_a\sum_{k=0}^a\frac{(-a)_k^2}{(1)_k^2}\cdot z^{2k}+\sum_{k=0}^a\frac{(-a)_k^2}{(1)_k^2}\cdot(H_a-H_{a-k}-H_k)\cdot z^{2k}\\
=&\sum_{k=0}^a\frac{(-a)_k^2}{(1)_k^2}\cdot(H_a-H_{a-k})\cdot z^{2k}+
\sum_{k=0}^a\frac{(-a)_{a-k}^2}{(1)_{a-k}^2}\cdot(H_a-H_{k})\cdot z^{2k}\\
=&\frac{d}{dx}\bigg(\sum_{k=0}^a\frac{(-a)_k(-a-x)_k}{(1)_k^2}\cdot z^{2k}+
z^{2a}\sum_{k=0}^a\frac{(-a)_{k}(-a-x)_k}{(1)_{k}^2}\cdot \frac1{z^{2k}}\bigg)\bigg|_{x=0}.
\end{align*}
Thus
\begin{align*}
&-2\sum_{k=0}^a\frac{(-a-p)_{p+k}(\frac12)_{p+k}}{(1)_{p+k}^2}\cdot \bigg(\frac{4z}{(1+z)^2}\bigg)^k\\
\equiv&4^{1-p}\phi(0)+\frac{p}{(1+z)^{2a}}\cdot\frac{d}{dx}\bigg(
{}_2F_1\bigg[\begin{matrix}
-a&-a-x\\
&1
\end{matrix}\bigg|\,z^2\bigg]+z^{2a}{}_2F_1\bigg[\begin{matrix}
-a&-a-x\\
&1
\end{matrix}\bigg|\,\frac1{z^2}\bigg]\bigg)\bigg|_{x=0}\\
\equiv&\frac{\Phi(0)}{4^{p-1}(1+z)^{2a-1}}+\frac{p}{(1+z)^{2a}}\cdot\frac{\Psi_1'(0)+z\cdot\Psi_2'(0)}{2}
\pmod{p^2}.
\end{align*}
\end{proof}
Let us return the proof of Theorem \ref{alphaalphaz2alpha124zz12}. By (\ref{alphaalphaz2alpha124zz12e}),
 we have $$\Psi_1(0)=z\cdot\Psi_2(0)=(1+z)\cdot\Phi(0),$$
whence
$$
\Psi_1(0)+z^2\cdot\Psi_2(0)=(1+z)^2\cdot\Phi(0).
$$
According to Lemma \ref{polynomialzsp}, it suffices to show that
$$
(1+z)^{2p}\cdot\Phi(p)\equiv\Psi_1(p)+z^{2p}\cdot\Psi_2(p)\pmod{p^2}.
$$
In view of (\ref{alphaalphaz2alpha124zz12e}), (\ref{apap1zapapzapap1z2F1c}) and (\ref{apk12pk4zz1kc}), we have
\begin{align*}
&{}_2F_1\bigg[\begin{matrix}
-a-p&\frac12\\
&1
\end{matrix}\bigg|\,\frac{4z}{(1+z)^2}\bigg]_{p-1}\\
=&{}_2F_1\bigg[\begin{matrix}
-a-p&\frac12\\
&1
\end{matrix}\bigg|\,\frac{4z}{(1+z)^2}\bigg]-\bigg(\frac{4z}{(1+z)^2}\bigg)^p\sum_{k=0}^a\frac{(-a-p)_{p+k}(\frac12)_{p+k}}{(1)_{p+k}^2}\cdot \bigg(\frac{4z}{(1+z)^2}\bigg)^k\\
\equiv&\frac{\Psi_1(p)+z^{2p+1}\Psi_2(p)}{(1+z)^{2a+2p}}+\frac{2z^p\cdot\Phi(0)}{(1+z)^{2p+2a-1}}+\frac{z^p\cdot p\big(\Psi_1'(0)+z\cdot\Psi_2'(0)\big)}{(1+z)^{2a+2p}}\pmod{p^2}.
\end{align*}
Clearly
\begin{align*}
z^p\cdot p\big(\Psi_1'(0)+z\cdot\Psi_2'(0)\big)\equiv
z\cdot \big(\Psi_1(p)-\Psi_1(0)\big)+z^{2p}\cdot \big(\Psi_2(p)-\Psi_2(0)\big)\pmod{p^2}.
\end{align*}
Therefore
\begin{align*}
(1+z)^{2p}\cdot\Phi(p)
\equiv&\Psi_1(p)+z^{2p}\cdot\Psi_2(p)+2z^{p}\cdot\Phi(0)-
\frac{z\cdot\Psi_1(0)+z^{2p}\cdot\Psi_2(0)}{1+z}\\
=&\Psi_1(p)+z^{2p}\cdot\Psi_2(p)-\frac{z(z^{p-1}-1)^2\cdot\Psi_1(0)}{1+z}\\
\equiv&\Psi_1(p)+z^{2p}\cdot\Psi_2(p)
\pmod{p^2}.
\end{align*}
\qed

\section{The $p$-adic quadratic ${}_2F_1$ transformation V: $z\to z^2/(4z-4)$}
\label{sectionq2F1zz24z4}
\setcounter{lemma}{0}
\setcounter{theorem}{0}
\setcounter{corollary}{0}
\setcounter{remark}{0}
\setcounter{equation}{0}
\setcounter{conjecture}{0}

In this section, we shall give the proof of 
Theorem \ref{alpha12zalpha12z1z}.
\begin{proof}[Proof of Theorem \ref{alpha12zalpha12z1z}]
Let $a=\langle-\alpha\rangle_{p}$.
Let
$$
\Psi_1(x,y)={}_2F_1\bigg[\begin{matrix}
-a+x&\frac12(1-y)\\
&1
\end{matrix}\bigg|\,z\bigg],$$
$$
\Psi_2(x,y)=(1-z)^{a}{}_2F_1\bigg[\begin{matrix}
-a+x&\frac12(1-y)\\
&1
\end{matrix}\bigg|\,\frac{z}{z-1}\bigg],
$$
and
$$
\Phi(x)=(1-z)^{\frac a2}\cdot{}_2F_1\bigg[\begin{matrix}
-\frac12(a-x)&\frac12+\frac12(a-x)\\
&1
\end{matrix}\bigg|\,\frac{z^2}{4z-4}\bigg].
$$
In view of (\ref{alpha12zalpha12z1ze}),
$$
\Psi_1(0,0)=\Psi_2(0,0)=(1-z)^{\frac a2}\cdot{}_2F_1\bigg[\begin{matrix}
-\frac12a&\frac12+\frac12a\\
&1
\end{matrix}\bigg|\,\frac{z^2}{4z-4}\bigg]=\Phi(0),
$$
i.e., $$\Psi_1(0,0)+\Psi_2(0,0)=2\Phi(0).$$
Furthermore, applying (\ref{alpha12zalpha12z1ze}) reversely,
we get
\begin{align}2\Phi(p)=&2(1-z)^{\frac a2}\cdot{}_2F_1\bigg[\begin{matrix}\frac12+\frac12(p-1-a)&-\frac12(p-1-a)\\
&1
\end{matrix}\bigg|\,\frac{z^2}{4z-4}\bigg]_{p-1}\notag\\
=&(1-z)^{a-\frac{p-1}{2}}\cdot\big(\Omega_1(0)+\Omega_2(0)\big),
\end{align}
where
$$
\Omega_1(y)={}_2F_1\bigg[\begin{matrix}
a+1-p&\frac12(1-y)\\
&1
\end{matrix}\bigg|\,z\bigg]_{p-1}$$
and
$$ 
\Omega_2(y)=(1-z)^{p-1-a}{}_2F_1\bigg[\begin{matrix}
a+1-p&\frac12(1-y)\\
&1
\end{matrix}\bigg|\,\frac{z}{z-1}\bigg]_{p-1}.
$$

By (\ref{alphaalphazz1e}), clearly
$$
\Omega_1(y)=\Omega_2(-y).
$$
Therefore
\begin{align*}
\Omega_1(0)-\Omega_1(p)\equiv -p\cdot\Omega_1'(0)=p\cdot\Omega_2'(0)
\equiv\Omega_2(p)-\Omega_2(0)\pmod{p^2},
\end{align*}
i.e.,
$$
\Omega_1(p)+\Omega_2(p)\equiv\Omega_1(0)+\Omega_2(0)\pmod{p^2}.
$$
It also follows from (\ref{alphaalphazz1e}) that
\begin{align*}
\Omega_1(p)=&{}_2F_1\bigg[\begin{matrix}
\frac12(1-p)&a+1-p\\
&1
\end{matrix}\bigg|\,z\bigg]\\
=&
(1-z)^{\frac{p-1}{2}}\cdot{}_2F_1\bigg[\begin{matrix}
\frac12(1-p)&p-a\\
&1
\end{matrix}\bigg|\,\frac{z}{z-1}\bigg]=(1-z)^{\frac{p-1}{2}-a}\cdot\Psi_2(p,p).
\end{align*}
Similarly, we have $\Omega_2(p)=(1-z)^{\frac{p-1}{2}-a}\cdot\Psi_1(p,p)$, too. On the other hand, clearly
\begin{align*}
\Psi_i(p,p)-\Psi_i(0,0)\equiv p\cdot\frac{\partial\Psi_i(x,0)}{\partial x}\bigg|_{x=0}+p\cdot
\frac{\partial\Psi_i(0,y)}{\partial y}\bigg|_{y=0}\pmod{p^2},\qquad i=1,2.
\end{align*}
Note that by (\ref{alphaalphazz1e}),
\begin{align*}
\Psi_1(0,y)={}_2F_1\bigg[\begin{matrix}
-a&\frac12(1-y)\\
&1
\end{matrix}\bigg|\,z\bigg]=(1-z)^a{}_2F_1\bigg[\begin{matrix}
-a&\frac12(1+y)\\
&1
\end{matrix}\bigg|\,\frac z{z-1}\bigg]=\Psi_2(0,-y),
\end{align*}
whence
$$
\frac{\partial\Psi_1(0,y)}{\partial y}\bigg|_{y=0}+\frac{\partial\Psi_2(0,y)}{\partial y}\bigg|_{y=0}=0.
$$
Thus
\begin{align*}
2\Phi(p)\equiv&(1-z)^{a-\frac{p-1}{2}}\cdot\big(\Omega_1(p)+\Omega_2(p)\big)
\equiv\Psi_1(p,p)+\Psi_2(p,p)\\
\equiv&\sum_{i=1}^2\bigg(\Psi_i(0,0)+
p\cdot\frac{\partial\Psi_i(x,0)}{\partial x}\bigg|_{x=0}+p\cdot
\frac{\partial\Psi_i(0,y)}{\partial y}\bigg|_{y=0}\bigg)\\
=&\sum_{i=1}^2\bigg(\Psi_i(0,0)+p\cdot\frac{\partial\Psi_i(x,0)}{\partial x}\bigg|_{x=0}\bigg)
\equiv\Psi_1(p,0)+\Psi_2(p,0)\pmod{p^2}.
\end{align*}
Letting $s=(\alpha+a)/p$, we get
$$
2\Phi(sp)\equiv \Psi_1(sp,0)+\Psi_2(sp,0)\pmod{p^2}.
$$
\end{proof}
Let us see an application of Theorem \ref{alpha12zalpha12z1z}.
\begin{corollary}
Suppose that $\alpha\in\Z_p$ with $\langle-\alpha\rangle_p$ is even. Then
\begin{align}\label{alpha1222F1}
{}_2F_1\bigg[\begin{matrix}
\alpha&\frac12\\
&1
\end{matrix}\bigg|\,2\bigg]_{p-1}\equiv
(-1)^{1+\frac12\langle-\alpha\rangle_p}\cdot\frac{\Gamma_p(\frac12)}{\Gamma_p(\frac12+\frac12\alpha)\Gamma_p(1-\frac12\alpha)}\pmod{p^2}.
\end{align}
\end{corollary}
For example, for any prime $p\equiv1\pmod{3}$,
$$
{}_2F_1\bigg[\begin{matrix}
\frac13&\frac12\\
&1
\end{matrix}\bigg|\,2\bigg]_{p-1}\equiv (-1)^{\frac{p+5}{6}}\cdot\frac{\Gamma_p(\frac12)}{\Gamma_p(\frac23)\Gamma_p(\frac56)}\pmod{p^2}.
$$
\begin{proof}
According to (\ref{lambdalambdas}), we have
$$
(-1)^{\lambda_p(\alpha)}\equiv(-1)^{-\langle-\alpha\rangle_p}=1\pmod{p^2}
$$ 
since $\langle-\alpha\rangle_p$ is even now.
So substituting $z=2$ in Theorem \ref{alpha12zalpha12z1z} and applying Theorem \ref{alphabeta11}, we get
\begin{align*}
{}_2F_1\bigg[\begin{matrix}
\alpha&\frac12\\
&1
\end{matrix}\bigg|\,2\bigg]_{p-1}\equiv&
(-1)^{\frac12\langle-\alpha\rangle_p}{}_2F_1\bigg[\begin{matrix}
\frac12\alpha&\frac12-\frac12\alpha\\
&1
\end{matrix}\bigg|\,1\bigg]_{p-1}\\
\equiv&\frac{(-1)^{1+\frac12\langle-\alpha\rangle_p}\cdot\Gamma_p(\frac12)}{\Gamma_p(\frac12+\frac12\alpha)\Gamma_p(1-\frac12\alpha)}\pmod{p^2}.
\end{align*}
\end{proof}
Substituting $\alpha=1/2$ and $z=2$ in (\ref{alphaz4z1zc}), we have
\begin{equation}
{}_2F_1\bigg[\begin{matrix}
\frac14&\frac14\\
&1
\end{matrix}\bigg|\,-8\bigg]_{p-1}\equiv{}_2F_1\bigg[\begin{matrix}
\frac12&\frac12\\
&1
\end{matrix}\bigg|\,2\bigg]_{p-1}\equiv (-1)^{\frac{p+3}{4}}\cdot\frac{\Gamma_p(\frac12)}{\Gamma_p(\frac34)^2}\pmod{p^2}.
\end{equation}
for prime $p\equiv 1\pmod{4}$.
Furthermore, by setting $\alpha=1/2$ and $z=2$ in (\ref{alpha2alpha214zz1c}) and using (\ref{2F1alpha1alpha12}), we obtain that for prime $p\equiv 1\pmod{4}$,
\begin{align}\label{2F1143489}
{}_2F_1\bigg[\begin{matrix}
\frac14&\frac34\\
&1
\end{matrix}\bigg|\,\frac89\bigg]_{p-1}
\equiv&\frac{1}{3^{1-\lambda_p(\frac12)}}\bigg(
{}_2F_1\bigg[\begin{matrix}
\frac12&\frac12\\
&1
\end{matrix}\bigg|\,2\bigg]_{p-1}+2^{1-\lambda_p(\frac12)}{}_2F_1\bigg[\begin{matrix}
\frac12&\frac12\\
&1
\end{matrix}\bigg|\,\frac12\bigg]_{p-1}
\bigg)\notag\\
\equiv&-\jacob{6}{p}\cdot\frac{\Gamma_p(\frac12)}{\Gamma_p(\frac34)^2}\pmod{p^2}.
\end{align}
Combining (\ref{2F1143489}) with Theorem \ref{alpha1alphaz124z1z}, we also have
\begin{align}\label{3F21434123281}
{}_2F_1\bigg[\begin{matrix}
\frac14&\frac34&\frac12\\
&1&1
\end{matrix}\bigg|\,\frac{32}{81}\bigg]_{p-1}
\equiv\frac{\Gamma_p(\frac12)^2}{\Gamma_p(\frac34)^4}\pmod{p^2}
\end{align}
for $p\equiv 1\pmod{4}$.

\section{The $p$-adic linear ${}_2F_1$ transformation.}
\label{sectionql2F1}
\setcounter{lemma}{0}
\setcounter{theorem}{0}
\setcounter{corollary}{0}
\setcounter{remark}{0}
\setcounter{equation}{0}
\setcounter{conjecture}{0}

Let us consider Theorem \ref{alphaalphazz1}, i.e., the $p$-adic analogue of the linear transformation (\ref{Euleralphabeta1}).
Let $a=\langle-\alpha\rangle_p$. Let
$$
\Psi_1(x)={}_2F_1\bigg[\begin{matrix}
-a-x&-a-x\\
&1
\end{matrix}\bigg|\,z\bigg]_{p-1},\quad
\Psi_2(x)=z^a{}_2F_1\bigg[\begin{matrix}
-a-x&-a-x\\
&1
\end{matrix}\bigg|\,\frac1z\bigg]_{p-1},$$
and let
$$
\Phi(x)=(1-z)^a{}_2F_1\bigg[\begin{matrix}
-a-x&1+a+x\\
&1
\end{matrix}\bigg|\,\frac z{z-1}\bigg]_{p-1}.
$$
Then $\Psi_1(0)=\Phi(0)$ by (\ref{alphaalphazz1e}). Moreover, by (\ref{aak1akak1k}), we also have
$$
\Psi_1(0)=\sum_{k=0}^a\frac{(-a)_k^2}{(1)_k^2}\cdot z^k
=\sum_{k=0}^a\frac{(-a)_{a-k}^2}{(1)_{a-k}^2}\cdot z^{k}=z^a
\sum_{k=0}^a\frac{(-a)_{k}^2}{(1)_{k}^2}\cdot \frac1{z^{k}}=\Psi_2(0).
$$
So
\begin{align*}
\Psi_1(0)-z\cdot \Psi_2(0)=(1-z)\cdot\Phi(0).
\end{align*}

Clearly
\begin{align}\label{dx1za2F1ax1a1}
&\frac{d}{d x}\bigg((1-z)^a{}_2F_1\bigg[\begin{matrix}
-a-x&1+a\\
&1
\end{matrix}\bigg|\,\frac{z}{z-1}\bigg]_{p-1}\bigg)\bigg|_{x=0}\notag\\
\equiv&
\frac{d}{d x}\bigg((1-z)^a{}_2F_1\bigg[\begin{matrix}
-a-x&1+a-p\\
&1
\end{matrix}\bigg|\,\frac{z}{z-1}\bigg]\bigg)\bigg|_{x=0}\notag\\
=&\frac{(1-z)^a}p\bigg({}_2F_1\bigg[\begin{matrix}
-a-p&1+a-p\\
&1
\end{matrix}\bigg|\,\frac{z}{z-1}\bigg]-{}_2F_1\bigg[\begin{matrix}
-a&1+a-p\\
&1
\end{matrix}\bigg|\,\frac{z}{z-1}\bigg]\bigg)\notag\\
=&\frac{1}p\bigg(\frac1{(1-z)^p}\cdot{}_2F_1\bigg[\begin{matrix}
-a-p&p-a\\
&1
\end{matrix}\bigg|\,z\bigg]-{}_2F_1\bigg[\begin{matrix}
-a&p-a\\
&1
\end{matrix}\bigg|\,z\bigg]\bigg)\pmod{p}.
\end{align}
For $0\leq k\leq p-1$, in view of (\ref{padicgammaderivative}), we have
\begin{align*}
(-a-p)_k(p-a)_k=\frac{\Gamma_p(-a+k-p)}{\Gamma_p(-a-p)}\cdot\frac{\Gamma_p(-a+k+p)}{\Gamma_p(-a+p)}
\equiv\frac{\Gamma_p(-a+k)^2}{\Gamma_p(-a)^2}=(-a)_k^2\pmod{p^2}.
\end{align*}
And if $0\leq k\leq a$,
\begin{align*}
\frac{(-a-p)_{p+k}(p-a)_{p+k}}{(1)_{p+k}^2}=&\frac{\Gamma_p(-a+k)}{\Gamma_p(-a-p)}\cdot\frac{\Gamma_p(2p-a+k)}{\Gamma_p(p-a)}\cdot\frac{\Gamma_p(1)^2}{\Gamma_p(p+k+1)^2}\cdot\frac{(-p)\cdot p}{p^2}\notag\\
\equiv&-\frac{(-a)_k^2}{(1)_{k}^2}\cdot\big(1+2p(H_{p-1-a+k}-H_k)\big)\pmod{p^2}.
\end{align*}
Thus
\begin{equation}
{}_2F_1\bigg[\begin{matrix}
-a-p&p-a\\
&1
\end{matrix}\bigg|\,z\bigg]\equiv\sum_{k=0}^a\frac{(-a)_k^2}{(1)_k^2}\cdot z^k-
z^p\sum_{k=0}^a\frac{(-a)_k^2}{(1)_k^2}\cdot\big(1+2p(H_{a-k}-H_k)\big)\cdot z^k\pmod{p^2}.
\end{equation}
On the other hand, we have
$$
\frac{d}{d x}\bigg(\frac{(-a-x)_k}{(1)_k}\bigg)\bigg|_{x=0}=\frac{(-a)_k}{(1)_k}\sum_{j=0}^{k-1}\frac{1}{a-j}=\frac{(-a)_k}{(1)_k}\cdot(H_{a}-H_{a-k}).
$$
It follows that
\begin{align*}
\frac{(-a)_k^2}{(1)_{k}^2}\cdot(H_{a-k}-H_k)=\frac{d}{d x}\bigg(\frac{(-a)_{a-k}(-a-x)_{a-k}}{(1)_{a-k}^2}-\frac{(-a)_k(-a-x)_{k}}{(1)_k^2}\bigg)\bigg|_{x=0}\pmod{p}.
\end{align*}
In view of (\ref{dx1za2F1ax1a1}), we get
\begin{align}\label{dx1za2F1ax1a1modp}
&\frac{d}{d x}\bigg((1-z)^a{}_2F_1\bigg[\begin{matrix}
-a-x&1+a\\
&1
\end{matrix}\bigg|\,\frac{z}{z-1}\bigg]_{p-1}\bigg)\bigg|_{x=0}\notag\\
\equiv&
\frac{1}p\bigg(\frac{1-z^p}{(1-z)^p}\cdot{}_2F_1\bigg[\begin{matrix}
-a&-a\\
&1
\end{matrix}\bigg|\,z\bigg]-{}_2F_1\bigg[\begin{matrix}
-a&p-a\\
&1
\end{matrix}\bigg|\,z\bigg]\bigg)\notag\\
&-\frac{2z^p}{(1-z)^p}\cdot\frac{d}{d x}\bigg(\sum_{k=0}^a\frac{(-a)_{a-k}(-a-x)_{a-k}}{(1)_{a-k}^2}\cdot z^k-\sum_{k=0}^a\frac{(-a)_k(-a-x)_{k}}{(1)_k^2}\cdot z^k\bigg)\bigg|_{x=0}\pmod{p}.
\end{align}

Note that
\begin{align}
\Psi_1'(0)=&2\cdot\frac{d}{dx}\bigg({}_2F_1\bigg[\begin{matrix}
-a& -a-x\\
&1
\end{matrix}\bigg|\,z\bigg]_{p-1}\bigg)\bigg|_{x=0}\notag\\=&
2\cdot\frac{d}{dx}\bigg((1-z)^a{}_2F_1\bigg[\begin{matrix}
-a& 1+a+x\\
&1
\end{matrix}\bigg|\,\frac z{z-1}\bigg]_{p-1}\bigg)\bigg|_{x=0}
\end{align}
and
\begin{align}
\Psi_2'(0)=&2z^a\cdot\frac{d}{d x}\bigg({}_2F_1\bigg[\begin{matrix}
-a&-a-x\\
&1
\end{matrix}\bigg|\,\frac1z\bigg]\bigg)\bigg|_{x=0}\notag\\
=&2z^a\cdot\frac{d}{d x}\bigg(\sum_{k=0}^a\frac{(-a)_{a-k}(-a-x)_{a-k}}{(1)_{a-k}^2}\cdot\frac{1}{z^{a-k}}\bigg)\bigg|_{x=0}.
\end{align}
Hence by (\ref{dx1za2F1ax1a1modp}),
\begin{align*}
&\Phi'(0)=(1-z)^a\cdot\frac{d}{d x}\bigg({}_2F_1\bigg[\begin{matrix}
-a&1+a+x\\
&1
\end{matrix}\bigg|\,\frac{z}{z-1}\bigg]_{p-1}+{}_2F_1\bigg[\begin{matrix}
-a-x&1+a\\
&1
\end{matrix}\bigg|\,\frac{z}{z-1}\bigg]_{p-1}\bigg)\bigg|_{x=0}\\
\equiv&\frac{\Psi_1'(0)}{2}+\frac{1}p\bigg({}_2F_1\bigg[\begin{matrix}
-a&-a\\
&1
\end{matrix}\bigg|\,z\bigg]-{}_2F_1\bigg[\begin{matrix}
-a&p-a\\
&1
\end{matrix}\bigg|\,z\bigg]\bigg)+\frac{1-z^p-(1-z)^p}{p\cdot(1-z)^p}\cdot{}_2F_1\bigg[\begin{matrix}
-a&-a\\
&1
\end{matrix}\bigg|\,z\bigg]\\
&+\frac{2z^p}{(1-z)^p}\cdot\frac{d}{d x}\bigg(\sum_{k=0}^a\frac{(-a)_k(-a-x)_{k}}{(1)_k^2}\cdot z^k-\sum_{k=0}^a\frac{(-a)_{a-k}(-a-x)_{a-k}}{(1)_{a-k}^2}\cdot z^k\bigg)\bigg|_{x=0}\\
\equiv&\bigg(1+\frac{z^p}{(1-z)^p}\bigg)\cdot\Psi_1'(0)-\frac{z^p}{(1-z)^p}\cdot\Psi_2'(0)+\frac{1-z^p-(1-z)^p}{p\cdot(1-z)^p}\cdot{}_2F_1\bigg[\begin{matrix}
-a&-a\\
&1
\end{matrix}\bigg|\,z\bigg]\pmod{p}.
\end{align*}
Thus
\begin{align}\label{alphaalphazz1phippsip}
&(1-z)^p\cdot\Phi(p)\equiv (1-z)^p\cdot p\Phi'(0)+(1-z)^p\cdot\Phi(0)\notag\\
\equiv&p\big(\Psi_1'(0)-z^p\cdot\Psi_2'(0)\big)+\big(1-z^p-(1-z)^p\big)\cdot\Psi_1(0)+(1-z)^p\cdot\Phi(0)\notag\\
\equiv&\big(\Psi_1(p)-\Psi_1(0)\big)-z^p\cdot\big(\Psi_2(p)-\Psi_2(0)\big)+(1-z^p)\cdot\Psi_1(0)\notag\\
=&\Psi_1(p)-z^p\cdot\Psi_2(p)\pmod{p^2}.
\end{align}
By Lemma \ref{polynomialzsp}, Theorem \ref{alphaalphazz1} is concluded.\qed

With help of Theorem \ref{alphaalphazz1}, we can give another proof of Corollary \ref{2F1alphaalpham1}. In fact, letting $z=1$ in (\ref{alphaalphazz1c}) and noting that $(-1)^{-\lambda_p(\alpha)}\equiv(-1)^{\langle-\alpha\rangle_p}=1\pmod{p^2}$ now, we have
\begin{align*}
2{}_2F_1\bigg[\begin{matrix}
\alpha&\alpha\\
&1
\end{matrix}\bigg|\,-1\bigg]_{p-1}\equiv2^{1-\lambda_p(\alpha)}{}_2F_1\bigg[\begin{matrix}
\alpha&1-\alpha\\
&1
\end{matrix}\bigg|\,\frac12\bigg]_{p-1}
\equiv-\frac{2^{1-\lambda_p(\alpha)}\cdot\Gamma_p(\frac12)}{\Gamma_p(1-\frac12\alpha)\Gamma_p(\frac12+\frac12\alpha)}\pmod{p^2},
\end{align*}
where (\ref{2F1alpha1alpha12}) is used in the second step. And substituting $x=\alpha/2$ and $m=2$ in (\ref{padicGMF}), 
$$
\Gamma_p\bigg(\frac\alpha2\bigg)\Gamma_p\bigg(\frac\alpha2+\frac12\bigg)=2^{-\lambda_p(\alpha)}\Gamma_p(\alpha)\Gamma_p\bigg(\frac12\bigg),
$$
i.e.,
\begin{align*}
\frac{2^{-\lambda_p(\alpha)}\cdot\Gamma_p(\frac12)}{\Gamma_p(1-\frac12\alpha)\Gamma_p(\frac12+\frac12\alpha)}=
-\frac{2\Gamma_p(1+\frac12\alpha)}{\Gamma_p(1+\alpha)\Gamma_p(1-\frac12\alpha)}.
\end{align*}

\section{Truncated ${}_2F_1$ series and complex multiplication}
\label{section2F1cm}
\setcounter{lemma}{0}
\setcounter{theorem}{0}
\setcounter{corollary}{0}
\setcounter{remark}{0}
\setcounter{equation}{0}
\setcounter{conjecture}{0}

Consider the elliptic curve
$$
\sE:\ y^2=x(x^2+Ax+B),
$$
where $A,B\in\Z$. Let $\omega,\omega'$ be a basis of periods of $\sE$. Then $\sE$ has {\it complex multiplication} if and only if $\omega'/\omega$ belongs to some imaginary quadratic field $\Q(\sqrt{-d})$ where $d$ is a square-free positive integer. An elliptic curve having complex multiplication is often called {\it CM elliptic curve}.

In \cite{CoHa91}, Coster and Hamme established a connection between the congruences for the Legendre polynomials and the CM elliptic curves.
\begin{theorem}\label{CoHa}
Let $\sE:\ y^2=x(x^2+Ax+B)$ be a CM elliptic curve, where $A,B\in\Z$. Suppose that 
$\omega,\omega'$ form a basis of periods of $\sE$ such that

\medskip\noindent
(i) $\tau=\omega'/\omega\in \Q(\sqrt{-d})$ for some square-free positive integer $d$, and the imaginary part of $\tau$ is positive;

\medskip\noindent
(ii)  $A=3\sP(\frac12\omega)$ and $\sqrt{\Delta}=\sP(\frac12\omega'+\frac12\omega)-\sP(\frac12\omega')
$, where $\Delta=A^2-4B$ and $\sP$ denotes the Weierstrass function.

\medskip\noindent
Suppose that $p$ is an odd prime with $\big(\frac{-d}p\big)=1$, and $\pi\bar{\pi}=p$ where $\pi\in\Q(\sqrt{-d})$. Write $\pi=u_1+v_1\tau$ and $\pi\tau=u_2+v_2\tau$, where $u_1,v_1,u_2,v_2\in\Z$ and $v_1$ is even. Then for each odd $m\geq $1,
\begin{equation}\label{CHLegendre}
P_{\frac{mp-1}2}\bigg(\frac A{\sqrt{\Delta}}\bigg)\equiv \bi^{-u_2v_2+v_2+p-2}\cdot \bar{\pi}\cdot P_{\frac{m-1}{2}}\bigg(\frac A{\sqrt{\Delta}}\bigg)\pmod{\pi^2},
\end{equation}
where $\bi=\sqrt{-1}$.
\end{theorem}
Substituting $z=(1-A/\sqrt{\Delta})/2$ and $\alpha=1/2$ in (\ref{F2114144z1zLegendre}), we get
\begin{equation}
{}_2F_1\bigg[\begin{matrix}
\frac14&\frac14\\
&1
\end{matrix}\bigg|\,-\frac{4B}{\Delta}\bigg]_{p-1}\equiv \begin{cases}\bi^{-u_2v_2+v_2+p-2}\cdot \bar{\pi}\pmod{\pi^2},&\text{if }p\equiv1\pmod{4},\\
\bi^{3(-u_2v_2+v_2+p-2)}\cdot\frac{A}{\sqrt{\Delta}}\cdot \bar{\pi}\pmod{\pi^2},&\text{if }p\equiv3\pmod{4}.\end{cases}
\end{equation}
For example, for each prime $p\equiv 1\pmod{3}$, we can uniquely write $p=a^2+3b^2$, where one of $a,b\in\Z$ is positive even integer and $a+b\equiv1\pmod{4}$. According to Theorem \ref{CoHa} and Tables II and III of \cite{CoHa91}, we know that
$$
P_{\frac{p-1}{2}}(\sqrt{-3})\equiv a+b\sqrt{-3}\pmod{p^2},
$$
where $\sqrt{-3}\in\Z_p$ with $\sqrt{-3}\equiv a/b\pmod{p}$.
It follows from (\ref{F2114144z1zLegendre}) that
\begin{equation}\label{2F114144}
{}_2F_1\bigg[\begin{matrix}
\frac14&\frac14\\
&1
\end{matrix}\bigg|\,4\bigg]_{p-1}\equiv \begin{cases}a+b\sqrt{-3}\pmod{p^2},&\text{if }p\equiv1\pmod{4},\\
a\sqrt{-3}-3b\pmod{p^2},&\text{if }p\equiv3\pmod{4}.\end{cases}
\end{equation}

Furthermore, we have the following theorem.
\begin{theorem}\label{2F114141434}
Under the assumptions of Theorem \ref{CoHa}, letting $z=-4B/\Delta$, we have
\begin{align}
{}_2F_1\bigg[\begin{matrix}
\frac14&\frac14\\
&1
\end{matrix}\bigg|\,z\bigg]_{p-1}\equiv& z^{-\lambda_p(\frac14)}{}_2F_1\bigg[\begin{matrix}
\frac14&\frac14\\
&1
\end{matrix}\bigg|\,\frac1z\bigg]_{p-1}\pmod{p^2}\label{F211414zF2114141z}\\
\equiv&(1-z)^{-\lambda_p(\frac14)}{}_2F_1\bigg[\begin{matrix}
\frac14&\frac34\\
&1
\end{matrix}\bigg|\,\frac z{z-1}\bigg]_{p-1}\pmod{p^2}.\label{F211414zF211434zz1}
\end{align}
\end{theorem}
\begin{proof}
For the CM elliptic curve $\sE:\,y^2=x(x^2+Ax+B)$, let
$$
\sF:\ y^2=x(x^2-2Ax+\Delta).
$$ 
Then we know (cf. \cite[pp. 91-96]{Hu87}) that $\sE$ is isogenous to $\sF$. Furthermore, if $\upsilon,\upsilon'$ form the basis corresponding to $\sF$, then we have  $\upsilon'/\upsilon=2\omega'/\omega$. So $\upsilon'/\upsilon\in\Q(\sqrt{-d})$, i.e., $\sF$ is also CM.

Note that
$$
-\frac{4\Delta}{(-2A)^2-4\Delta}=-\frac{\Delta}{4B}.
$$
Hence in view of Theorem  \ref{CoHa}, recalling that $z=-4B/\Delta$, we obtain that
$$
{}_2F_1\bigg[\begin{matrix}
\frac14&\frac14\\
&1
\end{matrix}\bigg|\,z\bigg]_{p-1}\equiv \delta\cdot {}_2F_1\bigg[\begin{matrix}
\frac14&\frac14\\
&1
\end{matrix}\bigg|\,\frac1z\bigg]_{p-1}\pmod{\pi^2}
$$
where $\delta \in\{\pm\bi,\pm1\}$.

Assume that $p\equiv 1\pmod{4}$. Then we have
$$
{}_2F_1\bigg[\begin{matrix}
\frac14&\frac14\\
&1
\end{matrix}\bigg|\,z\bigg]_{p-1}\equiv\sum_{k=0}^{\frac{p-1}4}\binom{\frac{p-1}{4}}{k}^2z^k=\sum_{k=0}^{\frac{p-1}4}\binom{\frac{p-1}{4}}{k}^2z^{\frac{1}{4}(p-1)-k}\equiv
z^{\frac{p-1}{4}}{}_2F_1\bigg[\begin{matrix}
\frac14&\frac14\\
&1
\end{matrix}\bigg|\,\frac1z\bigg]_{p-1}\pmod{p}.
$$
So we must have $z^{\frac{p-1}{4}}\equiv \delta\pmod{\pi}$. It follows from (\ref{lambdalambdas}) that
$$
z^{-\lambda_p(\frac14)}\equiv z^{\frac{p-1}{p}\cdot(\frac{p^2-1}{4}-\frac{p-1}{4})+\frac{p-1}{4}}=(z^{\frac{p-1}{4}})^p\equiv \delta^p+p(z^{\frac{p-1}{4}}-\delta)\equiv\delta\pmod{\pi^2}.
$$
Thus (\ref{F211414zF2114141z}) is valid when $p\equiv 1\pmod{4}$.

Now assume that $p\equiv 3\pmod{4}$. We have
$$
P_{\frac{3p-1}{2}}\bigg(\frac{A}{\sqrt{\Delta}}\bigg)={}_2F_1\bigg[\begin{matrix}
-\frac{3p-1}{2}&\frac{3p+1}{2}\\
&1
\end{matrix}\bigg|\,\frac{\sqrt{\Delta}-A}{2\sqrt{\Delta}}\bigg] 
\equiv
{}_2F_1\bigg[\begin{matrix}
\frac14&\frac14\\
&1
\end{matrix}\bigg|\,z\bigg]_{p-1}\pmod{p^2}.
$$
According to Theorem \ref{CoHa}, 
$$
P_{\frac{3p-1}{2}}\bigg(\frac{A}{\sqrt{\Delta}}\bigg)\equiv \bi^{3(-u_2v_2+v_2+p-2)}\cdot \bar{\pi}\cdot \frac A{\sqrt{\Delta}}\pmod{p^2}.
$$
Hence 
\begin{align*}
{}_2F_1\bigg[\begin{matrix}
\frac14&\frac14\\
&1
\end{matrix}\bigg|\,z\bigg]_{p-1}\equiv& \delta\cdot \frac A{\sqrt{\Delta}}\cdot\bigg(-\frac{\sqrt{A^2-4\Delta}}{2A}\bigg)\cdot {}_2F_1\bigg[\begin{matrix}
\frac14&\frac14\\
&1
\end{matrix}\bigg|\,\frac1z\bigg]_{p-1}\\
\equiv&-\delta\sqrt{-z}\cdot {}_2F_1\bigg[\begin{matrix}
\frac14&\frac14\\
&1
\end{matrix}\bigg|\,\frac1z\bigg]_{p-1}\pmod{\pi^2},
\end{align*}
where $\delta\in\{\pm\bi,\pm 1\}$. Also,
\begin{align*}
{}_2F_1\bigg[\begin{matrix}
\frac14&\frac14\\
&1
\end{matrix}\bigg|\,z\bigg]_{p-1}\equiv
{}_2F_1\bigg[\begin{matrix}
-\frac{3p-1}4&-\frac{3p-1}4\\
&1
\end{matrix}\bigg|\,z\bigg]\equiv
z^{\frac{3p-1}4}{}_2F_1\bigg[\begin{matrix}
\frac14&\frac14\\
&1
\end{matrix}\bigg|\,\frac1z\bigg]_{p-1}\pmod{p}.
\end{align*}
So $z^{\frac{3p-1}4}\equiv z^{-\frac{p-3}{4}}\equiv -\bi\delta\cdot\sqrt{z}\pmod{\pi}$. Then
$$
z^{-\lambda_p(\frac14)}\equiv z^{\frac{p-1}{p}\cdot(\frac{p^2-1}{4}-\frac{3p-1}{4})+\frac{3p-1}{4}}=z^{\frac{p+1}{2}}\cdot (z^{-\frac{p-3}{4}})^{-p}\equiv
z^{\frac{p+1}{2}}\cdot (-\bi\delta\cdot\sqrt{z})^{-p}
=-\bi\delta\cdot\sqrt{z}\pmod{\pi^2}.
$$

Finally, clearly (\ref{F211414zF211434zz1}) is a consequence of (\ref{F211414zF2114141z}) and Theorem \ref{alphaalphazz1}.
\end{proof}
For example, it follows from (\ref{2F114144}) and Theorem \ref{2F114141434} that
\begin{equation}
{}_2F_1\bigg[\begin{matrix}
\frac14&\frac14\\
&1
\end{matrix}\bigg|\,\frac14\bigg]_{p-1}\equiv \begin{cases}\big(\frac{2}{p}\big)\cdot (a+b\sqrt{-3})\pmod{p^2},&\text{if }p\equiv1\pmod{4},\\
\big(\frac{2}{p}\big)\cdot 2(a\sqrt{-3}-3b)\pmod{p^2},&\text{if }p\equiv3\pmod{4},\end{cases}
\end{equation}
and
\begin{equation}\label{F211434m13ab}
{}_2F_1\bigg[\begin{matrix}
\frac14&\frac34\\
&1
\end{matrix}\bigg|\,-\frac13\bigg]_{p-1}\equiv \begin{cases}\big(\frac{\sqrt{-3}}{p}\big)\cdot (a+b\sqrt{-3})\pmod{p^2},&\text{if }p\equiv1\pmod{4},\\
-\big(\frac{\sqrt{-3}}{p}\big)\cdot 3(a+b\sqrt{-3})\pmod{p^2},&\text{if }p\equiv3\pmod{4},\end{cases}
\end{equation}

On the other hand, using an identity of Goursat, we can obtain an equivalent form of (\ref{F211434m13ab}), which unifies two cases $p\equiv 1$ and $3\pmod{4}$.
\begin{theorem}\label{F21number}
Suppose that $p\equiv 1\pmod{3}$. Then
\begin{equation}\label{F211434m13}
{}_2F_1\bigg[\begin{matrix}
\frac14&\frac34\\
&1
\end{matrix}\bigg|\,-\frac13\bigg]_{p-1}\equiv-\jacob{2}{p}\cdot\frac{3\Gamma_p(\frac43)}{2\Gamma_p(\frac32)\Gamma_p(\frac56)}\pmod{p^2}.
\end{equation}
\end{theorem}
With help of the Gross-Koblitz formula, it is not difficult to that (\ref{F211434m13ab}) and (\ref{F211434m13}) are equivalent.
\begin{proof}
In fact, we will prove a stronger result:
\begin{equation}
{}_2F_1\bigg[\begin{matrix}
\frac14-\frac{sp}4&\frac34-\frac{sp}{4}\\
&1+\frac{sp}{2}
\end{matrix}\bigg|\,-\frac13\bigg]_{p-1}\equiv\jacob{2}{p}^{s-1}\cdot\frac{8^{\frac{s(p-1)}2}}{9^{\frac{s(p-1)}2}}\cdot\frac{3\Gamma_p(\frac43)\Gamma_p(1+\frac{sp}2)}{2\Gamma_p(\frac32)\Gamma_p(\frac56+\frac{sp}2)}\pmod{p^2}
\end{equation}
for any $s\in\Z_p$.
We need the  hypergeometric identity 
\begin{equation}\label{alpha12alpha322alpha13}
{}_2F_1\bigg[\begin{matrix}
\alpha&\frac12+\alpha\\
&\frac32-2\alpha
\end{matrix}\bigg|\,-\frac13\bigg]=\frac{8^{-2\alpha}}{9^{-2\alpha}}\cdot\frac{\Gamma(\frac43)\Gamma(\frac32-2\alpha)}{\Gamma(\frac32)\Gamma(\frac43-2\alpha)},
\end{equation}
which can be deduced from Goursat's cubic transformation \cite[Eq. (119)]{Goursat1881} and (\ref{Gaussidentity}).
Let
$$
\Psi(x)={}_2F_1\bigg[\begin{matrix}
\frac14-\frac14x&\frac34-\frac14x\\
&1+\frac12x
\end{matrix}\bigg|\,-\frac13\bigg]_{p-1}.
$$
It follows from (\ref{alpha12alpha322alpha13}) that 
$$
\Psi(p)=
{}_2F_1\bigg[\begin{matrix}
\frac{1-p}4&\frac{3-p}{4}\\
&1+\frac{p}{2}
\end{matrix}\bigg|\,-\frac13\bigg]=
\frac{8^{\frac{p-1}{2}}}{9^{\frac{p-1}{2}}}\cdot\frac{\Gamma(\frac43)\Gamma(\frac32+\frac{p-1}2)}{\Gamma(\frac32)\Gamma(\frac43+\frac{p-1}2)}
$$
and
$$
\Psi(3p)=
{}_2F_1\bigg[\begin{matrix}
\frac{1-3p}4&\frac{3-3p}{4}\\
&1+\frac{3p}{2}
\end{matrix}\bigg|\,-\frac13\bigg]=
\frac{8^{\frac{3p-1}{2}}}{9^{\frac{3p-1}{2}}}\cdot\frac{\Gamma(\frac43)\Gamma(\frac32+\frac{3p-1}2)}{\Gamma(\frac32)\Gamma(\frac43+\frac{3p-1}2)}.
$$
By Lemma \ref{taylorexpansionrational},
$$
\Psi(sp)-\Psi(p)\equiv(s-1)p\cdot\Psi'(0)\equiv\frac{s-1}{2}\cdot\big(\Psi(3p)-\Psi(p)\big)\pmod{p^2}.
$$
Since $p\equiv 1\pmod{3}$, we have
$$
\frac{\Gamma(\frac43)}{\Gamma(\frac43+\frac{p-1}2)}\cdot\frac{\Gamma(\frac32+\frac{p-1}2)}{\Gamma(\frac32)}=\frac{3}{p}\cdot
\frac{\Gamma_p(\frac43)}{\Gamma_p(\frac43+\frac{p-1}2)}\cdot\frac{p}{2}\cdot
\frac{\Gamma_p(\frac32+\frac{p-1}2)}{\Gamma_p(\frac32)},$$
$$  
\frac{\Gamma(\frac43)}{\Gamma(\frac43+\frac{3p-1}2)}\cdot\frac{\Gamma(\frac32+\frac{3p-1}2)}{\Gamma(\frac32)}=\frac{9}{4p}\cdot
\frac{\Gamma_p(\frac43)}{\Gamma_p(\frac43+\frac{3p-1}2)}\cdot\frac{3p}{4}\cdot
\frac{\Gamma_p(\frac32+\frac{3p-1}2)}{\Gamma_p(\frac32)}.
$$
Let
$$
\Omega(x)=\frac{3\Gamma_p(\frac43)\Gamma_p(1+\frac{x}2)}{2\Gamma_p(\frac32)\Gamma_p(\frac56+\frac{x}2)}
$$
Viewing $\Omega(x)$ as a function over $\Z_p$, we also have
$$
\Omega(sp)-\Omega(p)=(s-1)p\cdot\Omega'(p)\equiv \frac{s-1}{2}\cdot\big(\Omega(3p)-\Omega(p)\big)
\pmod{p^2}.
$$
Let
$$
\kappa=\jacob{2}p\cdot\frac{8^{\frac{p-1}{2}}}{9^{\frac{p-1}{2}}}.
$$
Clearly $\kappa\equiv 1\pmod{p}$. So
\begin{align}\label{kappasOmegaspkappaOmegap}
\kappa^{s}\Omega(sp)-
\kappa\Omega(p)=&
(\kappa^{s-1}-1)\cdot\kappa\Omega(sp)+
\kappa\cdot\big(\Omega(sp)-\Omega(p)\big)\notag\\
\equiv&\frac{s-1}{2}\cdot\big(
(\kappa^{2}-1)\cdot\kappa\Omega(3p)+
\kappa\cdot(\Omega(3p)-\Omega(p))\big)\notag\\
=&\frac{s-1}{2}\cdot\big(\kappa^{3}\Omega(3p)-
\kappa\Omega(p)\big)\pmod{p^2}.
\end{align}
It follows that
\begin{align*}
\Psi(sp)\equiv&\Psi(p)+\frac{s-1}{2}\cdot\big(\Psi(3p)-\Psi(p)\big)\\
=&\jacob{2}p\bigg(\kappa\Omega(p)+\frac{s-1}{2}\cdot\big(\kappa^{3}\Omega(3p)-\kappa\Omega(p)\big)\bigg)
\equiv\jacob{2}p\cdot\kappa^{s}\Omega(sp)\pmod{p^2}.
\end{align*}
\end{proof}
Similarly, combining (\ref{alpha12alpha322alpha13}) with (\ref{Eulertransformation}), we have
\begin{equation}\label{alpha12alpha322alpha14}
{}_2F_1\bigg[\begin{matrix}
\alpha&1-3\alpha\\
&\frac32-2\alpha
\end{matrix}\bigg|\,\frac14\bigg]=\frac{16^{-\alpha}}{27^{-\alpha}}\cdot\frac{\Gamma(\frac43)\Gamma(\frac32-2\alpha)}{\Gamma(\frac32)\Gamma(\frac43-2\alpha)}.
\end{equation}
With help of (\ref{alpha12alpha322alpha14}),
we can get
\begin{equation}\label{F2114144}
{}_2F_1\bigg[\begin{matrix}
\frac14&\frac14\\
&1
\end{matrix}\bigg|\,4\bigg]_{p-1}
\equiv-\jacob{2}{p}\cdot\frac{3^{1-\lambda_p(\frac14)}\Gamma_p(\frac43)}{2\Gamma_p(\frac32)\Gamma_p(\frac56)}\pmod{p^2},
\end{equation}
for prime $p\equiv1\pmod{3}$. And by Theorem \ref{alphaalphazz1}, we obtain that for prime $p\equiv1\pmod{3}$,
\begin{equation}\label{F21141414}
{}_2F_1\bigg[\begin{matrix}
\frac14&\frac14\\
&1
\end{matrix}\bigg|\,\frac14\bigg]_{p-1}
\equiv-\jacob{2}{p}\cdot\bigg(\frac 34\bigg)^{1-\lambda_p(\frac14)}\cdot\frac{2\Gamma_p(\frac43)}{\Gamma_p(\frac32)\Gamma_p(\frac56)}\pmod{p^2}.
\end{equation}

Moreover, combining (\ref{F211434m13}) with Theorem \ref{alpha1alphaz124z1z}, we get
\begin{equation}
{}_2F_1\bigg[\begin{matrix}
\frac14&\frac34&\frac12\\
&1&1
\end{matrix}\bigg|\,-\frac{16}9\bigg]_{p-1}\equiv\frac{9\Gamma_p(\frac43)^2}{4\Gamma_p(\frac32)^2\Gamma_p(\frac56)^2}\pmod{p^2}
\end{equation}
for $p\equiv 1\pmod{3}$, i.e.,
\begin{equation}
\sum_{k=0}^{p-1}\binom{2k}{k}^2\binom{4k}{2k}\cdot\frac{1}{144^k}\equiv\frac{9\Gamma_p(\frac43)^2}{4\Gamma_p(\frac32)^2\Gamma_p(\frac56)^2}\pmod{p^2}.
\end{equation}

According to (\ref{SunZHalphaz}), (\ref{F211434m13}) is equivalent to
\begin{equation}
{}_2F_1\bigg[\begin{matrix}
\frac14&\frac34\\
&1
\end{matrix}\bigg|\,\frac43\bigg]_{p-1}\equiv-\jacob{2}{p}\cdot\frac{3\Gamma_p(\frac43)}{2\Gamma_p(\frac32)\Gamma_p(\frac56)}\pmod{p^2}.
\end{equation}
Note that
$$
\binom{2k}{k}\binom{4k}{2k}=\frac{(\frac14)_k(\frac34)_k}{(1)_k^2}\cdot 64^k.
$$
So with help of the Gross-Koblitz formula, we can get
\begin{equation}
\sum_{k=0}^{p-1}\binom{2k}{k}\binom{4k}{2k}\cdot\frac{1}{48^k}\equiv2a-\frac{p}{2a}\pmod{p^2},
\end{equation}
where prime $p\equiv1\pmod{3}$ is written as $p=a^2+3b^2$ with $a\equiv 1\pmod{3}$. This confirms a conjecture of Sun \cite[Conjecture 5.14(i)]{SunZW11b} for the case $p\equiv1\pmod{3}$.

\section{$p$-adic ${}_3F_2$ and ${}_4F_3$ transformations}
\label{section3F24F3}
\setcounter{lemma}{0}
\setcounter{theorem}{0}
\setcounter{corollary}{0}
\setcounter{remark}{0}
\setcounter{equation}{0}
\setcounter{conjecture}{0}

In this section, we shall prove Theorems \ref{alphaalphabeta2beta}, \ref{alphabetagamma1delta13F2}, \ref{alphabetagammadeltaepsilon4F3} and \ref{alphabetagammadeltaepsilon4F3B}. 
First, Theorem \ref{alphabetagamma1delta13F2} can be proved in a similar way as Theorem \ref{alphabeta11}.
\begin{proof}[Proof of Theorem \ref{alphabetagamma1delta13F2}]
Let $a=\langle-\alpha\rangle_p$, 
$b=\langle-\beta\rangle_p$, $c=\langle-\gamma\rangle_p$ and $d=\langle-\delta\rangle_p$.
Let $M=\max\{a,b,c\}$ and $N=\max\{a,p-1-b,p-1-c\}$. 
Since $\max\{a,b,c\}\leq d$, $b+c\geq d$ and $a+b+c\leq p+d-1$, we have
$$
\frac{(\alpha)_k(\beta)_k(\gamma)_k}{(\delta)_k}\equiv 0\pmod{p^2},\qquad\frac{(\alpha)_j(1-\beta)_j(1-\gamma)_j}{(1+\delta-\beta-\gamma)_j}\equiv 0\pmod{p^2},
$$
for any $M+1\leq k\leq p-1$ and $N+1\leq j\leq p-1$.
Let
$$
\Psi(x)={}_3F_2\bigg[\begin{matrix}-a+x&\beta&\gamma\\
&1&\delta\end{matrix}\bigg|\,1\bigg]_{M},\qquad
\Phi(x)={}_3F_2\bigg[\begin{matrix}-a+x&1-\beta&1-\gamma\\
&1&1+\delta-\beta-\gamma\end{matrix}\bigg|\,1\bigg]_{N}
$$
Then
$$
{}_3F_2\bigg[\begin{matrix}\alpha&\beta&\gamma\\
&1&\delta\end{matrix}\bigg|\,1\bigg]_{p-1}\equiv\Psi(sp)\pmod{p^2},\qquad
{}_3F_2\bigg[\begin{matrix}\alpha&1-\beta&1-\gamma\\
&1&1+\delta-\beta-\gamma\end{matrix}\bigg|\,1\bigg]_{p-1}\equiv\Phi(sp)\pmod{p^2},
$$
where $s=(\alpha+a)/p$.
By (\ref{alphabetagamma1delta13F2H}), we have
\begin{align*}
\Psi(0)=&{}_3F_2\bigg[\begin{matrix}-a&\beta&\gamma\\
&1&\delta\end{matrix}\bigg|\,1\bigg]\\
=&\frac{\Gamma(\delta)\Gamma(1+\delta+a-\beta-\gamma)}{
\Gamma(\delta+a)\Gamma(1+\delta-\beta-\gamma)}\cdot{}_3F_2\bigg[\begin{matrix}-a&1-\beta&1-\gamma\\
&1&1+\delta-\beta-\gamma\end{matrix}\bigg|\,1\bigg].
\end{align*}
Since $a\leq d$ and $0\leq b+c-d<p-a$, it is easy to check that
$$
\frac{\Gamma(\delta)}{
\Gamma(\delta+a)}=(-1)^a\cdot\frac{\Gamma_p(\delta)}{
\Gamma_p(\delta+a)},\qquad
\frac{\Gamma(1+\delta+a-\beta-\gamma)}{
\Gamma(\Gamma(1+\delta-\beta-\gamma)}=(-1)^a\cdot\frac{\Gamma_p(1+\delta+a-\beta-\gamma)}{
\Gamma(\Gamma_p(1+\delta-\beta-\gamma)}.
$$
Hence
letting
$$
\Omega(x)=\frac{\Gamma_p(\delta)\Gamma_p(1+\delta+a-x-\beta-\gamma)}{\Gamma_p(\delta+a-x)\Gamma_p(1+\delta-\beta-\gamma)},
$$
we have
$$\Psi(0)=\Omega(0)\Phi(0).$$ 

On the other hand,
\begin{align*}
\Psi'(0)\equiv&\frac{d}{dx}\bigg({}_3F_2\bigg[\begin{matrix}-a+x&-b&p-c\\
&1&p-d\end{matrix}\bigg|\,1\bigg]\bigg)\bigg|_{x=0}\\
=&
\frac{d}{dx}\bigg(\frac{\Gamma(p-d)\Gamma(1+a+b+c-d-x)}{\Gamma(p-d+a-x)\Gamma(1+b+c-d)}\cdot{}_3F_2\bigg[\begin{matrix}-a+x&1+b&1+c-p\\
&1&1+b+c-d\end{matrix}\bigg|\,1\bigg]\bigg)\bigg|_{x=0}\\
\equiv&\frac{d}{dx}\bigg(\frac{\Gamma(p-d)\Gamma(1+a+b+c-d-x)}{\Gamma(p-d+a-x)\Gamma(1+b+c-d)}\bigg)\bigg|_{x=0}\cdot\Phi(0)+
\frac{(1+\delta-\beta-\gamma)_a}{(\delta)_a}\cdot\Phi'(0)\pmod{p}.
\end{align*}
Hence by Lemma \ref{Gammapadicderivativealphabeta}, we have
$$
\Psi(sp)-\Psi(0)\equiv\big(\Omega(sp)-\Omega(0)\big)\Phi(sp)+
\Omega(0)\big(\Phi(sp)-\Phi(0)\big)=\Omega(sp)\Phi(sp)-\Omega(0)\Phi(0)\pmod{p^2}.
$$
We are done.
\end{proof}
The proof of Theorem \ref{alphaalphabeta2beta} is a little different. We shall construct a polynomial $\Psi(x)$ such that both $\Psi(0)$ and $\Psi(p)$ can be explicitly computed.
\begin{proof}
Let $a=\langle-\alpha\rangle_p$, $b=\langle-\beta\rangle_p$  and $s=(\alpha+a)/p$. Since $b<p/2$, $(\beta)_k/(2\beta)_k$ is $p$-integral for each $0\leq k\leq p-1$. Let
$$
\Psi(x)={}_3F_2\bigg[\begin{matrix}
-a+x&1+a-x&\beta\\
&1&2\beta
\end{matrix}\bigg|\,1\bigg]_{p-1}
$$
and
$$
\Phi(x)=\frac{\Gamma_p(\frac12)\Gamma_p(\frac12+\beta)\Gamma_p(\beta)}{\Gamma_p(\frac12-\frac{a-x}2)\Gamma_p(1+\frac{a-x}2)\Gamma_p(\beta-\frac{a-x}2)\Gamma_p(\frac12+\beta+\frac{a-x}2)}.
$$

First, assume that $a$ is even. By (\ref{alphaalphabeta2betaH}),
\begin{align*}
\Psi(0)={}_3F_2\bigg[\begin{matrix}
-a&1+a&\beta\\
&1&2\beta
\end{matrix}\bigg|\,1\bigg]
=
\frac{\Gamma(\frac12)\Gamma(\frac12+\beta)\Gamma(\beta)}{\Gamma(\frac12-\frac12a)\Gamma(1+\frac12a)\Gamma(\beta-\frac12a)\Gamma(\frac12+\beta+\frac12a)}.
\end{align*}
Now we have
$$
\frac{\Gamma(\frac12+\beta)}{\Gamma(\frac12+\beta+\frac12a)}=\prod_{j=0}^{\frac12a-1}\frac1{\frac12+\beta+j}=(-1)^{\frac12a}\cdot\frac{\Gamma_p(\frac12+\beta)}{\Gamma_p(\frac12+\beta+\frac12a)}.
$$
since $b<p/2$. Similarly,
$$
\frac{\Gamma(\frac12)}{\Gamma(\frac12+\frac12a)}=(-1)^{\frac12a}\cdot\frac{\Gamma_p(\frac12)}{\Gamma_p(\frac12+\frac12a)},
\qquad \frac{\Gamma(\beta)}{\Gamma(\beta-\frac12a)}=(-1)^{\frac12a}\cdot\frac{\Gamma_p(\beta)}{\Gamma_p(\beta-\frac12a)}.
$$
It follows that
\begin{align*}
\Psi(0)=-
\frac{\Gamma_p(\frac12)\Gamma_p(\frac12+\beta)\Gamma_p(\beta)}{\Gamma_p(\frac12-\frac12a)\Gamma_p(1+\frac12a)\Gamma_p(\beta-\frac12a)\Gamma_p(\frac12+\beta+\frac12a)}=-\Phi(0).
\end{align*}

Furthermore,  in view of (\ref{alphaalphabeta2betaH}), we also have
\begin{align*}
\Psi(p)=&{}_3F_2\bigg[\begin{matrix}
-(p-1-a)&1+(p-1-a)&\beta\\
&1&2\beta
\end{matrix}\bigg|\,1\bigg]\\
=&
\frac{\Gamma(\frac12)\Gamma(\frac12+\beta)\Gamma(\beta)}{\Gamma(\frac12-\frac{p-1-a}2)\Gamma(1+\frac{p-1-a}2)\Gamma(\beta-\frac{p-1-a}2)\Gamma(\frac12+\beta+\frac{p-1-a}2)}\\
=&-
\frac{\Gamma_p(\frac12)\Gamma_p(\frac12+\beta)\Gamma_p(\beta)}{\Gamma_p(1+\frac{a-p}2)\Gamma_p(\frac12-\frac{a-p}2)\Gamma_p(\frac12+\beta+\frac{a-p}2)\Gamma_p(\beta-\frac{a-p}2)}=-\Phi(p).
\end{align*}
Thus we get
$$
\Psi(sp)\equiv\Psi(0)+sp\cdot\big(\Psi(p)-\Psi(0)\big)=
-\Phi(0)-sp\cdot\big(\Phi(p)-\Phi(0)\big)=-\Phi(sp)\pmod{p^2}.
$$ 

Next, assume that $a$ is odd. It follows from (\ref{alphaalphabeta2betaH}) that
\begin{align*}
\Psi(0)=&{}_3F_2\bigg[\begin{matrix}
-a&1+a&\beta\\
&1&2\beta
\end{matrix}\bigg|\,1\bigg]=\lim_{x\to0}{}_3F_2\bigg[\begin{matrix}
-a+x&1+a-x&\beta\\
&1&2\beta
\end{matrix}\bigg|\,1\bigg]\\
=&\lim_{x\to0}\frac{\Gamma(\frac12)\Gamma(\frac12+\beta)\Gamma(\beta)}{\Gamma(\frac12-\frac{a-x}2)\Gamma(1+\frac{{a-x}}2)\Gamma(\beta-\frac{a-x}2)\Gamma(\frac12+\beta+\frac{a-x}2)}=0,
\end{align*}
and
\begin{align*}
\Psi(p)=&\lim_{x\to0}{}_3F_2\bigg[\begin{matrix}
-a+p+x&1+a-p-x&\beta\\
&1&2\beta
\end{matrix}\bigg|\,1\bigg]\\
=&\lim_{x\to0}\frac{\Gamma(\frac12)\Gamma(\frac12+\beta)\Gamma(\beta)}{\Gamma(\frac12-\frac{p-1-a+x}2)\Gamma(1+\frac{{p-1-a+x}}2)\Gamma(\beta-\frac{p-1-a+x}2)\Gamma(\frac12+\beta+\frac{p-1-a+x}2)}=0.
\end{align*}
Hence
$$
\Psi(sp)\equiv\Psi(0)+sp\cdot\big(\Psi(p)-\Psi(0)\big)=0\pmod{p^2}.
$$\end{proof}

Now let us turn to the $p$-adic ${}_4F_3$ transformation. 
\begin{proof}[Proof of Theorem \ref{alphabetagammadeltaepsilon4F3}]
Let $a=\langle-\alpha\rangle_p$, $b=\langle-\beta\rangle_p$, $c=\langle-\gamma\rangle_p$, $d=\langle-\delta\rangle_p$, $e=\langle-\epsilon\rangle_p$
and $f=\langle-\rho\rangle_p$. Let $s=(\alpha+a)/p$.
According to the conditions (i), (ii) and (iii), it is easy to check that
$$
\frac{(\beta)_k(\gamma)_k}{(\delta)_k(\epsilon)_k},\ \frac{(1-\beta)_k(1-\gamma)_k}{(1+\delta-\beta-\gamma)_k(1+\epsilon-\beta-\gamma)_k}\in\Z_p
$$
for each $0\leq k\leq p-1$.
Let $$\varrho_a=\delta+\epsilon+a-\beta-\gamma.$$ 
Let
$$
\Psi(x)={}_4F_3\bigg[\begin{matrix}\varrho_a-x&-a+x&\beta&\gamma\\
&1&\delta&\epsilon\end{matrix}\bigg|\,1\bigg]_{p-1}
$$
and
$$
\Phi(x)={}_4F_3\bigg[\begin{matrix}\varrho_a-x&-a+x&1-\beta&1-\gamma\\
&1&1+\delta-\beta-\gamma&1+\epsilon-\beta-\gamma\end{matrix}\bigg|\,1\bigg]_{p-1}.
$$
By (\ref{alphabetagammadeltaepsilon4F3H}),
\begin{align*}
&\Psi(0)=\frac{(\delta-\varrho_a)_a(\epsilon-\varrho_a)_a}{(\delta)_a(\epsilon)_a}\cdot\Phi(0).
\end{align*}
By (ii) and (iii),
$$
\langle\epsilon-\varrho_a\rangle_p=\langle\beta+\gamma-\delta-a\rangle_p=p+d-b-c-a\leq p-a.
$$
Hence
$$
(\epsilon-\varrho_a)_a=(-1)^a\cdot\frac{\Gamma_p(\beta+\gamma-\delta)}{
\Gamma_p(\epsilon-\varrho_a)}.$$
Similarly, by (i), (ii) and (iii), we have $$\max\big\{\langle\delta-\varrho_a\rangle_p,\langle \delta\rangle_p,\langle \epsilon\rangle_p\big\}\leq p-a.$$
It follows that 
\begin{align*}
\Psi(0)=\frac{\Gamma_p(\beta+\gamma-\delta)}{\Gamma_p(\delta-\varrho_a)}\cdot \frac{\Gamma_p(\beta+\gamma-\epsilon)}{\Gamma_p(\epsilon-\varrho_a)}
\cdot\frac{\Gamma_p(\delta)}{\Gamma_p(\delta+a)}\cdot\frac{\Gamma_p(\epsilon)}{
\Gamma_p(\epsilon+a)}\cdot\Phi(0)=\Omega(0)\Phi(0),
\end{align*}
where 
$$
\Omega(x)=\frac{\Gamma_p(\beta+\gamma-\delta)\Gamma_p(\beta+\gamma-\epsilon)\Gamma_p(\delta)\Gamma_p(\epsilon)}{\Gamma_p(\delta-\varrho_a+x)\Gamma_p(\epsilon-\varrho_a+x)\Gamma_p(\delta+a-x)\Gamma_p(\epsilon+a-x)}.
$$

On the other hand, recalling that $f=\langle-\varrho_a\rangle_p$ and letting $t=(\varrho_a+f)/p$, we have
\begin{align*}
&\Psi(tp)={}_4F_3\bigg[\begin{matrix}-f&-a+\varrho_a+f&\beta&\gamma\\
&1&\delta&\epsilon\end{matrix}\bigg|\,1\bigg]\\
=&\frac{(\delta+a-\varrho_a-f)_f
(\epsilon+a-\varrho_a-f)_f
}{(\delta)_f(\epsilon)_f}\cdot{}_4F_3\bigg[\begin{matrix}-f&\varrho_a-a+f&1-\beta&1-\gamma\\
&1&\varrho_a-a-\delta+1&\varrho_a-a-\epsilon+1\end{matrix}\bigg|\,1\bigg]\\
=&\frac{\Gamma(\delta+a-\varrho_a)\Gamma(\epsilon+a-\varrho_a)\Gamma(\delta)\Gamma(\epsilon)}{\Gamma(\delta+a-\varrho_a-f)\Gamma(\epsilon+a-\varrho_a-f)\Gamma(\delta+f)\Gamma(\epsilon+f)}\cdot\Phi(tp).
\end{align*}
It is also easy to check that
\begin{align*}
&\frac{\Gamma(\delta+a-\varrho_a)\Gamma(\epsilon+a-\varrho_a)\Gamma(\delta)\Gamma(\epsilon)}{\Gamma(\delta+a-\varrho_a-f)\Gamma(\epsilon+a-\varrho_a-f)\Gamma(\delta+f)\Gamma(\epsilon+f)}\\
=&
\frac{\Gamma_p(\beta+\gamma-\delta)\Gamma_p(\beta+\gamma-\epsilon)\Gamma_p(\delta)\Gamma_p(\epsilon)}{\Gamma_p(\delta+a-tp)\Gamma_p(\epsilon+a-tp)\Gamma_p(\delta-\varrho_a+tp)\Gamma_p(\epsilon-\varrho_a+tp)}
=\Omega(tp).
\end{align*}
If $t\neq 0$, then
\begin{align*}
\Psi(sp)=&\Psi(0)+\frac{s}{t}\cdot\big(\Psi(tp)-\Psi(0)\big)\\
=&
\Omega(0)\Phi(0)+\frac{s}{t}\cdot\big(\Omega(tp)\Phi(tp)-\Omega(0)\Phi(0)\big)\equiv
\Omega(sp)\Phi(sp)\pmod{p^2}.
\end{align*}
Of course, if $t=0$, i.e., $\varrho_a=-f$, then we can easily check that 
\begin{align*}
&{}_4F_3\bigg[\begin{matrix}-f&\alpha&\beta&\gamma\\
&1&\delta&\epsilon\end{matrix}\bigg|\,1\bigg]_{p-1}={}_4F_3\bigg[\begin{matrix}-f&\alpha&\beta&\gamma\\
&1&\delta&\epsilon\end{matrix}\bigg|\,1\bigg]\\
=&\frac{\Gamma(\delta-\alpha+f)\Gamma(\epsilon-\alpha+f)\Gamma(\delta)\Gamma(\epsilon)}{
\Gamma(\delta-\alpha)\Gamma(\epsilon-\alpha)\Gamma(\delta+f)\Gamma(\epsilon+f)}\cdot{}_4F_3\bigg[\begin{matrix}-f&\alpha&1-\beta&1-\gamma\\
&1&\alpha-\delta-f+1&\alpha-\epsilon-f+1\end{matrix}\bigg|\,1\bigg]\\
=&\frac{\Gamma_p(\beta+\gamma-\delta)\Gamma_p(\beta+\gamma-\epsilon)\Gamma_p(\delta)\Gamma_p(\epsilon)}{
\Gamma_p(\delta-\alpha)\Gamma_p(\epsilon-\alpha)\Gamma_p(\delta+f)\Gamma_p(\epsilon+f)}\cdot {}_4F_3\bigg[\begin{matrix}-f&\alpha&1-\beta&1-\gamma\\
&1&\alpha-\delta-f+1&\alpha-\epsilon-f+1\end{matrix}\bigg|\,1\bigg]_{p-1}.
\end{align*}

\end{proof}

However, in order to prove (\ref{alphabetagammadeltaepsilon4F3GB}),
we need another hypergeometric identity
\begin{equation}\label{alphabetagammadeltaepsilon4F3HB}
{}_4F_3\bigg[\begin{matrix}-n&\alpha&\beta&\gamma\\
&1&\delta&\epsilon\end{matrix}\bigg|\,1\bigg]=\frac{(1-\alpha)_n(\epsilon-\alpha)_n}{n!\cdot(\epsilon)_n}\cdot
{}_4F_3\bigg[\begin{matrix}-n&\alpha&\delta-\beta&\delta-\gamma\\
&\delta&\alpha-n&1+\alpha-n-\epsilon\end{matrix}\bigg|\,1\bigg]
\end{equation}
where $n=\delta+\epsilon-\alpha-\beta-\gamma\in\N$. Clearly (\ref{alphabetagammadeltaepsilon4F3HB}) follows from (\ref{Whippletransformation4F3}) by setting $\delta=1$.

\begin{proof}[Proof of Theorem \ref{alphabetagammadeltaepsilon4F3B}]
Without loss of generality, we may assume that $d\geq e$. Then $\max\{b,c\}\leq d$ by the condition (i). Let
$$
\Psi(x)={}_4F_3\bigg[\begin{matrix}\varrho_a-x&-a+x&\beta&\gamma\\
&1&\delta&\epsilon\end{matrix}\bigg|\,1\bigg]_{p-1}
$$ 
where $\varrho_a=\delta+\epsilon+a-\beta-\gamma$. Note that $\varrho_a\in\Z_p^\times$ by (ii).
With help of (\ref{alphabetagammadeltaepsilon4F3HB}),
\begin{align}\label{alphabetagammadeltaepsilon4F3HBPsi0}
\Psi(0)=\frac{(1-\varrho_a)_a(\epsilon-\varrho_a)_a}{a!\cdot(\epsilon)_a}\cdot
{}_4F_3\bigg[\begin{matrix}\varrho_a&-a&\delta-\beta&\delta-\gamma\\
&\delta&\varrho_a-a&1+\varrho_a-a-\epsilon\end{matrix}\bigg|\,1\bigg].
\end{align}

If $a+b+c>p+d$, then by (i) and (ii), we have
$$
\langle\beta+\gamma-\delta-a\rangle_p=2p+d-a-b-c>p+d-b-c=\langle\beta+\gamma-\delta\rangle_p.
$$
It follows that
\begin{equation}\label{epsilonvarrhoaa0modp}
(\epsilon-\varrho_a)_a=(\beta+\gamma-\delta-a)_a\equiv0\pmod{p}.
\end{equation}
Of course, since $a\geq 1$, (\ref{epsilonvarrhoaa0modp}) also holds when $a+b+c=p+d$.
Similarly, we can get
\begin{equation}\label{1varrhoaa0modp}
(1-\varrho_a)_a=(1+\beta+\gamma-\delta-\epsilon-a)_a\equiv0\pmod{p}.
\end{equation}
Let $r_1=\nu_p((1-\varrho_a)_a)$ and $r_2=\nu_p((\epsilon-\varrho_a)_a)$. It is easy to see that $\nu_p((\varrho_a-a)_{p-1})=r_1$ and $\nu_p((1+\varrho_a-a-\epsilon)_{p-1})=r_2$.
Moreover, by (i) and (ii), we have
\begin{equation}\label{rhoaadeltabetap}
\langle-(\varrho_a-a)\rangle_p=d+e-b-c\geq d-b=\langle-(\delta-\beta)\rangle_p
\end{equation}
and
\begin{equation}\label{rhoaa1epsilondeltagammap}
\langle-(1+\varrho_a-a-\epsilon)\rangle_p=p+d-b-c-1\geq d-c=\langle-(\delta-\gamma)\rangle_p.
\end{equation}
So
$$
p^{r_1+r_2-2}\cdot\frac{(\varrho_a)_k(-a)_k}{(\delta)_k}\cdot\frac{(\delta-\beta)_k}{(\varrho_a-a)_k}\cdot\frac{(\delta-\gamma)_k}{(1+\varrho_a-a-\epsilon)_k}\in\Z_p
$$
for any $0\leq k\leq p-1$. It follows from (\ref{alphabetagammadeltaepsilon4F3HBPsi0}) that
\begin{align*}
\Psi(0)\equiv0\pmod{p^2}.
\end{align*}

Let
$$
\phi(x,y)=
{}_4F_3\bigg[\begin{matrix}-f-x&-a+y&\delta-\beta&\delta-\gamma\\
&\delta&-f-a-x+y&1-f-a-\epsilon-x+y\end{matrix}\bigg|\,1\bigg]_{p-1}.
$$
Since $(\beta)_k(\gamma)_k/((\delta)_k(\epsilon)_k)$ is $p$-integral for $0\leq k\leq p-1$, in view of (\ref{alphabetagammadeltaepsilon4F3HB}),
\begin{align*}
&\Psi'(0)\equiv\frac{d}{d x}\bigg({}_4F_3\bigg[\begin{matrix}-f-x&-a&\beta&\gamma\\
&1&\delta&\epsilon\end{matrix}\bigg|\,1\bigg]+{}_4F_3\bigg[\begin{matrix}-f&-a+x&\beta&\gamma\\
&1&\delta&\epsilon\end{matrix}\bigg|\,1\bigg]\bigg)\bigg|_{x=0}\\
\equiv&\frac{d}{d x}\bigg(\frac{(1+f+x)_a(\epsilon+f+x)_a}{a!\cdot(\epsilon)_a}\cdot\phi(x,0)
+\frac{(1+a-x)_f(\epsilon+a-x)_f}{f!\cdot(\epsilon)_f}\cdot\phi(0,x)\bigg)\bigg|_{x=0}\\
\equiv&\frac{d}{d x}\bigg(\frac{(1+f+x)_a(\epsilon+f+x)_a}{a!\cdot(\epsilon)_a}
+\frac{(1+a-x)_f(\epsilon+a-x)_f}{f!\cdot(\epsilon)_f}\bigg)\bigg|_{x=0}\cdot\phi(0,0)\\
&+\frac{(1+f)_a(\epsilon+f)_a}{a!\cdot(\epsilon)_a}\cdot\phi_x'(0,0)+\frac{(1+a)_f(\epsilon+a)_f}{f!\cdot(\epsilon)_f}\cdot\phi_y'(0,0)\pmod{p}.
\end{align*}
Clearly (\ref{epsilonvarrhoaa0modp}) and (\ref{1varrhoaa0modp}) implies that $(1+f)_a$
and $(\epsilon+f)_a$ are the multiples of $p$,
 i.e., $p$ divides 
both $(\epsilon+a)_f$ and $(1+a)_f$. 
So
$$
\frac{d}{d x}\bigg(\frac{(1+f+x)_a(\epsilon+f+x)_a}{a!\cdot(\epsilon)_a}
+\frac{(1+a-x)_f(\epsilon+a-x)_f}{f!\cdot(\epsilon)_f}\bigg)\bigg|_{x=0}\equiv 0\pmod{p}.
$$
Furthermore, since $e\geq a$ and $(\epsilon)_{p-1}\not\equiv0\pmod{p^2}$, we must have $(\epsilon+a)_{p-1},(\epsilon+f)_{p-1}\not\equiv 0\pmod{p^2}$. 
Hence by (\ref{epsilonvarrhoaa0modp}),  $\nu_p((\epsilon+f)_a)=1$, as well as $\nu_p((1-f-a-\epsilon)_a)=1$.
In view of (\ref{rhoaadeltabetap}) and (\ref{rhoaa1epsilondeltagammap}), 
\begin{align*}p\cdot\phi_x'(0,0)
=&\sum_{k=0}^a\frac{p\cdot (-a)_k(\delta-\beta)_k(\delta-\gamma)_k}{(\delta)_k(-f-a)_k(1-f-a-\epsilon)_k}\cdot\frac{d\big((-f-x)_k\big)}{d x}\bigg|_{x=0}\\
&+\sum_{k=0}^a\frac{(-f)_k(-a)_k(\delta-\beta)_k(\delta-\gamma)_k}{(\delta)_k(-f-a)_k(1-f-a-\epsilon)_k}\sum_{j=0}^{k-1}\bigg(\frac{p}{j-f-a}+\frac{p}{j+1-f-a-\epsilon}\bigg)
\end{align*}
is $p$-adic integral. Similarly,
we also have $p\cdot\phi_y'(0,0)\in\Z_p$. 
Hence
$$
\Psi(sp)\equiv\Psi(0)+sp\cdot\Psi'(0)\equiv 0\pmod{p^2}.
$$
\end{proof}

\section{$p$-adic Whipple's ${}_7F_6$ transformation I: Theorem \ref{alphabetagammadeltaepsilon7F6A}}
\label{section7F6I}
\setcounter{lemma}{0}
\setcounter{theorem}{0}
\setcounter{corollary}{0}
\setcounter{remark}{0}
\setcounter{equation}{0}
\setcounter{conjecture}{0}

In this section, we shall prove Theorem \ref{alphabetagammadeltaepsilon7F6A}. 
Of course, as we have mentioned, Theorems \ref{alphaalpha12alphabetagamma5F4A} and  \ref{alphabetagammadelta6F5A}, whose proofs wouldn't be given here,
can be proved in the same way. And at the end of this section, we shall explain how to deduce Theorem \ref{alphabetagammadeltan7F6A} from Theorem \ref{alphabetagammadeltaepsilon7F6A}.

Let $a=\langle-\alpha\rangle_p$, $b=\langle-\beta\rangle_p$, $c=\langle-\gamma\rangle_p$, $d=\langle-\delta\rangle_p$ and $e=\langle-\epsilon\rangle_p$.
Note that $b<a\leq\min\{c,d,e\}$ and $\max\{c+d,c+e,d+e\}\leq p+a-1$ by (i) and (ii). It is not difficult to verify that
$$
\frac{(\alpha)_k^2(1+\frac12\alpha)_k(\beta)_k(\gamma)_k(\delta)_k(\epsilon)_k}{(\frac12\alpha)_k(\alpha-\beta+1)_k(\alpha-\gamma+1)_k(\alpha-\delta+1)_k(\alpha-\epsilon+1)_k}\equiv0\pmod{p^2}
$$
for any $a+1\leq k\leq p-1$. And since
$\max\{p-1-e,b,c,d\}\leq p+a-1-e$ and $\max\{c,d\}\leq b+c+d-a$,
we obtain that
$$
\frac{(1-\epsilon)_k(\beta)_k(\gamma)_k(\delta)_k}{(\alpha-\epsilon+1)_k(\beta+\gamma+\delta-\alpha)_k}\equiv0\pmod{p^2}
$$
provided $$1+\max\big\{b,\min\{c,d,p-1-e\}\big\}\leq k\leq p-1.$$

Let
$$
\cA_k(x)=(-a-ax)_k^2(-b-bx)_k(\gamma)_k(\delta)_k(\epsilon)_k
$$
and
$$
\cB_k(x)=(1)_k^2(1-a-ax-\gamma)_k(1-a-ax-\delta)_k(1-a-ax-\epsilon)_k\prod_{\substack {1\leq j\leq k\\ j\neq a-b}}(j-a-ax+b+bx).
$$
Let
$$
\Psi(x)=\sum_{k=0}^{a-b-1}\frac{(1-\frac12a-\frac12ax)_k}{(-\frac12a-\frac12ax)_k}\cdot\frac{\cA_k(x)}{\cB_k(x)}+\sum_{k=a-b}^{M}\frac{(1-\frac12a-\frac12ax)_k}{(-\frac12a-\frac12ax)_k}\cdot\frac{1}{(b-a)x}\cdot\frac{\cA_k(x)}{\cB_k(x)},
$$
where $M=\min\{c,d,e\}$.
Write $\alpha=-a(1+sp)$.
Clearly
$$
\Psi(sp)=
{}_7F_6\bigg[\begin{matrix} \alpha&1+\frac12\alpha&\alpha&\beta&\gamma&\delta&\epsilon\\ &\frac12\alpha&1&1+\alpha-\beta&1+\alpha-\gamma&1+\alpha-\delta&1+\alpha-\epsilon\end{matrix}\bigg|\,1\bigg]_M.
$$
\begin{lemma}\label{PsixPxQx}
$$\Psi(x)=\frac{P(x)}{Q(x)},$$
where $P(x)$ and $Q(x)$ are two polynomials with $p\nmid Q(0)$. Furthermore, the coefficients of $P(x)$ and $Q(x)$ are all the polynomials in $a,b,\gamma,\delta,\epsilon$ with integral coefficients.
\end{lemma}
That is, we may write
\begin{equation}\label{PxQx}
P(x)=\sum_{k\geq 0}{\mathfrak p}_k(a,b,\gamma,\delta,\epsilon)x^k,\qquad
Q(x)=\sum_{k\geq 0}{\mathfrak q}_k(a,b,\gamma,\delta,\epsilon)x^k,
\end{equation}
where ${\mathfrak p}_k,{\mathfrak q}_k$ are polynomials whose coefficients are integers.
\begin{proof}
Clearly $p\nmid \cB_k(0)$ for any $0\leq k\leq M$. Also, the polynomial $\cA_k(x)$ is divisible by $x^3$ for those $a+1\leq k\leq M$.
Note that
$$\frac{(1-\frac12a-\frac12ax)_k}{(-\frac12a-\frac12ax)_k}=1-\frac{2k}{a+ax}.
$$
According to definition of $\Psi(x)$, 
we only need to prove that for each $a-b\leq k\leq a/2$, the constant term of the numerator of
$$
\bigg(1-\frac{2k}{a+ax}\bigg)\cdot\frac{\cA_k(x)}{\cB_k(x)}+
\bigg(1-\frac{2a-2k}{a+ax}\bigg)\cdot\frac{\cA_{a-k}(x)}{\cB_{a-k}(x)}
$$
is zero, i.e.,
\begin{equation}\label{a2kAkBak2kaAakBk}
(a-2k)\cA_k(0)\cB_{a-k}(0)+
(2k-a)\cA_{a-k}(0)\cB_{k}(0)=0.
\end{equation}
There is nothing to do when $2b<a$, since $\cA_k(0)=0$ for any $k>b$.
So we may assume that $2b\geq a$.
Since $\Gamma(z)\Gamma(1-z)=\pi/\sin(\pi z)$, for any non-integral $z\in\R$, we have
\begin{align*}
\frac{(z)_{a-k}}{(1-a-z)_{a-k}}=\frac{\Gamma(z+a-k)\Gamma(1-a-z)}{\Gamma(z)\Gamma(1-z-k)}
=(-1)^a\cdot\frac{\Gamma(1-a-z)\Gamma(z+k)}{\Gamma(1-a-z+k)\Gamma(z)}=\frac{(-1)^a(z)_{k}}{(1-a-z)_{k}},
\end{align*}
i.e.,
\begin{equation}\label{etal1aetaaketaak1aetak}
(z)_{k}(1-a-z)_{a-k}=(-1)^a(z)_{a-k}(1-a-z)_{k}.
\end{equation}
Clearly both sides of (\ref{etal1aetaaketaak1aetak}) are the polynomials in $z$. So (\ref{etal1aetaaketaak1aetak}) is factly valid for every $z\in\R$. Further,
\begin{align*}
&(-b)_k\prod_{\substack {1\leq j\leq a-k\\ j\neq a-b}}(j-a+b)=\lim_{z\to 0}\frac{(-b-z)_k(1-a+b+z)_{a-k}}{z}\\
=&
(-1)^a\lim_{z\to 0}\frac{(-b-z)_{a-k}(1-a+b+z)_{k}}{z}=(-1)^a(-b)_{a-k}\prod_{\substack {1\leq j\leq k\\ j\neq a-b}}(j-a+b).
\end{align*}
Thus (\ref{a2kAkBak2kaAakBk}) is immediately derived.
\end{proof}
Let
$$
\Omega_1(x)=\frac{\Gamma(1-a-ax+b+bx)}{\Gamma(1-a-ax)},
$$
\begin{align*}
\Omega_2(x)=\frac{\Gamma(1-a-ax-\gamma)\Gamma(1-a-ax-\delta)\Gamma(1-a-ax+b+bx-\gamma-\delta)}
 {\Gamma(1-a-ax+b+bx-\gamma)\Gamma(1-a-ax+b+bx-\delta)\Gamma(1-a-ax-\gamma-\delta)},
\end{align*}
and
$$
\Phi(x)={}_4F_3\bigg[\begin{matrix} 1-\epsilon&-b-bx&\gamma&\delta\\ &1&1-a-ax-\epsilon&a+ax-b-bx+\gamma+\delta\end{matrix}\bigg|\,1\bigg]_N
$$
where $N=\max\{b,\min\{c,d,p-1-e\}\}$. Note that
\begin{align*}
\lim_{x\to0}\Omega_1(x)=
 \lim_{x\to0}\frac{\Gamma(1-ax+bx)}{\Gamma(1-ax)}\cdot\frac{(1-a-ax)(2-a-ax)\cdots(-ax)}{(1-a-ax+b+bx)\cdots(-ax+bx)}=\frac{a\cdot(1-a)_b}{a-b}.
\end{align*}
So we set
$$
\Omega_1(0)=\frac{a}{a-b}\cdot(1-a)_b.
$$
Moreover, it is not difficult to check that $\Omega_1(x)$ is continuously differentiable on $x=0$, i.e.,
$$
\Omega_1'(0)=\lim_{x\to 0}\Omega_1'(x).
$$ 
Recall that $\beta,\delta,\epsilon$ are all negative integers, and $\cA_k(x)$ is divisible by $x^3$ for each $k>a$. In view of (\ref{alphabetagammadeltaepsilon7F6H}), we have
\begin{align*}
\Psi(0)=\lim_{x\to 0}\Psi(x)=\lim_{x\to 0}\Omega_1(x)\Omega_2(x)\Phi(x)
=\Omega_1(0)\Omega_2(0)\Phi(0).
\end{align*}

Let
\begin{align*}
\Upsilon(x)=&\frac{\Gamma_p(1-a-ax+b+bx)\Gamma_p(1-a-ax-\gamma)}{\Gamma_p(1-a-ax)\Gamma_p(1-a-ax+b+bx-\gamma)}\\
&\cdot\frac{\Gamma_p(1-a-ax-\delta)\Gamma_p(1-a-ax+b+bx-\gamma-\delta)}
 {\Gamma_p(1-a-ax+b+bx-\delta)\Gamma_p(1-a-ax-\gamma-\delta)}.
\end{align*}
Since $b<a\leq\min\{c,d\}$, we may get
$$
(1-a)_b=(-1)^{b}\cdot\frac{\Gamma_p(1-a+b)}{\Gamma_p(1-a)},
\qquad
\frac{\Gamma(1-a-\gamma)}{\Gamma(1-a+b-\gamma)}=(-1)^{b}\cdot\frac{\Gamma_p(1-a-\gamma)}{\Gamma_p(1-a+b-\gamma)}
$$
and
$$
\frac{\Gamma(1-a-\delta)}{\Gamma(1-a+b-\delta)}=(-1)^{b}\cdot\frac{\Gamma_p(1-a-\delta)}{\Gamma_p(1-a+b-\delta)}.
$$
Also, since $b+c+d\leq p+a-1$ by (ii), we have
$$
\frac{\Gamma(1-a+b-\gamma-\delta)}{\Gamma(1-a-\gamma-\delta)}=
(-1)^b\cdot\frac{\Gamma_p(1-a+b-\gamma-\delta)}{\Gamma_p(1-a-\gamma-\delta)}.
$$
Thus
$$
\Upsilon(0)=\frac{a-b}{a}\cdot\Omega_1(0)\Omega_2(0),
$$
i.e.,
 \begin{align}\label{7F6APsi0alphaalphagammaUpsilon0Phi0}
 \Psi(0)=\frac{a}{a-b}\cdot
\Upsilon(0)\Phi(0).
\end{align}

Let
$$
\psi(x)=\sum_{k=0}^{a-b-1}\frac{(1-\frac12a-\frac12ax)_k}{(-\frac12a-\frac12ax)_k}\cdot\frac{\cA_k^*(x)}{\cB_k^*(x)}+\sum_{k=a-b}^{M}\frac{(1-\frac12a-\frac12ax)_k}{(-\frac12a-\frac12ax)_k}\cdot\frac{1}{(b-a)x}\cdot\frac{\cA_k^*(x)}{\cB_k^*(x)},
$$
where
$$
\cA_k^*(x)=(-a-ax)_k^2(-b-bx)_k(-c)_k(-d)_k(p-e)_k
$$
and
$$
\cB_k^*(x)=(1)_k^2(1-a-ax+c)_k(1-a-ax+d)_k(1-a-ax+e-p)_k\prod_{\substack {1\leq j\leq k\\ j\neq a-b}}(j-a-ax+b+bx).
$$
In view of (\ref{PxQx}), we have
$
\psi(x)=P_*(x)/Q_*(x),
$
where
$$
P_*(x)=\sum_{k\geq 0}{\mathfrak p}_k(a,b,-c,-d,-e)x^k,\qquad
Q_*(x)=\sum_{k\geq 0}{\mathfrak q}_k(a,b,-c,-d,-e)x^k.
$$
So by Lemma \ref{PsixPxQx}, clearly
\begin{equation}\label{7F6APsi0psi0}
\Psi'(0)\equiv\psi'(0)\pmod{p}.
\end{equation}

Let
$$
\omega_2(x)=\frac{\Gamma(1-a-ax+c)\Gamma(1-a-ax+d)\Gamma(1-a-ax+b+bx+c+d)}
 {\Gamma(1-a-ax+b+bx+c)\Gamma(1-a-ax+b+bx+d)\Gamma(1-a-ax+c+d)}
$$
and
$$
\phi(x)={}_4F_3\bigg[\begin{matrix} 1+e-p&-b-bx&-c&-d\\ &1&1-a-ax+e-p&a+ax-b-bx-c-d\end{matrix}\bigg|\,1\bigg].
$$
By (\ref{alphabetagammadeltaepsilon7F6H}),
\begin{equation}\label{7F6ApsixOmega1xomega2xphix}
\psi(x)=\Omega_1(x)\omega_2(x)\phi(x).
\end{equation}
It is easy to see that
\begin{equation}\label{7F6APhi0phi0modp}
\Phi(0)\equiv \phi(0)\pmod{p},
\qquad
\Phi'(0)\equiv \phi'(0)\pmod{p},
\end{equation}
and
\begin{equation}\label{7F6Aomega20Omega20modp}
\omega_2(0)=\frac{(1-a+c+d)_b}
 {(1-a+c)_b(1-a+d)_b}\equiv\frac{(1-a-\gamma-\delta)_b}
 {(1-a-\gamma)_b(1-a-\delta)_b}=\Omega_2(0)\pmod{p}.
\end{equation}

According to Lemma \ref{Gammapadicderivativealphabeta},
$$
\Upsilon(sp)-\Upsilon(0)\equiv sp\cdot\frac{d}{dx}\bigg(\frac{\Gamma(p+1-a-ax+b+bx)}{\Gamma(p+1-a-ax)}\cdot\omega_2(x)\bigg)\bigg|_{x=0}\pmod{p^2}.
$$
We have
$$
\frac{\Gamma(p+1-a-ax+b+bx)}{\Gamma(p+1-a-ax)}=
\Omega_1(x)\cdot\frac{(1-a-ax+b+bx)_p}{(1-a-ax)_p}.
$$
Clearly
$$
\lim_{x\to0}\frac{(1-a-ax+b+bx)_p}{(1-a-ax)_p}=
\frac{a-b}{a}\cdot \prod_{\substack{1\leq j\leq p\\ j\neq a-b}}(j-a+b)\cdot\prod_{\substack{1\leq j\leq p\\ j\neq a}}\frac1{j-a}\equiv \frac{a-b}{a}\pmod{p}
$$
and
\begin{align*}
&\frac{d}{dx}\bigg(\frac{(1-a-ax+b+bx)_p}{(1-a-ax)_p}\bigg)\bigg|_{x=0}\\
=&\frac{a-b}{a}\cdot \prod_{\substack{1\leq j\leq p\\ j\neq a-b}}(j-a+b)\cdot\prod_{\substack{1\leq j\leq p\\ j\neq a}}\frac1{j-a}\cdot \bigg(\sum_{\substack{1\leq j\leq p\\ j\neq a}}\frac{a}{j-a}-\sum_{\substack{1\leq j\leq p\\ j\neq a-b}}\frac{a-b}{j-a+b}\bigg)\equiv 0\pmod{p}.
\end{align*}
It follows that
$$
\frac{d}{dx}\bigg(\frac{\Gamma(p+1-a-ax+b+bx)}{\Gamma(p+1-a-ax)}\bigg)\bigg|_{x=0}
\equiv
\frac{a-b}{a}\cdot\Omega_1'(0)\pmod{p}.
$$
Hence
\begin{equation}\label{7F6AUpsilonspUpsilon0acaOmega10omega20}
\Upsilon(sp)-\Upsilon(0)\equiv sp\cdot\frac{a-b}{a}\cdot\big(\Omega_1'(0)\cdot\omega_2(0)+\Omega_1(0)\cdot\omega_2'(0)\big)\pmod{p^2}.
\end{equation}

Since
$$
\Upsilon(sp)\equiv \Upsilon(0)\equiv\frac{a-b}{a}\cdot\Omega_1(0)\omega_2(0)\pmod{p},
$$
combining (\ref{7F6APsi0psi0}), (\ref{7F6ApsixOmega1xomega2xphix}), (\ref{7F6APhi0phi0modp}), (\ref{7F6Aomega20Omega20modp}) and (\ref{7F6AUpsilonspUpsilon0acaOmega10omega20}), we obtain that
\begin{align*}
\Psi(sp)-\Psi(0)\equiv&sp\cdot\psi'(0)\equiv sp\cdot\big(\Omega_1'(0)\omega_2(0)\phi(0)+
\Omega_1(0)\omega_2'(0)\phi(0)+\Omega_1(0)\omega_2(0)\phi'(0)\big)\\
\equiv&\frac{a}{a-b}\cdot\big(\Upsilon(sp)-\Upsilon(0)\big)\cdot\Phi(0)+\frac{a}{a-b}\cdot\Upsilon(0)\cdot(\Phi(sp)-\Phi(0))\\
\equiv&\frac{a}{a-b}\cdot\big(\Upsilon(sp)\Phi(sp)-\Upsilon(0)\phi(0)\big)\pmod{p^2}.
\end{align*}
It follows from (\ref{7F6APsi0alphaalphagammaUpsilon0Phi0}) that
$$
\Psi(sp)\equiv \frac{a}{a-b}\cdot\Upsilon(sp)\Phi(sp)=
\frac{\alpha}{\alpha-\beta}\cdot\Upsilon(sp)\Phi(sp)\pmod{p^2}.
$$
\qed

Finally,
let us give the proof of Theorem \ref{alphabetagammadeltan7F6A}. Let $a=\langle-\alpha\rangle_p$, $b=\langle-\beta\rangle_p$, $c=\langle-\gamma\rangle_p$, $d=\langle-\delta\rangle_p$ and $e=\langle-\epsilon\rangle_p$. By the condition (ii), clearly
$$
e=p+a-1-b-c-d\geq a.
$$
Note that now
$$
\big\langle-(\beta+\gamma+\delta)\big\rangle_p=b+c+d>b+c=\langle-\beta\rangle_p+
\langle-\gamma\rangle_p.
$$
According to Theorem \ref{nalphabeta1alphabetan}, we know that
\begin{align*}
&{}_4F_3\bigg[\begin{matrix} 1-\epsilon&\beta&\gamma&\delta\\ &1&\alpha-\epsilon+1&\beta+\gamma+\delta-\alpha\end{matrix}\bigg|\,1\bigg]_{p-1}\\
=&{}_3F_2\bigg[\begin{matrix} \beta&\gamma&\delta\\ &1&\beta+\gamma+\delta\end{matrix}\bigg|\,1\bigg]_{p-1}
\equiv-\frac{\Gamma_p(1-\beta-\gamma)\Gamma_p(1-\beta-\delta)\Gamma_p(1-\gamma-\delta)}{\Gamma_p(1-\beta)\Gamma_p(1-\gamma)\Gamma_p(1-\delta)\Gamma_p(1-\beta-\gamma-\delta)}\pmod{p^2}.
\end{align*}
Thus by Theorem \ref{alphabetagammadeltaepsilon7F6A}, we immediately get the desired result.

\section{$p$-adic Whipple's ${}_7F_6$ transformation II: Theorems \ref{alphabetagammadeltaepsilon7F6B} and 
\ref{alphabetagammadeltaepsilon7F6C}}

\label{section7F6II}
\setcounter{lemma}{0}
\setcounter{theorem}{0}
\setcounter{corollary}{0}
\setcounter{remark}{0}
\setcounter{equation}{0}
\setcounter{conjecture}{0}

\begin{proof}[Proof of Theorem \ref{alphabetagammadeltaepsilon7F6B}]
Let $a=\langle-\alpha\rangle_p$, $b=\langle-\beta\rangle_p$, $c=\langle-\gamma\rangle_p$, $d=\langle-\delta\rangle_p$ and $e=\langle-\epsilon\rangle_p$.
Since $a\leq\min\{b,c,d,e\}$ and $\max\{b+e,c+d\}\leq p+a-1$ by (i) and (ii),
we have
$$
\frac{(\alpha)_k^2(\frac12\alpha+1)_k(\beta)_k(\gamma)_k(\delta)_k(\epsilon)_k}{(\frac12\alpha)_k(\alpha-\beta+1)_k(\alpha-\gamma+1)_k(\alpha-\delta+1)_k(\alpha-\epsilon+1)_k}\equiv0\pmod{p^2}
$$
for each $a+1\leq k\leq p-1$.
Let
$$
\Psi(x)={}_7F_6\bigg[\begin{matrix} \alpha&\alpha&\frac12\alpha+1&-b+x&\gamma&\delta&\epsilon\\ &1&\frac12\alpha&\alpha+b-x+1&\alpha-\gamma+1&\alpha-\delta+1&\alpha-\epsilon+1\end{matrix}\bigg|\,1\bigg]_a
$$
In view of (\ref{alphabetagammadeltaepsilon7F6H}),
\begin{align*}
 \Psi(0)=\frac{(1+\alpha)_b(1+\alpha-\gamma-\delta)_b}
 {(1+\alpha-\gamma)_b(1+\alpha-\delta)_b}\cdot{}_4F_3\bigg[\begin{matrix} 1-\epsilon&-b&\gamma&\delta\\ &1&\alpha-\epsilon+1&\gamma+\delta-b-\alpha\end{matrix}\bigg|\,1\bigg].
 \end{align*}
Note that (iii) implies $p+a\leq b+c+d$.
As $a\leq b$ and $c+d-a<p$, we get
$$
(1+\alpha)_b\equiv(1+\alpha-\gamma-\delta)_b\equiv0\pmod{p}.
$$
Also, since $a\leq\min\{c,d\}$ and $p+a-b>\max\{c,d\}$,
neither $
(1+\alpha-\gamma)_b$ nor $(1+\alpha-\delta)_b$ is divisible by $p$.
Moreover, we have $a\leq e$, $p-1-e+a\geq b$ and $p-1-e\leq b+c+d-a-p$ by (i), (ii) and (iii) respectively. Letting $r=\nu_p\big((1+\alpha-\gamma-\delta)_b\big)$, we obtain that
$$
p^{r-1}\cdot\frac{(1-\epsilon)_k(-b)_k(\gamma)_k(\delta)_k}{(\alpha-\epsilon+1)_k(\gamma+\delta-b-\alpha)_k}
$$
is $p$-integral for each $0\leq k\leq b$, i.e.,
$$
p^{r-1}\cdot{}_4F_3\bigg[\begin{matrix} 1-\epsilon&-b&\gamma&\delta&\\ &1&\alpha-\epsilon+1&\gamma+\delta-b-\alpha\end{matrix}\bigg|\,1\bigg]\in\Z_p.
$$
It follows that
 \begin{align*}
 \Psi(0)\equiv0\pmod{p^2}.
\end{align*}

Let
$$
\psi(x)={}_7F_6\bigg[\begin{matrix} -a&1-\frac12a&-a&-b+x&-c&-d&-e\\ &-\frac12a&1&1-a+b-x&1-a+c&1-a+d&1-a+e\end{matrix}\bigg|\,1\bigg].
$$
By (\ref{alphabetagammadeltaepsilon7F6H}),
\begin{align*}
 \psi(x)=\frac{(1-a)_d(1-a+c+b-x)_d}
 {(1-a+c)_d(1-a+b-x)_d}\cdot{}_4F_3\bigg[\begin{matrix} 1+e&-b+x&-c&-d\\ &1&1-a+e&a-c-d-b+x\end{matrix}\bigg|\,1\bigg]
 \end{align*}
always vanishes, since $(1-a)_d=0$. Hence letting $s=(\beta+b)/p$, we have
$$
\Psi(sp)\equiv\Psi(0)+sp\cdot\Psi'(0)\equiv\Psi(0)+sp\cdot\psi'(0)\equiv 0\pmod{p^2}.
$$
\end{proof}

\begin{proof}[Proof of Theorem \ref{alphabetagammadeltaepsilon7F6C}]
Let $a=\langle-\alpha\rangle_p$, $b=\langle-\beta\rangle_p$, $c=\langle-\gamma\rangle_p$, $d=\langle-\delta\rangle_p$ and $e=\langle-\epsilon\rangle_p$. Let $s=(\beta+b)/p$. Let $M=\min\{b,c,d,e\}$ and $N=\max\{b,c,d\}$.
In view of (ii), we have
$$
(\alpha-\beta+1)_k(\alpha-\gamma+1)_k(\alpha-\delta+1)_k(\alpha-\epsilon+1)_k\not\equiv 0\pmod{p}
$$
for any $0\leq k\leq M$, and 
$$
\frac{(\alpha)_k^2(\frac12\alpha+1)_k(\beta)_k(\gamma)_k(\delta)_k(\epsilon)_k}{(\frac12\alpha)_k(\alpha-\beta+1)_k(\alpha-\gamma+1)_k(\alpha-\delta+1)_k(\alpha-\epsilon+1)_k}\equiv0\pmod{p^2}
$$
for any $M+1\leq k\leq p-1$.
Further, since $N<b+c+d-a<p$ and $p-1-e\leq p+a-1-e$, 
$$
\frac{(1-\epsilon)_k(\beta)_k(\gamma)_k(\delta)_k}{(\alpha-\epsilon+1)_k(\beta+\gamma+\delta-\alpha)_k}\equiv0\pmod{p^2}
$$
for any $N<k\leq p-1$.
Let
$$
\Psi(x)={}_7F_6\bigg[\begin{matrix} \alpha&\alpha&\frac12\alpha+1&-b+x&\gamma&\delta&\epsilon\\ &1&\frac12\alpha&\alpha+b-x+1&\alpha-\gamma+1&\alpha-\delta+1&\alpha+\epsilon+1\end{matrix}\bigg|\,1\bigg]_M,
$$
$$
\Phi(x)={}_4F_3\bigg[\begin{matrix} 1-\epsilon&-b+x&\gamma&\delta\\ &1&1+\alpha-\epsilon&\gamma+\delta-\alpha-b+x\end{matrix}\bigg|\,1\bigg]_N
$$
and
$$
\Omega(x)=\frac{\Gamma(1+\alpha+b-x)\Gamma(1+\alpha-\gamma)\Gamma(1+\alpha-\delta)\Gamma(1+\alpha-\delta-\gamma+b-x)}{
\Gamma(1+\alpha)\Gamma(1+\alpha-\gamma+b-x)\Gamma(1+\alpha-\delta+b-x)\Gamma(1+\alpha-\delta-\gamma)}.
$$
By (\ref{alphabetagammadeltaepsilon7F6H}),
\begin{equation}\label{7F6CPsi0Omega0Phi0}
\Psi(0)=\Omega(0)\Phi(0).
\end{equation}

Recall that $a\leq\min\{b,c,d\}$ and $0<b+c+d-a<p$. It is easy to check that
\begin{align*}
&\frac{\Gamma(\alpha+b+1)}{\Gamma(\alpha+1)}\cdot\frac{\Gamma(\alpha-\gamma+1)}{\Gamma(\alpha-\gamma+b+1)}\cdot\frac{\Gamma(\alpha-\delta+1)}{\Gamma(\alpha-\delta+b+1)}\cdot\frac{\Gamma(\alpha-\delta-\gamma+b+1)}{\Gamma(\alpha-\delta-\gamma+1)}\\
=&(\alpha+a)\cdot\frac{\Gamma_p(\alpha+b+1)}{\Gamma_p(\alpha+1)}\cdot\frac{\Gamma_p(\alpha-\gamma+1)}{\Gamma_p(\alpha-\gamma+b+1)}\cdot\frac{\Gamma_p(\alpha-\delta+1)}{\Gamma_p(\alpha-\delta+b+1)}\cdot\frac{\Gamma_p(\alpha-\delta-\gamma+b+1)}{\Gamma_p(\alpha-\delta-\gamma+1)}.
\end{align*}
Thus we have
\begin{equation}\label{7F6COmega0alphaaaUpsilon0}
\Omega(0)=(\alpha+a)\cdot\Upsilon(0),
\end{equation}
where
$$
\Upsilon(x)=\frac{\Gamma_p(1+\alpha+b-x)\Gamma_p(1+\alpha-\gamma)\Gamma_p(1+\alpha-\delta)\Gamma_p(1+\alpha-\delta-\gamma+b-x)}{
\Gamma(_p1+\alpha)\Gamma_p(1+\alpha-\gamma+b-x)\Gamma_p(1+\alpha-\delta+b-x)\Gamma_p(1+\alpha-\delta-\gamma)}.
$$

Applying Lemma \ref{taylorexpansionrational}, we have
$$
\Psi(sp)-\Psi(0)\equiv sp\cdot\Psi'(0)\equiv sp\cdot\psi'(0)\pmod{p^2},
$$ 
where
\begin{align*}
\psi(x)=&{}_7F_6\bigg[\begin{matrix} -a&-a&1-\frac12a&-b+x&-c&-d&-e\\ &1&-\frac12a&1-a+b-x&1-a+c&1-a+d&1-a+e\end{matrix}\bigg|\,1\bigg]\\
=&\frac{(1-a)_d(1-a+c+b-x)_d}
 {(1-a+c)_d(1-a+b-x)_d}\cdot{}_4F_3\bigg[\begin{matrix} 1+e&-b+x&-c&-d\\ &1&1-a+e&a-c-d-b+x\end{matrix}\bigg|\,1\bigg]
\end{align*}
by (\ref{alphabetagammadeltaepsilon7F6H}). Since $1\leq a\leq d$, we have $\psi(x)=0$.
Further, clearly
$$
\Upsilon(sp)\Phi(sp)\equiv \Upsilon(0)\Phi(0)\pmod{p}.
$$
So
$$
\Psi(sp)-\Psi(0)\equiv 0\equiv (\alpha+a)\cdot\big(\Upsilon(sp)\Phi(sp)-\Upsilon(0)\Phi(0)\big)\pmod{p^2}.
$$ 
It follows from (\ref{7F6CPsi0Omega0Phi0}) and (\ref{7F6COmega0alphaaaUpsilon0}) that
$$
\Psi(sp)\equiv (\alpha+a)\cdot\Upsilon(sp)\Phi(sp)\pmod{p^2}.
$$
\end{proof}

\section{A conjecture of Deines-Fuselier-Long-Swisher-Tu}
\label{sectionconjDFLST}
\setcounter{lemma}{0}
\setcounter{theorem}{0}
\setcounter{corollary}{0}
\setcounter{remark}{0}
\setcounter{equation}{0}
\setcounter{conjecture}{0}

Deines, Fuselier, Long, Swisher and Tu \cite[Conjecture 18]{DFLST16} conjectured
\begin{equation}\label{DFLSTcongruence}
\sum_{k=0}^{p-1}\bigg(\frac{(\frac12)_k}{k!}\bigg)^2\cdot(-1)^k\equiv p^2\sum_{k=\frac{p-1}2}^{p-1}\bigg(\frac{k!}{(\frac32)_k}\bigg)^2\cdot(-1)^k\pmod{p^2}
\end{equation}
for any prime $p\equiv 1\pmod{4}$. Note that $(3/2)_k$ is not divisible by $p$ for any $0\leq k<(p-1)/2$. Clearly (\ref{DFLSTcongruence}) is an equivalent form of (\ref{F2112121F321113232p2}).

It is easy to see that $$\frac{(1)_k}{(\frac32)_k}=\frac{(1)_{p-1}}{(1-p)_{p-1-k}}\cdot\frac{(\frac12-p)_{p-1-k}}{(\frac32)_{p-1}}.$$
We have
$$
p^2\sum_{k=\frac{p-1}2}^{p-1}\frac{(1)_k^2}{(\frac32)_k^2}\cdot(-1)^k=\frac{p^2\cdot(1)_{p-1}^2}{(\frac32)_{p-1}^2}\sum_{k=0}^{\frac{p-1}2}\frac{(\frac12-p)_k^2}{(1-p)_k^2}\cdot(-1)^k.$$
We need the following identity due to Kummer \cite[Corollary 3.1.2]{AAR99}:
\begin{equation}
{}_2F_1\bigg[\begin{matrix} \alpha&\beta\\ &\alpha-\beta+1\end{matrix}\bigg|\,-1\bigg]
=\frac{\Gamma(\alpha-\beta+1)\Gamma(\frac12\alpha+1)}{\Gamma(\alpha+1)\Gamma(\frac{1}{2}\alpha-\beta+1)}.
\end{equation}
Let
$$
\Psi(x)=\sum_{k=0}^{\frac{p-1}2}\frac{(\frac12-x)_k^2}{(1-x)_k^2}\cdot(-1)^k.
$$
Clearly
\begin{align*}
\Psi'(0)=&\frac{d}{dx}\bigg(2\sum_{k=0}^{\frac{p-1}2}\frac{(\frac12-x)_k\cdot(\frac12)_k}{(1-x)_k\cdot(1)_k}\cdot(-1)^k\bigg)\bigg|_{x=0}\\
\equiv&\frac{d}{dx}\bigg(2\sum_{k=0}^{\frac{p-1}2}\frac{(\frac12+\frac12p-x)_k\cdot(\frac12-\frac12p)_k}{(1+p-x)_k\cdot(1)_k}\cdot(-1)^k\bigg)\bigg|_{x=0}\\
=&2\cdot\frac{d}{dx}\bigg(\frac{\Gamma(1+\frac p2-x)\Gamma(\frac{5}{4}-\frac x{2})}{\Gamma(\frac{3}2-x)\Gamma(\frac{3}{4}+\frac p2-\frac x2)}\bigg)\bigg|_{x=0}\pmod{p}.
\end{align*}
We have
\begin{align*}
&\frac{d}{dx}\bigg(\frac{\Gamma(1+\frac p2-x)\Gamma(\frac{5}{4}-\frac x{2})}{\Gamma(\frac{3}2-x)\Gamma(\frac{3}{4}+\frac p2-\frac x2)}\bigg)\bigg|_{x=0}\\
=&\frac{d}{dx}\bigg(\frac{(\frac{3}2-x)_{\frac{p-1}{2}}}{(\frac{5}{4}-\frac x{2})_{\frac{p-1}{2}}}\bigg)\bigg|_{x=0}
=\frac{(\frac{3}2)_{\frac{p-1}{2}}}{(\frac{5}{4})_{\frac{p-1}{2}}}\sum_{j=1}^{\frac{p-1}{2}}\bigg(\frac{1}{2j+\frac12}-\frac1{j+\frac12}\bigg)\\
=&{}_2F_1\bigg[\begin{matrix} \frac12(1+p)&\frac12(1-p)\\ &1+p\end{matrix}\bigg|\,-1\bigg]\cdot\sum_{j=1}^{\frac{p-1}{2}}\bigg(\frac{1}{2j+\frac12}-\frac1{j+\frac12}\bigg).
\end{align*}
Note that since $p\equiv 1\pmod{4}$,
\begin{align*}
&\sum_{j=1}^{\frac{p-1}{2}}\bigg(\frac{1}{2j+\frac12}-\frac1{j+\frac12}\bigg)=
\sum_{\substack{1\leq j\leq\frac{p-1}{2}\\ j\neq\frac{1}{4}(p-1)}}\frac2{4j+1}-
\sum_{j=1}^{\frac{p-3}{2}}\frac2{2j+1}\\
\equiv&\bigg(\sum_{j=\frac{p+3}4}^{\frac{p-1}2}\frac2{4j+1-p}-\sum_{j=1}^{\frac{p-5}4}\frac2{p-(4j+1)}\bigg)+\sum_{j=1}^{\frac{p-3}{2}}\frac2{p-(2j+1)}
=\sum_{j=1}^{\frac{p-1}{2}}\frac1j\pmod{p}.
\end{align*}
Thus
\begin{align*}
\Psi'(0)=2\Psi(0)\cdot H_{\frac{p-1}{2}}\pmod{p}.
\end{align*}

On the other hand, by Lemma \ref{Gammapadicderivativealphabeta},
\begin{align*}
p^2\cdot\frac{(1)_{p-1}^2}{(\frac32)_{p-1}^2}=&\frac{\Gamma_p(1+p)^2\Gamma_p(\frac12)^2}{\Gamma_p(1)^2\Gamma_p(\frac12+p)^2}\\\equiv&1+p\cdot
\frac{\Gamma_p(\frac12)^2}{\Gamma_p(1)^2}\cdot\frac{d}{dx}\bigg(\frac{\Gamma(1+x)^2}{\Gamma(\frac{p+1}{2}+x)^2}\bigg)\bigg|_{x=0}
=1-2pH_{p-1}\pmod{p^2},
\end{align*}
since $\Gamma_p(1/2)^2=(-1)^{\langle-\frac12\rangle_p}=1$.
Hence
\begin{align*}
p^2\sum_{k=\frac{p-1}2}^{p-1}\frac{(1)_k^2}{(\frac32)_k^2}\cdot(-1)^k
\equiv&
\frac{p^2(1)_{p-1}^2}{(\frac32)_{p-1}^2}\cdot\big(\Psi(0)+p\cdot\Psi'(0)\big)\\
\equiv&(1-2pH_{\frac{p-1}2})\cdot\Psi(0)(1+2pH_{\frac{p-1}2})\equiv\Psi(0)\pmod{p^2}.
\end{align*}

\begin{acknowledgment}
We are grateful to Professors George Andrews, Richard Askey, Dennis Stanton and Wadim Zudilin for their helpful comments on this paper.
\end{acknowledgment}

\end{document}